\documentclass[	a4paper,
12pt,
BCOR=10mm,
DIV=12,
headinclude,
headheight=16mm,
oneside,
onecolumn,
openany,
parskip=half,
table,
]{article}

\usepackage[utf8]{inputenc}
\usepackage{fontenc}
\everymath{\displaystyle}
\usepackage{amssymb}
\usepackage{amsmath}
\usepackage{amsthm}
\usepackage{algorithm}
\usepackage{algpseudocode}
\usepackage{listings}
\usepackage{verbatim}
\usepackage[nottoc,numbib]{tocbibind}
\usepackage[toc,page]{appendix}
\usepackage[hidelinks]{hyperref}
\usepackage{geometry} 
\usepackage{setspace}
\usepackage{booktabs,colortbl,xcolor}
\usepackage{graphicx,wrapfig}
\usepackage[section]{placeins} 
\usepackage{float}
\usepackage{enumitem}
\usepackage{subfiles}
\usepackage{scrhack}

\usepackage{romannum}
\usepackage{enumitem}
\usepackage{graphicx}
\usepackage[absolute]{textpos}
\usepackage{color}

\usepackage[backend=bibtex]{biblatex}
\renewbibmacro{in:}{}
\usepackage{lastpage,scrlayer-scrpage}
\clearpairofpagestyles

\newtheorem{theorem}{Theorem}[section]
\newtheorem{corollary}{Corollary}[theorem]
\newtheorem{lemma}[theorem]{Lemma}
\theoremstyle{definition}
\newtheorem{remark}{Remark}
\theoremstyle{definition}
\newtheorem{definition}[theorem]{Definition}
\newtheorem{problem}{Problem}

\allowdisplaybreaks

\title{Bounding Radial Variation of positive harmonic Functions on Lipschitz Domains}
\author{Jakob Fromherz, Paul F.X. Müller and Katharina Riegler}

\begin{document}
\pagenumbering{arabic}
\maketitle
\begin{abstract}
	Let $D\subset\mathbb{R}^d$ be a Lipschitz domain (bounded or unbounded) and let $u$ be a bounded positive harmonic function on $D$. By a point of bounded radial variation, or a \textit{Bourgain} point, we will refer to a point $x$ on $\partial D$, such that the integral
	\begin{align*}
		\int_0^1\|\nabla u(x+tv)\|\mathrm{d}t
	\end{align*}
	is finite for some vector $v$ satisfying $(x,x+v])\subset D$. Let $\mathcal{V}^u(\partial D)$ denote the set of all Bourgain points on $\partial D$.
	
	This paper presents techniques going back to Havin and Mozolyako \cite{HavinMozol2016} that determine a large set of Bourgain points on $\partial D$. To be more precise, any ball $B$ with center on $\partial D$ contains a Bourgain point. As a consequence, the Hausdorff dimension of $\mathcal{V}^u(\partial D)\cap B(x,r)$ is greater than $(d-1)/2$, for arbitrary $x\in\partial D$ and $r>0$. 
	
	The intention of this paper is to provide and document an updated and refined version of \cite{MuellerRiegler2020}.
\end{abstract}		
	\section{Introduction}
	In 1993, J. Bourgain published two papers \cite{Bourgain931}, \cite{Bourgain932} in which he proved, that for any nonnegative harmonic on the unit disk $\mathbb{D}:=\lbrace z\in\mathbb{C}:|z|\leq1\rbrace$, the quantity
	\begin{align*}
		\mathrm{var}(u,\theta):=\int_0^1|\nabla u(\rho e^{\mathrm{i}\theta})|\mathrm{d}\rho
	\end{align*}
	is finite on a subset of $\partial\mathbb{D}$ with Hausdorff dimension 1. A compact overview of the two papers by Bourgain and their implications on minimal surfaces in $\mathbb{R}^2$ is featured in the Master's thesis of Katharina Riegler \cite{RieglerMaster}.
	
	Later, in 2016, Havin and Mozolyako considerably expanded the techniques developed by Bourgain in \cite{HavinMozol2016} and proved analogous results for positive harmonic functions $u$ on domains $D\subset\mathbb{R}^d$, $d\geq3$ with $C^2$ boundary. To be more precise, it was shown that on a subset of the boundary $\partial D$ that has Hausdorff dimension $d-1$, the variation of $u$ along the inward normal vector is finite i.e. the set
	\begin{align*}
		\left\lbrace x\in\partial D: \int_0^1\|\nabla u(x+\rho r(x)\vec{N}(x))\|\mathrm{d}\rho<\infty\right\rbrace
	\end{align*}
	has full Hausdorff dimension.

	In the introduction of the aforementioned paper, the authors comment on the possibility of further weakening regularity assumptions on the boundary of the domain and still producing analogous results for Lipschitz domains using their techniques. This remark was further expanded on by Müller and Riegler in 2020 \cite{MuellerRiegler2020}, providing the full proof and necessary modifications of how the paper by Havin and Mozolyako extends to Lipschitz domains. Using the results in \cite{MuellerRiegler2020}, they were able to lift a theorem on boundedness of radial variation of Bloch functions on the unit disk to arbitrary dimensions in \cite{MuellerRiegler20Bloch}. The combination of both papers forms the PhD-thesis of Katharina Riegler \cite{RieglerDiss} and serves as comprehensive source for what follows.
	
	The following treatise is concerned with presenting and sharpening the results in \cite{MuellerRiegler2020}:	For a bounded positive harmonic function $u$ on a Lipschitz domain $D$ (bounded or unbounded), we are interested in characterizing the set
	\begin{align}
		\label{eq:BourgainPoints}
		\mathcal{V}^u(\partial D):=\left\lbrace x\in \partial D:\exists v\in\mathbb{R}^d\text{ s.t. }(x,x+v]\subset D\text{ and }\int_0^1\|\nabla u(x+tv)\|\mathrm{d}t<\infty\right\rbrace,
	\end{align}
	in particular, estimating its Hausdorff dimension. In accordance to previous papers (\cite{HavinMozol2016},\cite{MuellerRiegler2020}) dealing with boundedness of radial variation, we shall call it the set of \textit{Bourgain} points.
	
	We will first solve the problem on very special Lipschitz domains, namely on \textit{near half spaces}. Then through locally identifying a Lipschitz domain with a corresponding near half space, we will carry over the results to arbitrary Lipschitz domains
	
	Section \ref{sec:prel} will introduce some preliminaries, while Section \ref{sec:rad_var_nhs} is dedicated to proving important variational estimates on near half spaces. In particular, it is shown, that in any surface ball on the boundary of a near half space, there exists a Bourgain point (Theorem \ref{thm:main}). Section \ref{sec:Hausdorff} will use results from previous sections to estimate the Hausdorff dimension of the set of all Bourgain points. This was not done in \cite{MuellerRiegler2020} and the resulting bound on the Hausdorff dimension is weaker than that proven for domains with $C^2$ boundaries in \cite{HavinMozol2016}. However, an improvement is very likely. As previously announced, Section \ref{sec:nhs_lip} contains the necessary steps to extend Hausdorff dimension estimates of the set of all Bourgain points to general Lipschitz domains.
	\section{Preliminaries}\label{sec:prel}
	\subsection{Notation}
	Let $n\in\mathbb{N}$, then for the set of subsets of $\lbrace1,\dots,n\rbrace$ of cardinality $k\in\lbrace0,\dots,n\rbrace$ we write $\left\lbrace\binom{n}{k}\right\rbrace$, i.e.
	\[
		\left\lbrace\binom{n}{k}\right\rbrace:=\left\lbrace M\subseteq \lbrace1,\dots,n\rbrace : |A|=k\right\rbrace.
	\]
	Furthermore, we denote the dyadic numbers on $[0,1]$ by $\mathcal{D}:=\left\lbrace\frac{k}{2^n}:n\in\mathbb{N}_0\quad k\in\lbrace0,\dots,2^n\rbrace\right\rbrace$. For a compact interval $[a,b]$ we denote with $\mathcal{D}([a,b])$ the dyadic numbers in $[0,1]$, transformed to $[a,b]$ i.e.
	\begin{align*}
		\mathcal{D}\left([a,b]\right):=\left\lbrace a+\frac{(b-a)k}{2^n}:n\in\mathbb{N}_0\quad k\in\lbrace0,\dots,2^n\rbrace\right\rbrace.
	\end{align*}
	
	By $B(x,r)$ we will denote the open ball around $x$ with radius $r>0$. Sometimes, if we want to specify the dimension $d\in\mathbb{N}$ in which we consider the ball, we add the $d$ as superscript, i.e. $B^d(x,r)=\lbrace y\in\mathbb{R}^d:\|x-y\|<r\rbrace$.
	
	Let $X$ be any locally compact  Hausdorff space, then we denote by
	\begin{align*}
		C_0(X):=\left\lbrace f:X\longrightarrow\mathbb{R}:f \text{ is continuous and } \forall\epsilon>0\exists K\subset X\text{ compact}:\left|f_{|X\setminus K}\right|<\epsilon \right\rbrace
	\end{align*}
	the normed space of continuous functions on $X$, that vanish at infinity, equipped with the norm $\left\|f\right\|_{C_0(X)}:=\sup_{z\in X}|f(z)|$. Riesz's representation theorem \cite{Rudin1999} states, that the dual of $C_0(X)$, is the space of regular signed Borel measures with finite total variation, i.e: $C_0(X)\cong\mathcal{M}(X)$ where
	\begin{align*}
		\mathcal{M}(X):=\left\lbrace\mu :\mu\text{ regular, signed measure with } |\mu(X)|<\infty\right\rbrace.
	\end{align*}
	We continue to define \textit{near half spaces}. Those are special Lipschitz domains which allow us to derive variational estimates of positive harmonic functions thereon.
	\begin{definition}[Near Half Space]
		\label{def:near-half_space}
		A Lipschitz domain $\mathcal{O}$ is called a near half space, if there exists a Lipschitz function $\Phi:\mathbb{R}^{d-1}\longrightarrow\mathbb{R}$ with the extra condition 
		\[
			\exists r\in(0,1):\quad \Phi\arrowvert_{\mathbb{R}^{d-1}\setminus B^{d-1}(0,r)}=0,
		\]
		such that $\mathcal{O}$ is given by the epigraph of $\Phi$, i.e. 
		\[
		\mathcal{O}=\left\lbrace(x,y)\in\mathbb{R}^d:x\in\mathbb{R}^{d-1},y\in\mathbb{R}\quad\text{and}\quad y>\Phi(x)\right\rbrace.
		\]
		The boundary of $\mathcal{O}$ is then given by the graph of $\Phi$, i.e.
		\[
			S:=\partial \mathcal{O}=\left\lbrace(x,y)\in\mathbb{R}^d:x\in\mathbb{R}^{d-1},y\in\mathbb{R}\quad\text{and}\quad y=\Phi(x)\right\rbrace.
		\] 
	\end{definition}
	\begin{remark}\textbf{(Convention)}
		Throughout this paper, the constant $c_S$ denotes a positive constant that depends only on the Lipschitz boundary of our domain $\mathcal{O}$. It will be used very flexibly, in one line of inequalities it might take on different values but remains dependent only on $S$ for example $c_S\cdot\exp(c_S)=c_S$ and $c_S\leq3\cdot c_S=c_S$.
		For $x\in\mathbb{R}^d$ we will use the notational convention $x_y:=x+y \vec{e_d}$, where $\vec{e_d}$ denotes the $d$-th unit vector of the standard basis and $y\in\mathbb{R}$. With our definition of the half-space $\mathcal{O}$ we immediately obtain $x_y\in \mathcal{O}$ for all $x\in S$  if $y>0$. Analogously, sets are shifted, i.e. for $A\subset\mathbb{R}^d$ we define $A_y:=y\vec{e_d}+A$. Now let $A,B\subseteq\mathbb{R}^d$ and $y\in \mathbb{R}$ such that $A_y\subseteq B$. For a function $\varphi$ defined on $B$ we define the function $\varphi_y$ on $A$ by $\varphi_y(x):=\varphi(x_y)$.
	\end{remark}
	Citations of referenced Theorems and Lemmas are given directly after the number of the Theorem/Lemma in question.
	
	\subsection{Harnack Inequality and Harnack chains}
	An important ingredient necessary to analyze positive harmonic functions is Harnack's inequality:
	\begin{theorem}[Harnack inequality \cite{Gilbarg2001}] 
		\label{thm:Harnack_ineq}
		Let $x_0\in\mathbb{R}^d$, $R>0$ and $u:B(x_0,R)\longrightarrow\mathbb{R}$ be non-negative harmonic function. Then for every $x\in\mathbb{R}^d$ with $\|x-x_0\|=r<R$ we have the double inequality
		\begin{align*}
			\frac{\left(1-\frac{r}{R}\right)}{\left(1+\frac{r}{R}\right)^{d-1}}u(x_0)\leq u(x)\leq\frac{\left(1+\frac{r}{R}\right)}{\left(1-\frac{r}{R}\right)^{d-1}}u(x_0).
		\end{align*}
	\begin{corollary}
		\label{cor:Harnack_gradient}
		Let $\Omega\subset\mathbb{R}^d$ be a domain and $u:\Omega\longrightarrow\mathbb{R}$ be positive harmonic. Then for $x\in\Omega$ we can bound the gradient of $u$ in the following way:
		\begin{align*}
			\|\nabla u(x)\|\leq d^{3/2}\frac{u(x)}{d(x,\partial\Omega)}.
		\end{align*}
	\end{corollary}
	\begin{proof}
		Let $x\in\Omega$ and denote $R=\textrm{dist}(x,\partial\Omega)$. Then for any $h\in(0,R)$ we have by Harnack's inequality
		\begin{align*}
			\frac{u(x+h\vec{e_i})-u(x)}{h}&\leq\left(\frac{\frac{1}{h}+\frac{1}{R}}{\left(1-\frac{h}{R}\right)^{d-1}}-\frac{1}{h}\right) u(x)\\
			&=\left(\frac{\frac{1}{h}\left(1-\left(1-\frac{h}{R}\right)^{d-1}\right)+\frac{1}{R}}{\left(1-\frac{h}{R}\right)^{d-1}}\right)u(x).
		\end{align*}
		Taking the limit $h\rightarrow0$ we obtain
		\begin{align*}
			\frac{\partial}{\partial x_i}u(x)\leq\frac{d}{R}u(x),
		\end{align*}
		since the constant terms in $p=1-\left(1-\frac{h}{R}\right)^{d-1}$ cancel and hence $p/h\rightarrow\frac{d-1}{R}$ as $h\rightarrow0$. A similar calculation also shows $\frac{\partial}{\partial x_i}u(x)\geq-\frac{d}{R}u(x)$ and thus $\left|\frac{\partial}{\partial x_i}u(x)\right|\leq\frac{d}{R}u(x)$. This implies the claimed inequality.
	\end{proof}
	\end{theorem}
	This next Lemma will enable us to properly exploit Harnack's inequality for positive harmonic functions on near half spaces by providing a simple formula estimating the distance of points in $\mathcal{O}$ to the boundary $S$.
	\begin{lemma}[\cite{MuellerRiegler2020} ]
		\label{lem:dist_to_boundary}
		Let $\mathcal{O}$ be a near-half space with boundary $S$. Then there exists $c_S\leq1$ such that for all $x\in S$ and every $y>0$ we have $y\geq d(x_y,S)\geq c_Sy$.
	\end{lemma}
	\begin{proof}
		The first inequality is trivial. For $z=(z_1,\dots,z_d)\in\mathbb{R}^d$ we denote $\bar{z}=(z_1,\dots,z_{d-1})\in\mathbb{R}^{d-1}$ so in particular we have $z=(\bar{z},\Phi(\bar{z}))$ for any $z\in S$. Choosing an arbitrary $z\in S$, we obtain
		\begin{align*}
			d(x_y,z)^2
			&=d((\bar{x},\Phi(\bar{x})+y),(\bar{z},\Phi(\bar{z})))^2\\
			&=d(\bar{x},\bar{z})^2+(y+\Phi(\bar{x})-\Phi(\bar{z}))^2\\
			&\geq \frac{1}{L^2}(\Phi(\bar{x})-\Phi(\bar{z}))^2+y^2+2y(\Phi(\bar{x})-\Phi(\bar{z}))+(\Phi(\bar{x})-\Phi(\bar{z}))^2\\
			&=\left(\frac{\sqrt{L^2+1}}{L}(\Phi(\bar{x})-\Phi(\bar{z}))+\frac{L}{\sqrt{L^2+1}}y\right)^2+y^2-\frac{L^2}{L^2+1}y^2\\
			&\geq \frac{1}{L^2+1}y^2.
		\end{align*}
		Therefore we have 
		\[
			\inf_{z\in S}d(x_y,z)\geq\frac{1}{\sqrt{L^2+1}}y.
		\]
	\end{proof}
	Equipped with Harnack's inequality and the distance formula from Lemma \ref{lem:dist_to_boundary}, we formulate a lemma on Harnack chains, a sequence of overlapping balls inside $\mathcal{O}$, that connects 2 points in $\mathcal{O}$.	
	\begin{lemma}
		\label{lem:harnack_chain}
		Let $u$ be a positive harmonic function on a near half space $\mathcal{O}$ with $\partial \mathcal{O} = S$ and $0<y_1\leq y_2$. Then there exists $\alpha_1,\alpha_2>0$, that solely depend on the Lipschitz boundary of $\mathcal{O}$ and $C_1,C_2>0$, depending on the the dimension $d$, such that for all $x\in S$ we have
		\begin{align*}
			C_1\left(\frac{y_1}{y_2}\right)^{\alpha_1}\leq\frac{u_{y_2}(x)}{u_{y_1}(x)}\leq C_2\left(\frac{y_2}{y_1}
			\right)^{\alpha_2}.
		\end{align*}
	\end{lemma}
	\begin{proof}
		Let $c_S=c$ be the constant from Lemma \eqref{lem:dist_to_boundary}, in particular we have that for all $y>0$, $u$ is positive harmonic on $B(x_y,\delta cy)$ since $B(x_y,\delta cy)\subset \mathcal{O}$ for any $0<\delta<1$. 	
		Furthermore, depending on $y_1,y_2$, there exists $N=N(y_1,y_2)$ such that 
		\begin{align}
			\label{eq:double_ineq_y1y2}
			\left(1-\frac{c}{2}\right)^Ny_2\leq y_1 \leq \left(1-\frac{c}{2}\right)^{N-1}y_2.
		\end{align}
		Denoting $a_n:=\left(1-\frac{c}{2}\right)^ny_2$, we have $x_{y_1}\in B(x_{a_{N-1}},a_{N-1}-a_N)$. An easy calculation shows $a_{n-1}-a_n=\frac{c}{2}a_{n-1}$. Applying the Harnack inequality to 
		\[
		B(x_{a_{n-1}},a_{n-1}-a_n)=B\left(x_{a_{n-1}},\frac{c}{2}a_{n-1}\right)\subsetneq B\left(x_{a_{n-1}},\frac{3}{4}ca_{n-1}\right)  
		\]
		for every $n$ yields
		\begin{align*}
			\frac{u(x_{a_n})}{u(x_{a_{n-1}})}\geq\frac{1-\frac{\frac{c}{2}a_{n-1}}{\frac{3c}{4}a_{n-1}}}{\left(1+\frac{\frac{c}{2}a_{n-1}}{\frac{3c}{4}a_{n-1}}\right)^{d-1}}=\frac{1}{3}\left(\frac{5}{3}\right)^{1-d}=:C_d^{(1)}
		\end{align*}
		and analogously 
		\begin{align*}
			\frac{u(x_{a_n})}{u(x_{a_{n-1}})}\leq\frac{1+\frac{\frac{c}{2}a_{n-1}}{\frac{3c}{4}a_{n-1}}}{\left(1-\frac{\frac{c}{2}a_{n-1}}{\frac{3c}{4}a_{n-1}}\right)^{d-1}}=\frac{5}{3}\left(\frac{1}{3}\right)^{1-d}=:C_d^{(2)}.
		\end{align*}
		Furthermore, since $x_{y_1}\in B\left(x_{a_{N-1}},\frac{c}{2}a_{N-1}\right)$ and thus $y_1\geq a_{N}$, we can once more use the Harnack inequality to estimate
		\begin{align*}
			 C_d^{(1)}\leq\frac{1-\frac{a_{N-1}-y_1}{\frac{3c}{4}a_{n-1}}}{\left(1+\frac{a_{N-1}-y_1}{\frac{3c}{4}a_{n-1}}\right)^{d-1}} \leq\frac{u(x_{y_1})}{u(x_{a_{N-1}})}\leq\frac{1+\frac{a_{N-1}-y_1}{\frac{3c}{4}a_{n-1}}}{\left(1-\frac{a_{N-1}-y_1}{\frac{3c}{4}a_{n-1}}\right)^{d-1}}\leq C_d^{(2)}.
		\end{align*}
		Writing 
		\begin{align*}
			\frac{u(x_{y_2})}{u(x_{y_1})}&=\prod_{i=1}^{N-1}\frac{u(x_{a_{i-1}})}{u(x_{a_i})}\cdot\frac{u(x_{a_{N-1}})}{u(x_{y_1})},
		\end{align*}
		we can finally bound
		\begin{align*}
			\left(C_d^{(2)}\right)^{-N}\leq\frac{u(x_{y_2})}{u(x_{y_1})}\leq\left(C_d^{(1)}\right)^{-N}.
		\end{align*}
		Taking the logarithm in \eqref{eq:double_ineq_y1y2} we obtain 
		\begin{align*}
			\frac{\ln\left(\frac{y_1}{y_2}\right)}{\ln\left(1-\frac{c}{2}\right)}\leq N\leq\frac{\ln\left(\frac{y_1}{y_2}\right)}{\ln\left(1-\frac{c}{2}\right)}+1.
		\end{align*}
		Since $C_d^{(1)}<1$, a simple calculation shows 
		\begin{align*}
			\left(C_d^{(1)}\right)^{-N}\leq \left(C_d^{(1)}\right)^{-1}\left(\frac{y_2}{y_1}\right)^{\alpha_1} \quad \text{where}\quad \alpha_1=\frac{\ln\left(C_d^{(1)}\right)}{\ln\left(1-\frac{c}{2}\right)}>0.
		\end{align*}
		Moreover, since $C_d^{(2)}>1$, we can derive
		\begin{align*}
			\left(C_d^{(2)}\right)^{-N}\geq\left(C_d^{(2)}\right)^{-1}\left(\frac{y_1}{y_2}\right)^{\alpha_2}\quad\text{where}\quad\alpha_2=-\frac{\ln\left(C_d^{(2)}\right)}{\ln\left(1-\frac{c}{2}\right)}>0.
		\end{align*}
	\end{proof}
	\newpage
	\begin{figure}[h!]
		\centering
		\includegraphics[width=0.4\textwidth]{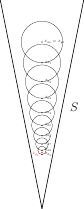}
		\caption{A Harnack chain connecting $x_{y_2}$ and $x_{y_1}$}
	\end{figure}
	
	\subsection{The Dirichlet problem and Harmonic Measure}
	One of the key tools in our analysis of harmonic functions on the domains above a Lipschitz graph will be the harmonic measure, a measure defined on the boundary $S$, such that integration of bounded and continuous functions on the boundary yields the unique harmonic extension, i.e. the extension is harmonic inside the domain and converges back to the boundary values. We take a step back and formulate the problem and aims for general domains $D\subset\mathbb{R}^d$.\par
	In the following we will consider differential operators $L:C^2(\mathbb{R}^d)\longrightarrow C(\mathbb{R}^d)$ of the following form
	\begin{align}
		\label{eq:diff_op}
		L=\sum_{i,j=1}^{d}a_{ij}\frac{\partial^2}{\partial x_i\partial x_j}+\sum_{i=1}^db_i\frac{\partial}{\partial x_i},
	\end{align}
	where $a_{ij}$ and $b_i$ are continuous functions. We call $L$ \textbf{semi-elliptic} if the eigenvalues of $(a_{ij})_{i,j=1}^d$ are non-negative, \textbf{elliptic} if they are positive and \textbf{uniformly elliptic} if the eigenvalues are positiv and bounded away from $0$.
	We now state the problem of interest.
	
	\begin{problem}(Dirichlet problem) \label{prob:DirichletProblem}
		Given a domain $D\subset\mathbb{R}^d$, a semi-elliptic differential operator $L$ as in \eqref{eq:diff_op} on $D$ and some bounded measurable function $\phi$ defined on $\partial D$, find a function $u\in C^2(D)$ such that the following two properties hold
		\begin{enumerate}[label=(\roman*)]
			\item $Lu=0$ in $D$,
			\item $\lim_{\substack{x\rightarrow y\\x\in D}}u(x)=\phi(y)$ for all $y\in\partial D$.
		\end{enumerate}
	\end{problem}
	
	A unique solution to Problem \ref{prob:DirichletProblem}, will be called \textbf{harmonic extension}. However, we are not only interested in existence and uniqueness of a solution to problem \ref{prob:DirichletProblem}, we also want the unique solution to be representable as integration of the boundary values $\phi:\partial D\rightarrow\mathbb{R}$ against a  measure, i.e. there exists a a family of measures $\lbrace\mu^x\rbrace_{x\in D}$ on the Borel subsets of $\partial D$, such that the solution $u\in C^2(D)$ to Problem \ref{prob:DirichletProblem} is given by
	\begin{align}
		\label{eq:harm_ext_form}
		u(x)=\int_{\partial D}\phi(\xi)\mathrm{d}\mu^x(\xi).
	\end{align}

	For $L=\Delta$ on bounded domains $D$ and continuous boundary values (i.e. $\phi\in C(\partial D)$ in definition \ref{prob:DirichletProblem}), recall that the Perron operator $H_D:C(\partial D)\rightarrow C^2(D)\cap C(\bar{D})$ realizes the desired harmonic extension \cite{Helms1969}. The family of measure that we look for is constructed using Riesz's representation theorem: For $x\in D$ define the functional $L_x:C(\partial D)\rightarrow\mathbb{R}$ where $\phi\longmapsto (H_D\phi)(x)$. The maximum principle assures boundedness of $L_x$ moreover, $L_x$ is positive and $L_x(1)=1$, thus by Riesz's representation theorem, $L_x$ is uniquely representable as a probability measure $\mu^x$ in the sense of \eqref{eq:harm_ext_form}. The probability measure $\mu^x$ is then called the \textbf{harmonic measure} of $D$ w.r.t. the pole $x\in D$.
	
	Rather than extending this classical approach of defining the harmonic measure for unbounded domains, which is done in \cite{Helms1969}, we will use the following probabilistic line of arguments found in \cite{Oksendal2003} and \cite{Bass}. We begin by defining Brownian motion.\\
	Let $(\Omega,\mathcal{F},\mathbb{P})$ be a probability space and $\mathcal{B}$ the $\sigma$ algebra generated by the open subsets of $[0,\infty)$. A stochastic process $X(t,\omega):=X_t(\omega):[0,\infty)\times\Omega\rightarrow\mathbb{R}^d$, is a measurable mapping w.r.t to $\mathcal{B}\otimes\mathcal{F}$, where $\mathcal{B}=\mathcal{B}(\mathbb{R})$ denotes the $\sigma$-algebra generated by the open subsets of the real line.
	\begin{definition}(Brownian Motion)
		A stochastic process $B_t$ is a one-dimensional Brownian Motion if it satisfies the following conditions
		\begin{enumerate}[label=(\roman*)]
			\item $B_0=0$ a.s.,
			\item $\forall s\leq t: B_t-B_s\sim\mathcal{N}(0,s-t)$, i.e. $B_s-B_t$ is normally distributed with mean $0$ and variance $s-t$,
			\item $\forall s<t: B_s-B_t$ is independent of $\sigma(\lbrace B_r:r\in[0,t]\rbrace)$, 
			\item The paths $t\mapsto B_t$ are continuous a.e.
		\end{enumerate}
		If $B_t^{(1)},\dots B_t^{(d)}$ are independent one-dimensional Brownian Motions, then
		\begin{align*}
			B_t:=\left(B_t^{(1)},\dots,B_t^{(d)}\right)
		\end{align*}
		defines a $d$-dimensional Brownian Motion.
	\end{definition}
	There are numerous proofs of existence of such a stochastic process, for some constructions refer to \cite{Bass} and \cite{Oksendal2003}. However, the most intuitive and direct approach works on a probability space $(\Omega,\mathcal{F},\mathbb{P})$ on which there exists a countable family of independent Gaussian random variables with mean $0$ and variance $1$ and uses the Haar functions \cite[11]{Bass}.
	\begin{definition}(Stopping time)
		Let $\lbrace\mathcal{F}_t\rbrace_{t\geq0}$ be a filtration of $\mathcal{F}$. A mapping $\tau:\Omega\longrightarrow[0,\infty)$ that satisfies $\lbrace\tau<t\rbrace\in\mathcal{F}_t$ for all $t\geq0$, is called a stopping time.
	\end{definition}
	Now let $D\subset\mathbb{R}^d$ be a domain and define for any $x\in\mathbb{R}^d$ 
	\begin{align}
		\tau_D^x:=\inf\left\lbrace t>0:x+B_t\notin D\right\rbrace,
	\end{align}
	then $\tau_D^x$ is a stopping time w.r.t to the filtration $\mathcal{F}_t=\sigma(B_r:r\leq t)$, such that for any $\omega\in\Omega$, $\tau_D^x(\omega)$ gives the time when the path $B_\cdot(\omega)$ of Brownian motion started at $x$ leaves the domain $D$ for the first time. In particular, since $D$ is open, $x+B_{\tau_D^x}\in\partial D$ for all $x\in D$. This gives rise to a notion of how likely it is for Brownian motion started at a point inside $D$ to land in a subset of $\partial D$.
	\begin{definition}[Harmonic measure] \label{def:harm_meas}For $x\in D\subset\mathbb{R}^d$ we define the \textbf{harmonic measure} of $D$ w.r.t the pole $x$ as
	\begin{align}
		\omega_D^x(A):=\mathbb{P}\left(x+B_{\tau_D^x}\in A\right)\quad \text{for all}\quad A\in\mathcal{B}(\partial D).
	\end{align}
	\end{definition}
	The harmonic measure can thus also be interpreted as the hitting distribution of Brownian motion started inside the domain. If we now consider any ball $B(y,r)$, that is compactly contained in $D$ and define the function 
	\begin{align}
		\label{eq:harm_fctn}
		\varphi(x)=\mathbb{E}(f(x+B_{\tau_D^x}))\quad x\in D
	\end{align}
	for any bounded, measurable function $f$ on $\partial D$, then $\varphi$ satisfies the following mean-value property:
	\begin{align}
		\label{eq:mean_val_prop}
		\varphi(x)=\int_{\partial B(y,r)} \varphi(\xi)\mathrm{d}\omega_{B(y,r)}^x(\xi) \quad\text{for all }x\in B(y,r).
	\end{align}
	This formula is a consequence of the strong Markov property of Brownian motion, for exact proofs, refer to \cite[116\psq]{Oksendal2003}. Since the harmonic measure $\omega_{B(y,r)}^x$ on the ball can be computed to be the normalized surface measure on $\partial B(y,r)$ \cite[14]{Bass}, choosing $x=y$ in \eqref{eq:mean_val_prop}, yields that $\varphi$ is harmonic. Functions of the form given in \eqref{eq:harm_fctn}, will thus be candidates for solutions of the Dirichlet problem \ref{prob:DirichletProblem}. Setting $f=\chi_A$ for $A\in\mathcal{B}(\partial D)$ in \eqref{eq:harm_fctn}, we can even infer harmonicity of the function $x\longmapsto\omega_D^x(A)$ for $x\in D$, since
	\begin{align*}
		\omega_D^x(A)=\mathbb{E}\left(\chi_A\left(x+B_{\tau_D^x}\right)\right).
	\end{align*}
	Observe that if $D$ is the epigraph of a Lipschitz function, then by Harnack's inequality and a Harnack chain argument, for any $x_0\in D$, there exist positive constants $c_1,c_2$, depending on $S,x,x_0$ but not on $A$, such that
	\begin{align*}
		c_1\omega_D^{x_0}(A)\leq\omega_D^x(A)\leq c_2\omega_D^{x_0}(A)
	\end{align*}
	Hence if $A\subset \partial D$ is measurable and not a null set of $\omega_D^x$, it is not a null set for $\omega_D^{x_0}$. In other words, for such sets $A$, the function $x\longmapsto\omega^x_D(A)$ is positive harmonic in $D$.
	
	We now specify the domain $D$ to be the epigraph of a Lipschitz function (which includes near-half spaces in Definition \ref{def:near-half_space}) and $L=\Delta$ and formulate existence and uniqueness of solutions to the Dirichlet Problem \ref{prob:DirichletProblem} in the special case \cite[Theorem 9.1.2, 9.2.14]{Oksendal2003}:
	\begin{theorem}[\cite{Oksendal2003}]
		\label{thm:ex_harm_ext}
		Suppose $D$ is an epigraph of a Lipschitz function. Moreover, let $\phi$ be a bounded and continuous function on $\partial D$. For $x\in D$, define the function
		\begin{align*}
			u(x)=\mathbb{E}\left(\phi\left(x+B_{\tau_D^x}\right)\right),
		\end{align*}
		then  for all $\alpha<1$, $u\in C^{2+\alpha}(D)$ and $u$ is the unique solution to the Dirichlet Problem \ref{prob:DirichletProblem}. i.e.
		\begin{enumerate}[label=(\roman*)]
			\item $\Delta u=0$ in $D$,
			\item $\lim_{\substack{y\rightarrow x\\x\in D}}u(y)=\phi(x)$ for all $x\in\partial D$,
		\end{enumerate}
		where $C^{k+\alpha}(D)=\left\lbrace f:D\rightarrow\mathbb{R}:\forall i=(i_1,\dots, i_d)\in\mathbb{N}^d:\partial^i f \text{ is Hölder-}\alpha\text{ continuous}\right\rbrace$
	\end{theorem}
	Observe that
	\begin{align*}
		\mathbb{E}\left(\phi\left(x+B_{\tau_D^x}\right)\right)=\int_\Omega\phi\left(x+B_{\tau_D^x}\right)\mathrm{d}\mathbb{P}=\int_{\partial D}\phi(\xi)\mathrm{d}\omega_D^x(\xi),
	\end{align*}
	implying that the harmonic measure in Definition \ref{def:harm_meas} yields the integral representation of the harmonic extension in the sense of \eqref{eq:harm_ext_form}.
	\subsection{Harmonic Measure and Hausdorff Measure: Dahlberg's Theorem}
	In 1977, B.E.J. Dahlberg \cite{Dahlberg1977}  established an important connection between the harmonic measure and Hausdorff measure. To be more precise, he proved that the harmonic measure $\omega^{x_0}$ on the boundary $\partial D$ of a Lipschitz domain $D\subset\mathbb{R}^d$, w.r.t to a pole $x_0\in D$, is mutually absolutely continuous to the $d-1$-dimensional Hausdorff measure, $\mathcal{H}^{d-1}$, on $\partial D$. Moreover, the Radon-Nikodym derivative of $\omega^{x_0}$ w.r.t to $\mathbb{H}^{d-1}$ even is locally in $L^{2+\varepsilon}(\partial D)$, for some $\varepsilon>0$. Since the  surface measure on $\partial D$ is just a multiple of $\mathcal{H}^{d-1}$ restricted to the boundary of a Lipschitz domain, we have the same statements for the surface measure $s$ on $\partial D$, as stated in \cite{Bass}. 
	\begin{theorem}[Dahlberg \cite{Dahlberg1977}]
		\label{thm:Dahlberg}
		Let $D\subset\mathbb{R}^d$ be either a bounded Lipschitz domain or the epigraph of a Lipschitz function and $x_0\in D$. Then the following statements hold:
		\begin{enumerate}[label=(\roman*)]
			\item $\omega^{z_0}$ and $s$ are mutually absolutely continuous.
			\item If $d$ is the Radon-Nikodym derivative of $\omega^{z_0}$ with respect to $s$, then there exists $\varepsilon>0$, such that $d\in L^{2+\varepsilon}_{\text{loc}}(S)$.
		\end{enumerate}
	\end{theorem}
	This powerful theorem enables us to transfer knowledge about the geometry of Lipschitz domains to properties of their harmonic measure.´
	\subsection{The Martin kernel for epigraphs of Lipschitz domains}
	We will now observe that epigraphs of Lipschitz functions have "Poisson" kernels. \cite[192\psq]{Bass}
	\begin{theorem}[Martin Kernel \cite{Bass}] 
		Let $D\subset\mathbb{R}^d$ be the epigraph of a Lipschitz function. Then for any $x,x_0\in D$, the harmonic measures $\omega^x$ is absolutely continuous to $\omega^{x_0}$. Furthermore there exists a positive and continuous function $k^{x_0}:D\times\partial D\longrightarrow\mathbb{R}^+$ satisfying
		\begin{align*}
			\omega^x(A)=\int_{A}k^{x_0}(x,\xi)\mathrm{d}\omega^{x_0}(\xi)\quad\text{for any}\quad A\in\mathcal{B}(\partial D)
		\end{align*}
		and 
		\begin{align*}
			x\longmapsto k^{x_0}(x,\xi) \quad\text{is harmonic for all}\quad\xi\in\partial D.
		\end{align*}
		We call the function $k^{x_0}$ the Martin kernel of $D$ at $x_0$.
		In particular, the Radon-Nikodym derivative of $\omega^x$ with respect to $\omega^{x_0}$ is given by $k^{x_0}(x,\cdot)$.
	\end{theorem}
	The fact that $\omega_D^x$ is absolutely continuous w.r.t. $\omega_D^{x_0}$ is a consequence of Harnack's inequality, i.e. since for any measurable $A\subset D$, that is no $\omega^x_D$-nullset, the mapping $x\mapsto\omega_D^x(A)$ is a positive harmonic function, thus there exists a constant $c$, independent of $A$, such that $\omega^x_D(A)\leq c\omega_D^{x_0}(A)$. With the same reasoning and the roles reversed, we even obtain, that $\omega_D^x$ and $\omega_D	^{x_0}$ are equivalent. Their Radon-Nikodym derivative is thus positive $\omega_D^x$ a.e. and $\omega_D^{x_0}$ a.e.. Now given that $k^{x_0}$ is continuous, it is positive everywhere on $D$.
	
	Combining the Martin kernel with Theorem \ref{thm:ex_harm_ext}, an immediate consequence is the next lemma.	
	\begin{lemma}
		\label{lem:Martinkernel_Harmmeas}
		Let $D$ be the epigraph of a Lipschitz domain and $x_0,x\in D$, then for all continuous and bounded $f:\partial D\longrightarrow\mathbb{R}$ we have
		\begin{align*}
			\int_{\partial D} 	f(\xi)\mathrm{d}\omega^x(\xi)=\int_{\partial D}k^{x_0}(x,\xi)f(\xi)\mathrm{d}\omega^{x_0}(\xi),
		\end{align*}
		i.e. integration of $f$ against the Martin kernel, retrieves the harmonic extension of $f$ to $D$. In particular, if $u$ is harmonic in $D$ and continuous up to the boundary, then
		\[
		u(x)=\int_{\partial D} k^{x_0}(x,\xi)u(\xi)\mathrm{d}\omega^{x_0}(\xi).
		\]
	\end{lemma}
	What Lemma \ref{lem:Martinkernel_Harmmeas} states, is that we can solve the Dirichlet problem on epigraphs of Lipschitz functions with continuous boundary values, using only $\omega^{x_0}$ at a fixed $x_0$ and the corresponding Martin kernel $k^{x_0}$. This will be exploited heavily throughout.
	
	The Martin kernel on Lipschitz domains has been well studied by Hunt and Wheeden in \cite{HuntWheeden68}, \cite{HuntWheeden70} and has  later been examined on a more general class of domains, so called NTA domains, by Jerison and Kenig \cite{JerisonKenig82}.
	The first inequality of its kind, found by Hunt and Wheeden, characterizes the decay of the Martin kernel on bounded Lipschitz domains and is stated below.
	\begin{lemma}[Hunt Wheeden\cite{HuntWheeden68}]
		\label{lem:MartinKernelDecayHuntWheeden}
		Let $D$ be a bounded Lipschitz domain, $z_0\in D$ and $q_0\in\partial D$. Given $p\in D$ with $\|p-q_0\|=a$, define $\Delta_j:=B(q_0,2^ja)\cap S$ for $j\in\lbrace0,\dots,N\rbrace$ and $R_j:=\Delta_j\setminus\Delta_{j-1}$ for $j\in\lbrace1,\dots,N\rbrace$ with $R_0:=\Delta_0$, where $N=N_a$ is large enough, s.t. $D\subset\Delta_N$. Then, the following estimate holds:
		\begin{align}
			\label{eq:MartinKernelDecayBounded}
			\|k(p,q)\|_{L^\infty(R_j)}\leq c\frac{c_j}{\omega^{z_0}(\Delta_j)},
		\end{align}
		where additionally $\sum_{j=0}^{N}c_j\leq c'<\infty$ and $c,c'$ only depend on $\partial D$.
	\end{lemma}
	\begin{remark}
		In fact, the inequality above, even holds for for NTA domains and the sequence of constants $(c_j)_{j=0}^N$ has geometric decay, as shown in \cite[Lemma 4.14]{JerisonKenig82}. To be more precise, $c_j\leq c 2^{-\alpha j}$, where $c$ depends on $z_0$ and $D$ and $\alpha>0$ only depends on $D$.
	\end{remark}
	Reading and understanding the proof of Lemma 4.14 in \cite{JerisonKenig82}, it is not too hard to extend the proof, such that an inequality analogous to \eqref{eq:MartinKernelDecayBounded} holds for the Martin kernel on near half spaces.
	\begin{lemma}[Decay of Martin Kernel on Near Half Spaces]
		\label{lem:MartinKernelDecay}
	Let $D$ denote an epigraph of a Lipschitz function, $z_0\in D$ and $q_0\in\partial D$. For $p\in D$ satisfying $\|p-q_0\|=y>0$, define $\Delta_j:=B(q_0,2^jy)\cap S$ and $R_j:=\Delta_j\setminus\Delta_{j-1}$ where $j\in\mathbb{N}_0$ and $R_0:=\Delta_0$. Let $i\in\mathbb{N}$ be such that $2^iy<\|z_0-q_0\|\leq2^{i+1}y$. Then the following estimate holds.
	
	\begin{align}
		\label{eq:MartinKernelDecayUnbounded}
		\|k(p,q)\|_{L^\infty(R_j)}\leq
		\begin{cases}
			\frac{c2^{-\alpha j}}{\omega^{z_0}(\Delta_j)}, \quad &j\in\lbrace0,\dots,i-1\rbrace\\
			c2^{-\alpha j},\quad &j\geq i+1
		\end{cases}
	\end{align}
	where $c$ only depends on $z_0$ and $D$ and $\alpha>0$ only depends on $D$.
	\end{lemma}

	Being slightly inaccurate we can drop the case distinction in \eqref{eq:MartinKernelDecayUnbounded}. We will mostly use this inequality in the following form:
	\begin{align*}
		\forall j\in\mathbb{N}_0:\quad \|k(p,q)\|_{L^\infty(R_j)}\leq\frac{c2^{-\alpha j}}{\omega^{z_0}(\Delta_j)}.
	\end{align*}
	In particular, an immediate consequence is boundedness of the Martin kernel on the shifted boundaries $S_y$, $y>0$ of near half spaces.
	\begin{corollary}
		\label{thm:mart-kern-bdd}
		Let $\mathcal{O}$ be a near half space, $S=\partial \mathcal{O}$ and $k:\mathcal{O}\times S\longrightarrow\mathbb{R}$. Then $k_y:\mathcal{K}\times S\longrightarrow\mathbb{R}$ is bounded for $y>0$ and any compact subset $\mathcal{K}$ of $S$. 
	\end{corollary}
	\begin{proof}
		Lemma \ref{lem:MartinKernelDecay} immediately yields
		\begin{align*}
			k(x_y,\xi)\leq\frac{c}{\omega^{z_0}(B(x,y))} \quad\text{for any }x,\xi\in S.
		\end{align*}
		If $x\in\mathcal{K}\subset S$ for some compact subset $\mathcal{K}$, then $\omega^{z_0}(B(x,y))\geq\alpha>0$.
	\end{proof}
	
	\section{Radial Variation on Near Half Spaces}\label{sec:rad_var_nhs}
	We fix a near half space $\mathcal{O}$ and denote $S:=\partial\mathcal{O}$. Furthermore, let $z_0\in\mathcal{O}$ be arbitrary but fixed for this hole section. The point $z_0$ will serve as pole of the harmonic measure on $S$ as described below Lemma \ref{lem:Martinkernel_Harmmeas}. Unless necessary, we drop the notational dependence of the corresponding Martin kernel $k^{z_0}$, i.e. $k^{z_0}:=k$.
	
	The object of analysis will be a positive harmonic function $u:\mathcal{O}\longrightarrow\mathbb{R}$ satisfying the following two conditions:
	
	\begin{enumerate}
		\item For any $y>0$, the function $u_y:\mathcal{O}\longrightarrow\mathbb{R}$ is bounded on $\mathcal{O}$.
		\item $u$ vanishes at infinity on all positive shifts of the boundary i.e.
		\begin{align*}
			\forall y>0:\quad\lim_{\substack{\|x\|\rightarrow\infty\\x\in S}} u(x_y)=0.
		\end{align*}
	\end{enumerate}
	Fix $u$ as described above.
	\subsection{Kernels}
	The term \textbf{kernel} will refer to any function on $S\times S$ that defines a meaningful integral operator in the following way: Let $q:S\times S\longrightarrow\mathbb{R}$ be a kernel and set
	\[
	M_q:=\left\lbrace f:S\longrightarrow\mathbb{R}:\int_Sq(x,\xi)f(\xi)\mathrm{d}\omega^{z_0}(\xi)\quad\text{exists in }\mathbb{R}\right\rbrace.
	\]
	Then the integral operator $Q:M_q\longrightarrow\mathbb{R}$ is defined as
	\[
	(Qf)(x):=\int_Sq(x,\xi)f(\xi)\mathrm{d}\omega^{z_0}(\xi).
	\]
	Furthermore we will make use of the convention, that kernels are denoted as lowercase letters and that the integral operator induced by the kernel is denoted by the respective uppercase letter. The concatenation of two kernels $p,q$ is defined as
	\[
	(p\circ q)(x,\xi):=\int_Sp(x,\zeta)q(\zeta,\xi)\mathrm{d}\omega^{z_0}(\zeta),
	\]
	if well-defined.
	
	\begin{definition}(Martin Kernel and Harmonic Extension Operator)
		The operator 
		\[
			(K_yu)(x)=\int_S k(x_y,\xi)u(\xi)\mathrm{d}\omega^{z_0}(\xi)
		\]
		is well-defined since integration of a function $u$ given on $S$ with respect to the measure $k_y\cdot\mathrm{d}\omega^{z_0}$, yields the evaluation of the harmonic extension of $u$ to $\mathcal{O}$ at the point $x_y$
	\end{definition}
	\begin{lemma}[Properties of kernel $k_y$ \cite{MuellerRiegler2020}] \label{lem:MartinKernel}
		Let $k$ be the Martin kernel for the domain $\mathcal{O}$ with boundary $\partial \mathcal{O}=S$ and $y>0$. Then $k_y$ has the following properties
		\begin{enumerate}
			\item Let $u$ be a harmonic function on $\mathcal{O}$, then for all $x\in S$ and $y_1,y_2>0$ we have:
			\begin{align*}
				K_{y_2}(u_{y_1}|_S)(x) = u_{y_1+y_2}(x).
			\end{align*}
			\item (Semi-group property) For $y_1,y_2>0$ we have 
			\begin{align*}
				k_{y_1}\circ k_{y_2} = k_{y_1+y_2}.	
			\end{align*}
			\item (Quotient bound) For $0<y_1\leq y_2$, the quotient $\frac{k_{y_2}}{k_{y_1}}$ admits the following bound
			\begin{align*}
				\frac{k_{y_2}}{k_{y_1}}\leq c \left(\frac{y_2}{y_1}\right)^\alpha.
			\end{align*}
			where $\alpha=\alpha_S$ is a constant only depending on the boundary $S$ and $c=c_d$ only depends on the dimension $d$.
			\item $K_y(1)=1$.
			\item Let $y\geq z>0$ and $\tau\geq0$ then 
			\begin{align*}
				|k_{y+\tau}-k_y| \leq c_S\left(\left(y+\tau\right)^\alpha-y^\alpha\right)\frac{k_z}{z^\alpha}.
			\end{align*}
			Furthermore, if $(a,b)\subset\mathbb{R}^+$ is a non-degenerated interval, we have
			\begin{align*}
				\int_{a}^{b}\left|k_{y+\tau}-k_y\right|\mathrm{d}y\leq c_S\tau(b-a)(b+\tau)^{\alpha-1}\frac{k_z}{z^\alpha}.
			\end{align*}
		\end{enumerate}
	\end{lemma}
	\begin{proof}
		We prove properties 1 to 5\\
		\begin{enumerate}
			\item By Lemma \eqref{lem:Martinkernel_Harmmeas} and the definition of the harmonic measure we immediately obtain
			\begin{align*}
				(K_{y_2}(u_{y_1}|_S))(x)&=\int_Sk(x_{y_2},\xi)u(\xi_{y_1})\mathrm{d}\omega^{z_0}(\xi)
				\\
				&=\int_Su_{y_1}(\xi)\mathrm{d}\omega^{x_{y_2}}(\xi)
				\\
				&=u_{y_1}(x_{y_2})=u_{y_1+y_2}(x).
			\end{align*}
			\item 
			Since $k_{y_2}(\zeta,\xi)$ is positve harmonic for every fixed $\xi\in S$, we get similarily to the proof of property 1:
			\begin{align*}
				(k_{y_1}\circ k_{y_2})(x,\xi)&=\int_S k_{y_1}(x,\zeta)k_{y_2}(\zeta,\xi)\mathrm{d}\omega^{z_0}(\zeta) 
				\\
				&=\int_S k_{y_2}(\zeta,\xi)\mathrm{d}\omega^{x_{y_1}}(\zeta)= k_{y_1+y_2}(x,\xi).
			\end{align*}
			\item Follows directly from Lemma \eqref{lem:harnack_chain} since $k(\cdot,\xi)$ is positive and harmonic for every $\xi\in S$.
			\item By definition of the Martin kernel and Lemma \eqref{lem:Martinkernel_Harmmeas} we obtain
			\begin{align*}
				(K_y1)(x)=\int_Sk(x_y,\xi)\mathrm{d}\omega^{z_0}(\xi)=\int_S\mathrm{d}\omega^{x_y}(\xi)=1,
			\end{align*}
			since the harmonic measure is a probability measure.
			\item We begin by observing that for $x,\xi\in S$ and $t,\tau>0$ we have due to the Harnack's inequality
			\begin{align*}
				\left|\frac{\mathrm{d}}{\mathrm{d}t}k_{y+t}(x,\xi)\right|
				=\left|\frac{\partial}{\partial_x} k(x_{y+t},\xi)\cdot e_d\right|\leq\left|\nabla_x k(x_{y+t},\xi)\right|\leq \frac{k_{y+t}(x,\xi)}{y+t}.
			\end{align*}
			Using this and property 3, we obtain
			\begin{align*}
				\left|k_{y+\tau}-k_y\right|
				&\leq\int_{0}^{\tau}\left|\frac{\mathrm{d}}{\mathrm{d}y}k_{y+t}\right|\mathrm{d}t\\
				&\leq\int_{0}^{\tau}\frac{k_{y+t}}{y+t}\mathrm{d}t\\
				&\leq c_S \frac{k_y}{y^\alpha} \int_{0}^{\tau}(y+t)^{\alpha-1}\mathrm{d}t\\
				&=c_S\frac{k_y}{y^\alpha}\left((y+\tau)^\alpha-y^\alpha\right)\\
				&\leq c_S\left((y+\tau)^\alpha-y^\alpha\right)\frac{k_z}{z^\alpha}.
			\end{align*}
		Now we proceed with
		\begin{align*}
			\int_{a}^{b}\left|k_{y+\tau}-k_y\right|\mathrm{d}y
			&\leq c_S\frac{k_z}{z^\alpha}\int_{a}^{b}((y+\tau)^\alpha-y^\alpha)\mathrm{d}y\\
			&=c_S\frac{k_z}{z^\alpha}\int_{a}^{b}\int_{y}^{y+\tau}\alpha r^{\alpha-1}\mathrm{d}r\mathrm{d}y\\
			&\leq c_S\frac{k_z}{z^\alpha}(b-a)\tau(b+\tau)^{\alpha-1}.
		\end{align*}
		\end{enumerate}
	\end{proof}
	We will continue to introduce another two kernels with implicit dependence on the positive harmonic function $u$.
	\begin{definition}[Definition of $c_y$]
		For $y>0$ we define the kernel $c_y:S\times S\longrightarrow\mathbb{R}$ as
		\begin{align*}
			c_y(x,\xi):=\frac{\partial^1k}{\partial\sigma(x_{2y})}(x_y,\xi)=\left\langle\nabla^1k(x_y,\xi),\sigma(x_{2y})\right\rangle,
		\end{align*}
		where 
		\[
		\sigma(x):=
			\begin{cases}
				\frac{\nabla u(x)}{\|\nabla u(x)\|} & \nabla u(x)\neq0\\
				0 & \nabla u(x)=0
			\end{cases}
		\]
		and $\partial^1$ denotes the differentiation with respect to the first variable $x_y\in \mathcal{O}$ for $x\in S$. Note, that this kernel depends on the the positive harmonic function $u$, however for notational ease and since we work with fixed $u$, we will not indicate this dependence in a written manner.
	\end{definition}
	\begin{lemma}[Properties of $c_y$ \cite{MuellerRiegler2020}]
		\label{lem:cy_properties}
		The kernel $c_y$ has the following properties:
		\begin{enumerate}
			\item For any $\varphi\in C_b(S)$ we have $(C_y\varphi)(x)=\left\langle\nabla\varphi(x_y),\sigma(x_{2y})\right\rangle$ for all $x\in S$, where $\nabla\varphi$ denotes the derivative of the harmonic extension of the boundary values $\varphi$, so in particular
			\[
				\|\nabla u(x_{2y})\|=(C_yu_y)(x).
			\]
			\item$\begin{aligned}
				|c_y(x,\xi)|\leq c_S\frac{k_y(x,\xi)}{y}.
			\end{aligned}$
			\item$\begin{aligned}
				C_y(1)=0.
			\end{aligned}$
		\end{enumerate}
	\end{lemma}

	\begin{proof}
		We proof the properties $(1)-(3)$:
		\begin{enumerate}
			\item Using the definition of $c_y$ and property $(1)$ of $k_y$, we obtain
			\begin{align*}
				(C_y\varphi)(x)&= \int_S\left\langle\nabla^1k_y(x,\xi)\varphi(\xi)\mathrm{d}\omega^{z_0}(\xi),\sigma(x_{2y})\right\rangle\\
				&=\left\langle\int_S\nabla^1 k_y(x,\xi)\varphi(\xi)\mathrm{d}\omega^{z_0}(\xi) ,\sigma(x_{2y})\right\rangle\\
				&\stackrel{(*)}{=}\left\langle\nabla (K_y\varphi)(x) ,\sigma(x_{2y})\right\rangle= \left\langle\nabla \varphi(x_{y}),\sigma(x_{2y})\right\rangle.
			\end{align*}
			In the step marked with $(*)$ we are allowed to exchange integration and differentiation since $k_y$ is harmonic and Corollary \eqref{cor:Harnack_gradient} yields an integrable majorant:
			\begin{align*}
				\|\nabla^1k_y(x,\xi)\|\leq \frac{k_y(x,\xi)}{d(x_y,S)}.
			\end{align*}
			Setting $\varphi=u_y|_S$, we further get $(C_yu_y)(x)=\left\langle\nabla u(x_{2y}),\sigma(x_{2y})\right\rangle=\|\nabla u(x_{2y})\|$.
			\item Applying Lemma \eqref{lem:dist_to_boundary}, in particular the inequality $d(x_y,S)\geq c_Sy$ and Corollary \eqref{cor:Harnack_gradient} yields
			\begin{align*}
				|c_y(x,\xi)|\leq\|\nabla^1k(x_y,\xi)\|\leq\frac{k(x_y,\xi)}{d(x_y,S)}\leq\frac{1}{c_S}\frac{k(x_y,\xi)}{y}.
			\end{align*}
			\item This is a special case of property $(1)$. Inserting $u=1$ immediately yields the desired equality.
		\end{enumerate}
	\end{proof}	
	\begin{definition}[Definition of $b_y$]
		For $y>0$ we define the kernel $b_y:S\times S\longrightarrow\mathbb{R}$ as
		\begin{align*}
			b_y(x,\xi)=(k_y\circ c_y)(x,\xi)=\int_Sk_y(x,\zeta)c_y(\zeta,\xi)\mathrm{d}\omega^{z_0}(\zeta).
		\end{align*}
	\end{definition}
	\begin{lemma}[Properties of $b_y$ \cite{MuellerRiegler2020}]
		\label{lem:properties_b_y}
		The kernel $b_y$ has the following properties
		\begin{enumerate}
			\item $(y,x,\xi)\longmapsto b_y(x,\xi)$ is continuous on $(0,\infty)\times S\times S$.
			\item $\begin{aligned}
				|b_y(x,\xi)|\leq c_S\frac{k_y(x,\xi)}{y}.
			\end{aligned}$
			\item Let $\varphi\in C_b(S)$, then $\begin{aligned}
				(B_y\varphi)(x)=(K_y(C_y\varphi))(x).
			\end{aligned}$
			\item$\begin{aligned}
				B_y(1)=0.
			\end{aligned}$
		\end{enumerate}
	\end{lemma}
	\begin{proof}
		\begin{enumerate}
			\item This proof is not as straight forward as the other statements of this Lemma and in particular requires a characterization of the roots of the gradient of $u$. We will thus dedicate the next subsection to it.
			\item By property $(2)$ of $c_y$ in Lemma \eqref{lem:cy_properties}, the semi-group property and the quotient bound of $k_y$ in Lemma \eqref{lem:MartinKernel}, we get
			\begin{align*}
				|b_y(x,\xi)|&=|(k_y\circ c_y)(x,\xi)|\leq (k\circ|c_y|)(x,\xi)|\leq c_S \frac{k_{2y}(x,\xi)}{y}\leq2^\alpha c_S\frac{k_y(x,\xi)}{y}.
			\end{align*}
			\item With Fubini we obtain
			\begin{align*}
				(B_y\varphi)(x)&=\int_S\int_Sk_y(x,\zeta)c_y(\zeta,\xi)\mathrm{d}\omega^{z_0}(\zeta)\varphi(\xi)\mathrm{d}\omega^{z_0}(\xi)\\
				&=\int_Sk_y(x,\zeta)\left(\int_Sc_y(\zeta,\xi)\varphi(\xi)\mathrm{d}\omega^{z_0}(\xi)\right)\mathrm{d}\omega^{z_0}(\zeta)\\
				&=(K_y(C_y\varphi))(x).
			\end{align*}
			\item This is a special case of $(2)$. Plugging in $\varphi=1$ establishes the desired equality since $C_y1=0$ by Lemma \ref{lem:cy_properties}.
			
		\end{enumerate}
	\end{proof}
	\subsubsection{Continuity of $b_y$}
	We begin with the promised characterization of the Nullset of the gradient of harmonic functions $v$ defined on the epigraph of a Lipschitz domain. 
	\begin{lemma}[\cite{HavinMozol2016}(10.4)]
		\label{lem:ZeroOfGradient}
		Let $v$ be harmonic on the domain $D$, denoting an epigraph of a Lipschitz function and $S:=\partial D$. Furthermore let $\mathcal{H}^{d-1}$ denote the $(d-1)$-dimensional Hausdorff measure on $\mathcal{B}(\mathbb{R}^d)$. If for some $y>0$ and $E\subseteq S_y$ with $\mathcal{H}^{d-1}(E)>0$, the gradient of $v$ vanishes on $E$, i.e. $\nabla v|_E=0$, then $v$ is constant on $D$.
	\end{lemma}
	In essence, the above Lemma tells us that any harmonic function whose derivatives vanish on some subset of the shifted boundary with positive Hausdorff measure, is already constant. The proof can be found in \cite{HavinMozol2016} where it is stated for domains that admit tangent hyperplanes at every boundary point. With little effort, the argument can however be extended to domains whose boundary is almost everywhere differentiable.
	
	Let $\Phi:\mathbb{R}^{d-1}\longrightarrow\mathbb{R}$ denote the Lipschitz function defining the domain $D$ in Lemma \ref{lem:ZeroOfGradient} and $s$ be the surface measure on $S$. Since the $(d-1)$-dimensional Hausdorff measure, restricted to the graph of $\Phi$ is just a multiple of $s$ \cite[353]{Folland1999} and by Dahlberg's theorem, the measures $s$ and $\omega^z$ for $z\in \mathcal{O}$ are equivalent \cite[213]{Bass}, we have the same statement of Lemma \ref{lem:ZeroOfGradient} for $\omega^{z_0}$ replacing $\mathcal{H}^{d-1}$. With these considerations we can proceed to proof continuity.
	\begin{lemma} [Continuity of $b_y$ \cite{HavinMozol2016}]
		The function $(y,x,\xi)\longmapsto b_y(x,\xi)$ is continuous on $(0,\infty)\times S\times S$.
	\end{lemma}
	\begin{proof}
		Recall that, by definition
		\begin{align*}
			b_y(x,\xi)=\int_Sk(x_y,\zeta)c(\zeta_y,\xi)\mathrm{d}\omega^{z_0}(\zeta)
		\end{align*}
		and $c(\zeta_y,\xi)$ is discontinuous only when $\zeta_{2y}$ is such that $\nabla u(\zeta_{2y})=0$, by Lemma \ref{lem:ZeroOfGradient}. Now observe that the limit
		\begin{align*}
			\lim_{\substack{x\rightarrow a\\ \xi\rightarrow b\\ y\rightarrow y_0}} k(x_y,\zeta)c(\zeta_y,\xi) = k(a_{y_0},\zeta)c(\zeta_{y_0},b) \quad\omega^{z_0}\text{-a.e.},
		\end{align*}
		since $u$ is non-constant on $\mathcal{O}$ i.e. the gradient of $u$ is non-zero on $S_{2y_0}$, $\omega^{z_0}$-a.e. Applying Lemma \ref{lem:cy_properties} and \ref{lem:MartinKernel}, we obtain
		\begin{align*}
			\left|k(x_y,\zeta)c(\zeta_y,\xi)\right|\leq\frac{k(x_y,\zeta)k(\zeta_y,\xi)}{y}\leq\frac{y^{2\alpha-1}}{y_0^{2\alpha}}k(x_{y_0},\zeta)k(\zeta_{y_0},\xi)\leq\frac{y^{2\alpha-1}}{y_0^{2\alpha}}\left\|k_{y_0}\right\|^2_{C(\mathcal{K}\times S)},
		\end{align*}
		where $\mathcal{K}$ is a compact neighborhood of $(a,b)$. Continuity now follows from dominated convergence.
	\end{proof}
	\begin{remark}
		\label{rem:operator_height_continuity}
		For $\alpha\in L^\infty(S)$ and $x\in S$, the mappings 
		\begin{align*}
			y\longmapsto (K_y\alpha)(x)\\
			y\longmapsto (B_y\alpha)(x)\\
		\end{align*}
		where $y>0$, are continuous. This simply follows from the continuity of the kernels $k_y$ and $b_y$ and the fact that $|b_y|\leq c_S k_y/y$ is an integrable majorant if $y>y_0$ for some positive $y_0$, combined with Lebesgue's continuity criterion.
	\end{remark}

	\subsection{Variation and the set $\mathcal{V}(S)$}
	The following vertical variation will be the center of our examination and bounds will be derived under certain circumstances.
	\begin{definition}
		For $x\in S$ and $u:S\longrightarrow\mathbb{R}$ positive and harmonic, we define the mean vertical variation
		\begin{align*}
			V(x):=\int_{0}^{1}(B_yu_y)(x)\mathrm{d}y.
		\end{align*}
	\end{definition}
	We will show that this variation is indeed non-negative. By part (2) in Lemma \eqref{lem:properties_b_y} and part (1) of Lemma \eqref{lem:cy_properties}, we have
	\begin{align*}
		V(x)=\int_{0}^{1}(K_y(C_yu_y))(x)\mathrm{d}y=\int_{0}^{1}(K_y(\|\nabla u_{2y}\|))(x)\mathrm{d}y\geq \int_0^1\|\nabla u_{3y}(x)\|\mathrm{d}y.
	\end{align*}
	The last inequality follows from the fact that $\nabla u$ is also harmonic and property (1) in Lemma \eqref{lem:MartinKernel}:
	\begin{align*}
		(K_y(\|\nabla u_{2y}\|))(x)\geq\left\|\int_Sk_y(x,\xi)\nabla u_{2y}(\xi)\mathrm{d}\omega^{z_0}(\xi)\right\|=\|\nabla u_{3y}(x)\|.
	\end{align*}
	This means, that a bound on $V(x)$ also is a bound on the vertical variation 
	\[
		\int_0^1\|\nabla u_y(x)\|\mathrm{d}y,
	\]
	that is usually studied. We will denote by 
	\begin{align*}
		\mathcal{V}=\left\lbrace x\in S: V(x)<\infty\right\rbrace.
	\end{align*}
	the set of all points on $S$, such that their vertical variation is bounded. Comparing to \eqref{eq:BourgainPoints}, we may write $\mathcal{V}\subseteq\mathcal{V}^u(S)$.
	\subsection{Main Theorem For Near Half Spaces}
	We have so far established all the ingredients that are needed to formulate the main theorem of Riegler and Müller \cite{MuellerRiegler2020}.
	\begin{theorem}[\cite{MuellerRiegler2020}]
		\label{thm:main}
		Let $\mathcal{O}$ be a near-half space, $S=\partial \mathcal{O}$, and $y>1$. Then for any ball $B=B(x_0,r)$, with $x_0\in S$ and $r>0$ there exists $x\in B\cap S$ such that the variation $V(x)$ is bounded in the following way:\par 
		There exists a constant $c=c_{S,r}$, depending on the Lipschitz function parameterizing $S$, and the radius of $B$, such that
		\[
			V(x)\leq cu(x_y).
		\]
	\end{theorem}
	\begin{figure}[h!]
		\includegraphics[width=\textwidth]{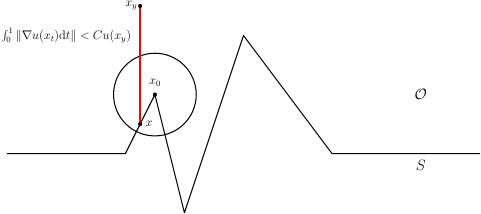}
		\caption{Overview of statement in Main Theorem \ref{thm:main}.}
	\end{figure}
	We begin preparations for the proof by introducing some notation.
	\subsection{Notation for Partitions}
	We define the set of all non-degenerate compact intervals in $\mathbb{R}^+$ as $\mathrm{segm}^+$, i.e.
	\[
		\mathrm{segm}^+=\left\lbrace[a,b]:a,b\in\mathbb{R}^+\quad\mathrm{and}\quad a<b\right\rbrace.
	\]
	Moreover, for $\Delta\in\mathrm{segm}^+$ denote 
	\begin{align*}
		m(\Delta):=\min(\Delta),&&M(\Delta)=\max(\Delta),&&|\Delta| =M(\Delta)-m(\Delta).
	\end{align*} 
	In order to ease notation later, we introduce the quantity $\varrho(\Delta):=\frac{M(\Delta)}{m(\Delta)}$.\\
	For any collection $\mu\subset\mathrm{segm}^+$, define $\mathcal{U}(\mu)$ as the union of $\mu$, i.e.
	\[
	\mathcal{U}(\mu) = \bigcup_{\Delta\in\mu}\Delta.
	\]
	A finite collection $\mu\subset\mathrm{segm}^+$ is called a partition of $\Delta\in\mathrm{segm}^+$ if $\mathcal{U}(\mu)=\Delta$ and for all $j_1,j_2\in\mu$ with $j_1\neq j_2$ we have $j_1\cap j_2\subseteq\lbrace x\rbrace$ for some $x\in\mathbb{R}^+$ or $j_1\cap j_2=\emptyset$ , i.e. intervals in $\mu$ can have at most one common point at their boundary. A partition $\tau$ is called a refinement of $\mu$, denoted with $\tau\succ\mu$, if for every $j_1\in\tau$ there exists $j_2\in\mu$ such that $j_1\subseteq j_2$.
	\subsection{Construction of the kernel $\omega_{\Delta}$}
	\begin{definition}[Definition of $b_{\Delta}$]
		For $\Delta\in\mathrm{segm}^+$, define the kernel $b_{\Delta}:S\times S\rightarrow\mathbb{R}$ as
		\[
			b_{\Delta}(x,\xi):=\int_{\Delta}b_y(x,\xi)\mathrm{d}y.
		\] 
	\end{definition}
	\begin{lemma}[Properties of $b_\Delta$ \cite{MuellerRiegler2020}]\label{lem:b_Delta}
		For $\Delta\in\mathrm{segm}^+$, the kernel $b_{\Delta}$ has the following properties:
		\begin{enumerate}
			\item $b_{\Delta}$ is continuous.
			\item $|b_\Delta(x,\xi)|\leq c_S\frac{\varrho(\Delta)^{\alpha-1}}{m(\Delta)}|\Delta|k_{m(\Delta)}(x,\xi)$.
		\end{enumerate}
	\begin{proof}
		\begin{enumerate}
			\item To begin with, we have the inequality
			\begin{align*}
				|b_y(x,\xi)|\leq c_S\frac{k_y(x,\xi)}{y}\leq c_S\frac{M(\Delta)^{\alpha-1}}{m(\Delta)^\alpha}k_{m(\Delta)}(x,\xi),
			\end{align*}
			where the last term in the inequality is constant in $y$ and integrable over $\Delta$. Since $b_y$ is continuous for every $y>0$ we can exchange limit and integration, i.e.
			\begin{align*}
				\lim_{(x,\xi)\rightarrow(x_0,\xi_0)}b_\Delta(x,\xi)=\lim_{(x,\xi)\rightarrow(x_0,\xi_0)}\int_\Delta b_y(x,\xi)\mathrm{d}y=\int_\Delta \lim_{(x,\xi)\rightarrow(x_0,\xi_0)}b_y(x,\xi)\mathrm{d}y=b_\Delta(x_0,\xi_0).
			\end{align*}
			\item With property $(2)$ of $b_y$ from Lemma \eqref{lem:properties_b_y} and the quotient bound for $k_y$ in Lemma \eqref{lem:MartinKernel}, we can estimate
			\begin{align*}
				|b_\Delta(x,\xi)|\leq c_S\int_\Delta\frac{k_y(x,\xi)}{y}\mathrm{d}y\leq c_S\int_\Delta k_{m(\Delta)}(x,\xi)\frac{y^{\alpha-1}}{m(\Delta)^\alpha}\mathrm{d}y\leq c_S\frac{M(\Delta)^{\alpha-1}}{m(\Delta)^\alpha}|\Delta|k_{m(\Delta)}(x,\xi).
			\end{align*}
		\end{enumerate}
	\end{proof}

	\end{lemma}
	\begin{definition}[Definition of $\tilde{\omega}_{\Delta}$]
		\label{def:omega-tilde}
		For positive $\varepsilon$ and $\Delta\in\mathrm{segm}^+$ we define the continuous kernel $\tilde{\omega}_{\Delta}:S\times S\longrightarrow\mathbb{R}$ as
		\begin{align*}
			\tilde{\omega}_{\Delta}:=k_{|\Delta|}-\varepsilon b_{\Delta}.
		\end{align*}
		Furthermore, notice that $\tilde{\omega}_\Delta\in L^1(\kappa\otimes\omega^{z_0})$ for any bounded measure $\kappa$ on $S$. 
	\end{definition}
	In the course of this thesis, we will impose several restrictions on how large $\varepsilon$ can be at most, for certain properties of the kernels $\tilde{\omega}_\Delta$, $\omega_\Delta$ and other objects that are derived from these kernels to hold. (e.g. Positivity of $\omega_\Delta$ Lemma \ref{lem:positivity_omega}, non-vanishing property of the measure $\nu_\varepsilon$ Lemma \ref{lem:nu_prop}, comparableness of $k_{1-y}$ and $\omega_y$ in Lemma \ref{lem:omega_ky_comparison}, etc.). The range of values for $\varepsilon$ will be determined by a constant $\varepsilon(S)$, which only depends on the Lipschitz boundary of the domain.
	We will now define the kernel $\Pi^\mu$ which depends on an partition $\mu$ of the interval $\Delta\in\mathrm{segm}^+$ and will proceed by algebraic manipulating the kernel to find estimates.
	\begin{definition}[Definition of $\Pi^{\mu}$]\label{def:Pi_mu}	
		Let $\Delta\in\mathrm{segm}^+$ and $\mu=\lbrace j_1,\dots,j_K\rbrace$ be a finite partition of $\Delta$ with $M(j_k)=m(j_{k+1})$, $1\leq k\leq K-1$. The kernel $\Pi^\mu$ is defined as follows:
		\begin{align*}
			\Pi^\mu:=\tilde{\omega}_{j_K}\circ\tilde{\omega}_{j_{K-1}}\circ\dots\circ\tilde{\omega}_{j_1}.
		\end{align*}
	\end{definition}
	The next step is to break up the kernel $\Pi^\mu$ into a sum of three different terms for which we will each derive estimates separately. Inserting definition \eqref{def:omega-tilde} we obtain
	\begin{align*}
		\Pi^\mu=\left(k_{|j_K|}-\varepsilon b_{j_K}\right)\circ\cdots\circ\left(k_{|j_1|}-\varepsilon b_{j_1}\right).
	\end{align*}
	For $l\in\left\lbrace\binom{K}{q}\right\rbrace$, we introduce the notation
	\begin{align*}
		r_s^l:=
		\begin{cases}
			-\varepsilon b_{j_s} & s\in l\\
			k_{|j_s|} & s\notin l
		\end{cases}
	\end{align*}
	and additionally define 
	\[
		\pi_l:=r_K^l\circ r_{K-1}^l\circ\cdots\circ r_1^l.
	\] 
	By Vieta's theorem we can now write the kernel $\Pi^\mu$ as follows:
	\begin{align}
		\label{eq:Pi-mu-repr}
		\Pi^\mu=k_{|\Delta|}+\sum_{q=1}^{K}\sum_{l\in\left\lbrace\binom{K}{q}\right\rbrace}\pi_l.
	\end{align}
	Our next goal is to isolated the summand where $q=1$ in equation \eqref{eq:Pi-mu-repr} and treat it separately. In the case $q=1$, we have that for every $l\in\left\lbrace\binom{K}{q}\right\rbrace$ there is only one kernel of the form $-\varepsilon b_j$, $j\in\mu$, in $\pi_l$. Therefore the semi-group property of the kernel $k_y$ enables us to write
	\begin{align*}
		\sum_{l\in\left\lbrace\binom{K}{1}\right\rbrace}\pi_l=-\varepsilon\sum_{s=1}^{K} k_{M(\Delta)-M(j_s)}\circ b_{j_s}\circ k_{m(j_s)-m(\Delta)},
	\end{align*}
	where $k_0$ is understood as the identity. Considering the identity
	\[
		b_\Delta=\sum_{s=1}^{K}b_{j_s},
	\]
	we write
	\begin{align*}
		\Pi^\mu=k_{|\Delta|}-\varepsilon b_\Delta + \varepsilon\sum_{s=1}^{K}\left(b_{j_s}-k_{M(\Delta)-M(j_s)}\circ b_{j_s}\circ k_{m(j_s)-m(\Delta)}\right)+\sum_{q=2}^{K}\sum_{l\in\left\lbrace\binom{K}{q}\right\rbrace}\pi_l.
	\end{align*}
	Denoting  
	\begin{align}
		\label{def:v_and_rho}
		v_s:=k_{M(\Delta)-M(j_s)}\circ b_{j_s} \quad \text{and} \quad
		\rho_{\mu}:=\sum_{q=2}^{K}\sum_{l\in\left\lbrace\binom{K}{q}\right\rbrace}\pi_l,
	\end{align}
	we finally arrive at the desired representation
	\begin{align*}
		\Pi^\mu = \tilde{\omega}_{\Delta}+\varepsilon\sum_{s=1}^{K}(b_{j_s}-v_s) +\rho_{\mu}.
	\end{align*}
	We proceed by proving bounds on all of three summands of $\Pi^\mu$ that guarantee convergence of $\Pi^{\mu_n}$ under certain conditions on the sequence of partitions $(\mu_n)_n$.
	\begin{lemma}[\cite{MuellerRiegler2020}]
		\label{lem:sum}
		Let $\Delta\in\mathrm{segm}^+$, then the following estimate holds:
		\begin{align*}
			\sum_{s=1}^{K}|b_{j_s}-v_s|\leq c_S\frac{\varrho(\Delta)^{\alpha-1}+\varrho(\Delta)^{2\alpha-1}}{m(\Delta)^2}|\Delta|^2 k_{m(\Delta)}.
		\end{align*}
	\end{lemma}
	\begin{proof}
		We start by bounding for every $1\leq s\leq K$:
		\begin{align*}
			|b_{j_s}-v_s|\leq \underbrace{|b_{j_s}-k_{M(\Delta)-M(j_s)}\circ b_{j_s}|}_{=:\Romannum{1}}+ \underbrace{|k_{M(\Delta)-M(j_s)}\circ b_{j_s}-v_s|}_{=:\Romannum{2}}.
		\end{align*}
		Let us estimate $\Romannum{1}$. Using the definition and available bounds for the involved kernels, the semi-group property of $k_y$ and Property 5 with $z=m(\Delta)$ we obtain:
		\begin{align*}
			\Romannum{1}
			&=\left|b_{j_s}-k_{M(\Delta)-M(j_s)}\circ b_{j_s}\right|\\
			&\leq \left|\int_{j_s}b_y\mathrm{d}y-k_{M(\Delta)-M(j_s)}\circ\int_{j_s} b_y\mathrm{d}y\right|\\
			&\stackrel{\mathrm{Fub}}{\leq} \int_{j_s} \left|k_y\circ c_y-k_{M(\Delta)-M(j_s)+y}\circ c_y\right|\mathrm{d}y\\
			&\stackrel{\eqref{lem:MartinKernel}}{\leq} \frac{c_S}{m(\Delta)}\int_{j_s}\left|k_{2y}-k_{M(\Delta)-M(j_s)+2y}\right|\mathrm{d}y\\
			&\leq \frac{c_S}{m(\Delta)}|j_s|(M(\Delta)-M(j_s))M(\Delta)^{\alpha-1}\frac{k_{m(\Delta)}}{m(\Delta)^\alpha}\\
			&\leq c_S\frac{M(\Delta)^{\alpha-1}}{m(\Delta)^{\alpha+1}}|j_s||\Delta|k_{m(\Delta)}.
		\end{align*}
			In order to factorize in the third line of the following calculation, we will need the identity with respect to convolution, the Dirac measure $\delta$ on $S$, for which the following holds: 
		\begin{align*}
			(p\circ\delta)(x,\xi) = \int_S p(x,\zeta)\mathrm{d}\delta_{\xi}(\zeta)=p(x,\xi),
		\end{align*}
		for any continuous kernel $p$.
		\begin{align*}
			\Romannum{2}
			&=\left|k_{M(\Delta)-M(j_s)}\circ b_{j_s}-k_{M(\Delta)-M(j_s)}\circ b_{j_s}\circ k_{m(j_s)-m(\Delta)}\right|\\
			&\leq\int_{j_s}\left|k_{M(\Delta)-M(j_s)}\circ k_y\circ c_y-k_{M(\Delta)-M(j_s)}\circ k_y\circ c_y\circ k_{m(j_s)-m(\Delta)}\right|\mathrm{d}y\\
			&\leq\int_{j_s}k_{M(\Delta)-M(j_s)+y}\circ |c_y| \circ \left|\delta-k_{m(j_s)-m(\Delta)}\right|\mathrm{d}y\\
			&\leq \frac{c_S}{m(\Delta)} \int_{j_s}\frac{(M(\Delta)-M(j_s)+y)^\alpha}{m(\Delta)^\alpha}k_{m(\Delta)}\circ\left|k_y-k_{m(j_s)-m(\Delta)+y}\right|\mathrm{d}y\\
			&\leq \frac{c_SM(\Delta)^\alpha}{m(\Delta)^{\alpha+1}}\int_{j_s} \left|k_{y+m(\Delta)}-k_{m(j_s)+y}\right|\mathrm{d}y
			\\
			&\leq \frac{c_SM(\Delta)^\alpha}{m(\Delta)^{\alpha+1}}\frac{(m(j_s)-m(\Delta))|j_s|(M(j_s)+m(j_s)-m(\Delta))^{\alpha-1}}{m(\Delta)^{\alpha}}k_{m(\Delta)}\\ 
			&\leq c_S\frac{M(\Delta)^{2\alpha-1}}{m(\Delta)^{2\alpha+1}}|\Delta||j_s|k_{m(\Delta)}.
		\end{align*}
	Summing up, we obtain 
	\begin{align*}
		\sum_{s=1}^{K}\left|b_{j_s}-v_s\right|
		&\leq c_{S}\left(\frac{M(\Delta)^{\alpha-1}}{m(\Delta)^{\alpha+1}}+\frac{M(\Delta)^{2\alpha-1}}{m(\Delta)^{2\alpha+1}}\right)\sum_{s=1}^{K}|j_s||\Delta|k_{m(\Delta)}\\
		&=c_{S}\frac{\varrho(\Delta)^{\alpha-1}+\varrho(\Delta)^{2\alpha-1}}{m(\Delta)^2}|\Delta|^2k_{m(\Delta)}.
	\end{align*}
	\end{proof}
	Before proceeding, we want to describe partitions $\mu$ of an interval $\Delta$ for which the elements of largest and smallest measure are comparable. To convey this notion, we define regular partitions.
	\begin{definition}[$\lambda$-regular partitions]
		Let $\mu$ be a finite partition of $\Delta\in\mathrm{segm}^+$ and $\lambda\geq1$. We call $\mu$ a $\lambda$-regular partition of $\Delta$ if
		\begin{align}
			\label{eq:lambda_reg1}
			\sup_{j\in\mu} |j|\leq \lambda\inf_{j\in\mu}|j|
		\end{align}
	\end{definition}
	\begin{remark}
		For $\lambda$-regular partitions $\mu$, the mean measure of the elements in $\mu$ is comparable to the largest measure of all elemtents in $\mu$. To see this, observe that
		\begin{align}
			\label{eq:lambda_reg2}
			|\mu|\inf_{j\in\mu}|j|\leq\sum_{j\in\mu}|j|=|\Delta|
		\end{align}
		and thus
		\[
		\sup_{j\in\mu}|j|\leq\lambda\frac{|\Delta|}{|\mu|}=\lambda\sum_{j\in\mu}\frac{|j|}{|\mu|}\leq\lambda\sup_{j\in\mu}|j|.
		\]
		Moreover, a partition $\mu$ is $1$-regular if and only if all elements have the same measure i.e. $\forall j_1, j_2 \in\mu$ we have $|j_1|=|j_2|$. We now apply this regularity of partitions in the following lemma.
	\end{remark}
	\begin{lemma}[\cite{MuellerRiegler2020}]
		\label{lem:rho}
		For $\lambda$-regular partitions $\mu=\left\lbrace j_1,\cdots,j_K\right\rbrace$ of $\Delta$ satisfying $|\Delta|\leq \beta m(\Delta)$ for some $\beta>0$ $, |\rho_{\mu}|$ is bounded as follows:
		\begin{align*}
			|\rho_{\mu}|\leq c_{\lbrace S,\lambda,\beta\rbrace}\frac{\varepsilon^2|\Delta|^2}{m(\Delta)^2}k_{m(\Delta)}, 
		\end{align*}
		where $c_{\lbrace S,\lambda,\beta\rbrace}$ does not depend on $m(\Delta)$.
	\end{lemma}
	\begin{proof}
		For $A\in\left\lbrace\binom{K}{q}\right\rbrace$ we define
		\[
		\gamma(A):=\sum_{s\in A}m(j_s)+\sum_{s\notin A}|j_s|.
		\]
		Obviously, $\gamma$ is bounded by $\gamma(A)\leq q M(\Delta) +|\Delta|\leq 2M(\Delta)q$. Using $\lambda$-regularity of $\mu$, in particular \eqref{eq:lambda_reg2}, the assumption $|\Delta|\leq\beta m(\Delta)$, which in turn implies $\varrho(\Delta)\leq1+\beta$, and the bounds for the kernels $b_{j_s}$ we obtain for $A\in\left\lbrace\binom{K}{q}\right\rbrace$:
		\begin{align}
			\label{eq:pi_bound}
			\begin{aligned}
				|\pi_A|
				&\leq k_{\gamma(A)}\prod_{s\in A}\left(c_S\varepsilon\varrho(\Delta)^{\alpha-1}\frac{|j_s|}{m(\Delta)}\right)\\	
				&\leq k_{\gamma(A)}\frac{1}{K^q}\left(c_{S}\lambda(1+\beta)^{\alpha-1}\frac{\varepsilon |\Delta|}{m(\Delta)}\right)^q\\
				&\leq k_{m(\Delta)}c_S\frac{\left(2M(\Delta)\right)^{\alpha}}{m(\Delta)^{\alpha}}\frac{q^{\alpha}}{K^q}\left(c_{S,\beta}\lambda\frac{\varepsilon|\Delta|}{m(\Delta)}\right)^q\\
				&\leq k_{m(\Delta)}\underbrace{c_S (2(1+\beta))^\alpha (c_{S,\beta})^2}_{:=c_{S,\beta}}\frac{\varepsilon^2|\Delta|^2}{m(\Delta)^2}\frac{q^{\alpha}}{K^q}\lambda^2\left(c_{S,\beta}\lambda\frac{\varepsilon|\Delta|}{m(\Delta)}\right)^{q-2}
			\end{aligned}
		\end{align}
		Since $\binom{K}{q}\leq\frac{K^q}{q!}$, we obtain the following bound:
		\begin{align*}
			|\rho_{\mu}|
			&\leq c_{S,\beta}k_{m(\Delta)}\frac{\varepsilon^2|\Delta|^2}{m(\Delta)^2}\sum_{q=2}^{K}\frac{q^{\alpha}}{q!}\lambda^2\left(c_{S,\beta}\lambda\varepsilon \frac{|\Delta|}{m(\Delta)}\right)^{q-2}.
 		\end{align*}
 		With the inequality
 		\begin{align*}
			\sum_{q=2}^{K}\frac{q^{\alpha}}{q!}\lambda^2\left(c_{S,\beta}\lambda\varepsilon\frac{|\Delta|}{m(\Delta)}\right)^{q-2} \leq
			\sum_{q=2}^{\infty}\frac{q^{\alpha}}{q!}\lambda^2(\varepsilon c_{S,\lambda,\beta})^{q-2} \stackrel{\varepsilon<1}{\leq}
			c_{\lbrace S,\lambda,\beta\rbrace},
 		\end{align*}
 		we arrive at
 		\begin{align*}
			|\rho_{\mu}|\leq c_{\lbrace S,\lambda,\beta\rbrace}\frac{\varepsilon^2|\Delta|^2}{m(\Delta)^2}k_{m(\Delta)}. 
 		\end{align*}
	\end{proof}
	\begin{lemma}[\cite{MuellerRiegler2020}]
		\label{lem:Pi-omega}
		For $\lambda$-regular partitions $\mu=\left\lbrace j_1,\cdots,j_K\right\rbrace$ of $\Delta\in\mathrm{segm}^+$ satisfying $|\Delta|\leq \beta m(\Delta)$ for some $\beta>0$, we have the following bound
		\begin{align*}
			|\Pi^{\mu}-\tilde{\omega}_{\Delta}|&\leq c_{\lbrace S,\lambda,\beta\rbrace}\left(\varrho(\Delta)^{\alpha-1}+\varrho(\Delta)^{2\alpha-1}+1\right)\frac{\varepsilon|\Delta|^2}{m(\Delta)^{2}} k_{m(\Delta)}\\
			&\leq c_{\lbrace S,\lambda,\beta\rbrace}\frac{\varepsilon|\Delta|^2}{m(\Delta)^{2}} k_{m(\Delta)}.
		\end{align*}
	\end{lemma}
	\begin{proof}
		Combining Lemmata \ref{lem:sum} and \ref{lem:rho}, we immediately obtain the desired inequality 
		\begin{align*}
			|\Pi^{\mu}-\tilde{\omega}_{\Delta}|
			&\leq \varepsilon\sum_{s=1}^{K}|b_{j_s}-v_s|+|\rho_{\mu}|\\
			&\leq\left(c_S\left(\varrho(\Delta)^{\alpha-1}+\varrho(\Delta)^{2\alpha-1}\right)+c_{\lbrace S,\lambda,\beta\rbrace}\right)\frac{\varepsilon|\Delta|^2}{m(\Delta)^{2}}k_{m(\Delta)},
		\end{align*}
		since $\varepsilon<1$. The second inequality stated in the Lemma follows from the simple observation that $\varrho(\Delta)\leq1+\beta$.
	\end{proof}
	\begin{lemma}[\cite{MuellerRiegler2020}]
		\label{lem:Pi}
		For $\lambda$-regular partitions $\mu=\left\lbrace j_1,\cdots,j_K\right\rbrace$ of $\Delta\in\mathrm{segm}^+$ satisfying $|\Delta|\leq \beta m(\Delta)$ for some $\beta>0$ and $0<m(\Delta)<1$, we have the following estimate for $|\Pi^{\mu}|$
		\begin{align*}
			|\Pi^{\mu}|\leq k_{|\Delta|} + c_{\lbrace S,\lambda,\beta\rbrace}\frac{\varepsilon|\Delta|}{m(\Delta)}k_{m(\Delta)}\leq k_{|\Delta|} + c_{\lbrace S,\lambda,\beta\rbrace}\varepsilon k_{m(\Delta)}.
		\end{align*}	
	\end{lemma}
	\begin{proof}
		Considering the definition of $\tilde{\omega}_{\Delta}=k_{|\Delta|}-\varepsilon b_{\Delta}$ we obtain the bound 
		\begin{align*}
			|\tilde{\omega}_{\Delta}|&\leq k_{|\Delta|}+\varepsilon c_S\frac{\varrho(\Delta)^{\alpha-1}}{m(\Delta)}|\Delta|k_{m(\Delta)}\\
			&\leq k_{|\Delta|}+\varepsilon c_S\frac{(1+\beta)^{\alpha-1}}{m(\Delta)}|\Delta|k_{m(\Delta)}.
		\end{align*}
		Furthermore Lemma \ref{lem:Pi-omega} yields the bound
		\begin{align*}
			|\Pi^{\mu}-\tilde{\omega}_{\Delta}|&\leq c_{\lbrace S,\lambda,\beta\rbrace}\frac{\varepsilon|\Delta|^2}{m(\Delta)^{2}} k_{m(\Delta)}\leq c_{\lbrace S,\lambda,\beta\rbrace}\frac{\varepsilon\beta|\Delta|}{m(\Delta)} k_{m(\Delta)}.
		\end{align*}
		Combining these two estimate with the estimate, we immediately obtain the desired bound.
	\end{proof}
	The bounds on $\Pi^\mu$, finally give reassurance that in fact the iterated kernel is again nice in a sense that it is majorized by the Martin kernel. In particular, it is again continuous and integrable w.r.t. to the measure $\kappa\otimes\omega^{z_0}$, for an arbitrary bounded measure $\kappa$ on $S$.
	
	Now we present one more minor technical detail, which seems trivial but is critical in the next lemma, thus it will be separately stated.
	\begin{lemma}
	\label{lem:lambda_beta}
	Let $\mu$ be any partition of $\Delta\in\mathrm{segm}^+$, $\nu\subseteq\mu$ a subpartition of $\mu$  and $\beta>0, \lambda\geq1$. Then the following holds:
	\begin{enumerate}
		\item If $|\Delta|\leq\beta m(\Delta)$ then $|\mathcal{U}(\nu)|\leq\beta m(\mathcal{U}(\nu)).$
		\item If $\mu$ is a $\lambda$-regular partition of $\Delta$, then $\nu$ is a $\lambda$-regular partition of $\mathcal{U}(\nu)$.
	\end{enumerate}
\end{lemma}
\begin{proof}
	\begin{enumerate}
		\item Obviously we have $m(\Delta)\leq m(\mathcal{U}(\nu))$. With this we obtain,
		\begin{align*}
			|\mathcal{U}(\nu)|=\sum_{j\in\nu}|j|\leq|\Delta|\leq\beta m(\Delta)\leq\beta m(\mathcal{U}(\nu)).
		\end{align*}
		\item Since $\nu\subseteq\mu$ we immediately have the inequality
		\begin{align*}
			\sup_{\Delta\in\nu}|\Delta|\leq\sup_{\Delta\in\mu}|\Delta|\leq\lambda\inf_{\Delta\in\mu}|\Delta|\leq\lambda\inf_{\Delta\in\nu}|\Delta|.
		\end{align*}
	\end{enumerate}
\end{proof}
	The main message of the above lemma is that $\lambda$-regularity is preserved for subpartitions, a circumstance that will be exploited in the following Lemma, in which we prove the convergence of the kernels $\Pi^\mu$.
	\begin{lemma}[\cite{MuellerRiegler2020}]
		\label{lem:refinement_of_part_kern}
		Let $\lambda\geq 1$ and $\sigma,\tau$ be $\lambda$-regular partitions of $\Delta\in\mathrm{segm}^+$ such that $\sigma\succ\tau$ and $|\Delta|\leq \beta m(\Delta)$, $\beta>0$. Then we have
		\begin{align*}
			|\Pi^{\tau}-\Pi^{\sigma}|\leq \frac{c_{\lbrace S,\lambda,\beta\rbrace}}{m(\Delta)}\sup_{j\in\tau}|j| k_{m(\Delta)}.
		\end{align*}
	\end{lemma}
	\begin{proof}
		Suppose $|\tau|=K$ and write $\tau=\lbrace\Delta_1,\dots,\Delta_K\rbrace$ with $m(\Delta_1)<\dots<m(\Delta_K)$. Define $\sigma_k$ as the partition of $\Delta_k$ by $\sigma$ , thus $\sigma_k:=\lbrace j\in\sigma|j\subseteq\Delta_k\rbrace$. Obviously $\sigma=\bigcup_{1\leq k\leq K}\sigma_k.$
		Now fix $i\in\lbrace2,\dots,K\rbrace$ and define $\sigma_i^-:=\bigcup_{0\leq k<i}\sigma_k$ as the the partition of the interval $\left(m(\Delta_1),m(\Delta_i)\right)$ by $\sigma$. Analogously define $\tau_i^+=\bigcup_{i<k\leq K}\lbrace\Delta_k\rbrace$ as the partition of the interval $\left(m(\Delta_{i+1}),M(\Delta_K)\right)$ by $\tau$. Additionally, we set $\sigma_1^-=\emptyset=\tau_{K+1}^+$, Finally we denote $\tau(i)=\sigma_i^-\cup\lbrace\Delta_i\rbrace\cup\tau_i^+$ and $\tau(K+1)=\sigma$. Writing $\Pi^{\tau(i)}:=\Pi_i$, this yields by telescoping
		\begin{align}
			\label{eq:Pi_telescoping}
			\Pi^{\tau}-\Pi^{\sigma}=\sum_{i=1}^{K}\left(\Pi_i-\Pi_{i+1}\right).
		\end{align}
		For $1\leq i\leq K$ we now have
		\begin{align}
			\label{eq:subparts}
			\Pi_i-\Pi_{i+1}=\Pi^{\tau_i^+}\circ\left(\tilde{\omega}_{\Delta_i}-\Pi^{\sigma_i}\right)\circ\Pi^{\sigma_i^-}.
		\end{align}
		In equation \eqref{eq:subparts} we have rewritten the difference of two subsequent refinements from left to right as concatenation of kernels depending on the partitions:
		\begin{align*}
			\tau_i^+ \quad&\mathrm{of} \quad \Delta_i^+:=\mathcal{U}(\tau_i^+)\\
			\sigma_i\quad&\mathrm{of}\quad \Delta_i\\
			\sigma_i^-\quad&\mathrm{of}\quad \Delta_i^-:=\mathcal{U}(\sigma_i^-)
		\end{align*}
		Since $|\Delta|\leq \beta m(\Delta)$, Lemma \ref{lem:lambda_beta} guarantees, that for each partition above, the measure of its union is bounded by the smallest number in the union with constant $\beta$. Furthermore, again by Lemma \ref{lem:lambda_beta}, since $\tau$ and $\sigma$ are $\lambda$-regular, the partitions above all are again $\lambda$-regular. Hence Lemmata \ref{lem:Pi-omega} and \ref{lem:Pi} apply and we have bounds for every summand in \eqref{eq:Pi_telescoping}. By the triangle inequality and $\varepsilon<1$, we get after lengthy computations:
		\begin{align*}
			|\Pi_i-\Pi_{i+1}|&\leq \left(k_{|\Delta_i^+|} + c_{\lbrace S,\lambda,\beta\rbrace} k_{m(\Delta_i^+)}\right)
			\circ
			\left( c_{\lbrace S,\lambda,\beta\rbrace} \frac{|\Delta_i|^2}{m(\Delta_i)^2} k_{m(\Delta_i)}\right)\circ
			\left(k_{|\Delta_i^-|} + c_{\lbrace S,\lambda,\beta\rbrace} k_{m(\Delta_i^-)}\right)
			\\
			&=c_{\lbrace S,\lambda,\beta\rbrace} \frac{|\Delta_i|^2}{m(\Delta_i)^2}\left(k_{|\Delta_i^+|+m(\Delta_i)}+c_{\lbrace S,\lambda,\beta\rbrace} k_{m(\Delta_i^+)+m(\Delta_i)}\right)\circ\left(k_{|\Delta_i^-|} + c_{\lbrace S,\lambda,\beta\rbrace}k_{m(\Delta_i^-)}\right)
			\\
			&\leq c_{\lbrace S,\lambda,\beta\rbrace} \frac{|\Delta_i|^2}{m(\Delta_i)^2}\left(k_{|\Delta_i^+|+m(\Delta_i)+|\Delta_i^-|}+c_{\lbrace S,\lambda,\beta\rbrace} k_{m(\Delta_i^+)+m(\Delta_i)+|\Delta_i^-|}\right.\\ &\mkern175mu \left.+c_{\lbrace S,\lambda,\beta\rbrace}k_{|\Delta_i^+|+m(\Delta_i)+m(\Delta_i^-)}+c_{\lbrace S,\lambda,\beta\rbrace} k_{m(\Delta_i^+)+m(\Delta_i)+m(\Delta_i^-)}\right).
		\end{align*}
		A rough estimation yields that the heights
		\begin{align*}
			&|\Delta_i^+|+m(\Delta_i)+|\Delta_i^-|\\
			&m(\Delta_i^+)+m(\Delta_i)+|\Delta_i^-|\\
			&|\Delta_i^+|+m(\Delta_i)+m(\Delta_i^-)\\
			&m(\Delta_i^+)+m(\Delta_i)+m(\Delta_i^-)
		\end{align*} 
		are all in the interval $[m(\Delta),3M(\Delta)]$. Thus we proceed using the quotient bound for the Martin kernel in Lemma \ref{lem:MartinKernel} and obtain
		\begin{align*}
			|\Pi_i-\Pi_{i+1}|&\leq c_{\lbrace S,\lambda,\beta\rbrace}\frac{|\Delta_i|^2}{m(\Delta_i)^2}\cdot c_S \left(\frac{3M(\Delta)}{m(\Delta)}\right)^\alpha k_{m(\Delta)}\\
			&\leq c_{\lbrace S,\lambda,\beta\rbrace}\frac{|\Delta_i|^2}{m(\Delta_i)^2} k_{m(\Delta)},		
		\end{align*}
		since $\varrho(\Delta)\leq1+\beta$..
		Calculating the sum in \eqref{eq:Pi_telescoping}, we finally obtain the estimate
		\begin{align*}
			|\Pi^{\tau}-\Pi^{\sigma}|&\leq \sum_{i=1}^{K}c_{\lbrace S,\lambda,\beta\rbrace}\frac{|\Delta_i|^2}{m(\Delta_i)^2} k_{m(\Delta)}\\
			&\leq \frac{c_{\lbrace S,\lambda,\beta\rbrace}}{m(\Delta)^2}\sup_{j\in\tau}|j|\sum_{i=1}^{K}|\Delta_i|k_{m(\Delta)}\\
			&\leq \frac{\beta c_{\lbrace S,\lambda,\beta\rbrace}}{m(\Delta)}\sup_{j\in\tau}|j| k_{m(\Delta)},
		\end{align*}
		as $|\Delta|\leq\beta m(\Delta)$.
 	\end{proof}
 	\subsubsection{Weak $\lambda$-regularity}
 		We compare the condition of $\lambda$-regularity given in \eqref{eq:lambda_reg1} with the following condition:\par 
 		Let $\tau$ be a partition of $\Delta\in\mathrm{segm}^+$ and $\lambda\geq1$. We call $\tau$ a weakly $\lambda$-regular partition of $\Delta$ if
		\begin{align}
 		\label{eq:weak_lambda_reg}
 			\sup_{j\in\tau}|j|\leq\lambda\frac{\Delta}{|\tau|}.
 		\end{align}
 		The name of this condition is justified, since \eqref{eq:lambda_reg1} implies \eqref{eq:lambda_reg2} and thus \eqref{eq:weak_lambda_reg}.
 		We make this distinction because in the paper by Havin and Mozolyako \cite{HavinMozol2016}, merely condition \eqref{eq:weak_lambda_reg} was imposed on the refining partitions in the proof of Lemma \ref{lem:refinement_of_part_kern}. This is sufficient for $1$-regular partitions, since $1$-regularity implies that all elements in the partition have equal measure and consequently every subpartition of a $1$-regular partition is again $1$-regular. For $\lambda>1$, we will however, demonstrate that this condition leaves a logical gap in the proof of Lemma \ref{lem:refinement_of_part_kern}, by showing that in general weak $\lambda$-regularity is not preserved for subpartitions if $\lambda>1$. 
 		
 		As announced, we will shortly digress from the topics above to attend to the following general Problem: Suppose we are given an arbitrary partition $\tau$ of $\Delta$, that is weakly $\lambda$-regular. The natural question that arises is whether any subpartition $\tau_1\subset\tau$ is again a weakly $\lambda$-regular partition of $\mathcal{U}(\tau_1)$. This statement can be promptly refuted with an easy counterexample.
 		Consider the intervals 
 		\begin{align*}
 		\Delta_1=\left[0,\frac{1}{16}\right) && \Delta_2=\left[\frac{1}{16},\frac{7}{16}\right) && \Delta_3=\left[\frac{7}{16},\frac{1}{2}\right) && \Delta_4=\left[\frac{1}{2},1\right].
 		\end{align*}
 		Now $\tau:=\left\lbrace\Delta_i|i\in\lbrace1,\dots,4\rbrace\right\rbrace$ is a weakly $2$-regular partition of $\left[0,1\right]$, since
 		\[
 		\frac{1}{2}=|\Delta_4|=\max_{\Delta\in\tau}|\Delta|\leq 2\frac{|\left[0,1\right]|}{|\tau|}=\frac{1}{2}.
 		\]
 		However the partition $\lbrace\Delta_1,\Delta_2,\Delta_3\rbrace$ of $\left[0,\frac{1}{2}\right]$ is no longer weakly $2$-regular as
 		\[
 		\frac{3}{8}=|\Delta_2|=\max\lbrace|\Delta_1|,|\Delta_2|,|\Delta_3|\rbrace >2\frac{|\left(0,\frac{1}{2}\right)|}{3}=\frac{1}{3}.
 		\]
 		The next question that arises, is whether it is possible to derive a bound on the constant of regularity of an arbitrary subpartition. The following attempt delivers the sought after estimates but only under some restricting assumptions on the subpartition.
		\begin{theorem}
			\label{thm:reg_bound}
			Let $\Delta$ be an interval with $|\Delta|<\infty$ and $\tau$ be a weakly $\lambda$-regular partition of $\Delta$. Furthermore let $\mu\subset\tau$ such that $\lambda|\mu|<|\tau|$. Then the partition $\tau_1:=\tau\setminus\mu$ of the set $\mathcal{U}(\tau_1)$ is weakly regular with constant $
			\frac{\lambda|\tau_1|}{|\tau|-\lambda|\mu|}
			$, so in particular we have 
			\[
				\sup_{j\in\tau_1}|j|\leq\frac{\lambda|\tau_1|}{|\tau|-\lambda|\mu|}\frac{|\mathcal{U}(\tau_1)|}{|\tau_1|}=\frac{\lambda}{|\tau|-\lambda|\mu|}|\mathcal{U}(\tau_1)|.
			\]
		\end{theorem}
		\begin{proof}
			Since $\tau$ is weakly $\lambda$-regular and $\sup_{j\in\mu}|j|\leq\sup_{j\in\tau}|j|$, we have the estimates
			\begin{align*}
				\lambda\frac{|\mathcal{U}(\tau_1)|}{|\tau_1|}&=\frac{\lambda}{|\tau_1|}\left(|\mathcal{U}(\tau)|-|\mathcal{U}(\mu)|\right)\\
				&\geq\frac{1}{|\tau_1|}\left(|\tau|\sup_{j\in\tau}|j|-\lambda\sum_{j\in\mu}|j|\right)\\
				&\geq\frac{|\tau|-\lambda|\mu|}{|\tau_1|}\sup_{j\in\tau}|j|.
			\end{align*}
		Rearranging yields the desired inequality.
		\end{proof}
		\begin{remark}
			In the setting of Theorem \ref{thm:reg_bound}, the estimate above gives us a useful bound on the constant of regularity of the partition $\tau_1$ that results after taking away a sufficiently small subset $\mu$ of the original partition $\tau$. The condition $\lambda|\mu| <|\tau|$ can be reformulated as 
			\[
				\lambda|\mu| <|\tau|\Longleftrightarrow\lambda\left(|\tau|-|\tau_1|\right) <|\tau|\Longleftrightarrow |\tau_1|>\frac{\lambda-1}{\lambda}|\tau|.
			\]
			Thus, depending on the constant of weak regularity of $\tau$ we obtain information about the constant of weak regularity of subsets $\tau_1$ with cardinality strictly larger than $\frac{\lambda-1}{\lambda}|\tau|$.\par
			For $\lambda=1$ we can obtain a bound since $|\tau_1|\geq1>0$, however in this case we anyways have that $\tau_1$ is weakly $1$ regular since weak $1$-regularity of $\tau$ is equivalent to all elements of $\tau$ having the same length.\par
			If $\lambda$ is large, we obtain no useful information from Theorem \ref{thm:reg_bound}. Now consider the case that $\lambda=1+\varepsilon$ with small $\varepsilon>0$. Then we can bound subpartitions $\tau_1$ that have at least $\frac{\varepsilon}{1+\varepsilon}|\tau|$ elements however, no useful bound can be derived for subpartitions with smaller cardinality. This begs the question, whether it is even possible to a bound the constant of weak regularity for arbitrarily small subpartitions with respect to $\tau$. This problem is negatively answered in the following Theorem.
		\end{remark}
		\begin{theorem}
			\label{thm:counterex_reg}
			Let $\Delta$ with $|\Delta|<\infty$, $A\in\mathbb{N}$ and $\lambda>1$. Then there exists a weakly $\lambda$-regular partition $\tau$ of $\Delta$ and disjoint subpartitions $\tau_1,\tau_2\subset\tau$, i.e. $\tau_1\cap\tau_2=\emptyset$, such that $\tau=\tau_1\cup\tau_2$ and
			\begin{align}
				\label{eq:arb_irr}
				\sup_{j\in\tau_1}|j|\geq A\cdot\frac{|\mathcal{U}(\tau_1)|}{|\tau_1|}. 		
			\end{align}
		\end{theorem}
		\begin{proof}
			We first construct a partition without taking into account the interval we obtain when taking the union over the partition and then specify the construction to any given bounded interval.\par 
			In the first step we construct $\tau_1$. Let $\omega>0$ and define intervals $\Delta_0,\dots,\Delta_A$ such that
			\[
				M(\Delta_i)=m(\Delta_{i+1})\quad\text{for all } i\in\lbrace0,\dots,A-1\rbrace,
			\]
			$|\Delta_0|=\omega$ and $|\Delta_i|=\frac{\omega}{A^2}$ for $i\in\lbrace1,\dots, A\rbrace$. Obviously $|\Delta_0|\geq|\Delta_i|$ holds. We set $\tau_1=\lbrace\Delta_0,\dots\Delta_A\rbrace$ and verify equation \eqref{eq:arb_irr}:
			\begin{align*}
				A\cdot\frac{|\mathcal{U}(\tau_1)|}{|\tau_1|}=\frac{A}{A+1}\left(\omega+A\cdot\frac{\omega}{A^2}\right)=\omega=\sup_{j\in\tau_1}|j|.
			\end{align*}
			Choose $\varepsilon>0$ with $\frac{1+\varepsilon}{\lambda}<1$ and define $N\in\mathbb{N}$, $N>A+1$ large enough such that the inequality 
			\begin{align}
				\label{eq:lim_cond}
				(N-A-1)(1+\varepsilon)\geq \lambda\left(\frac{N}{\lambda}-1-\frac{1}{A}\right)
			\end{align}
			is satisfied. This choice is possible since the quotient
			\[
				q_n:=\frac{n-A-1}{\frac{n}{\lambda}-1-\frac{1}{A}}
			\]
			converges to $\lambda$, so in particular $q_n\geq\frac{\lambda}{1+\varepsilon}$ for large enough $n$. We proceed the construction by defining intervals $\Delta_{A+1},\dots,\Delta_N$ satisfying
			\[
				M(\Delta_i)=m(\Delta_{i+1})\quad\text{for all } i\in\lbrace A,\dots,N-1\rbrace,
			\]
			and $|\Delta_i|=\frac{1+\varepsilon}{\lambda}\omega$ for all $i\in\lbrace A+1,\dots, N\rbrace$. By our choice of $\varepsilon$ we have $|\Delta_i|<\omega$, so denoting $\tau:=\lbrace\Delta_0,\dots,\Delta_N\rbrace$ we obtain $\sup_{j\in\tau}|j|=\omega$. This construction yields a weakly $\lambda$-regular partition $\tau$ since
			\begin{align*}
				\lambda\frac{|\mathcal{U}(\tau)|}{|\tau|}&=\frac{\lambda}{N}\left(\sum_{j\in\tau_1}|j|+\sum_{j\in\tau\setminus\tau_1}|j|\right)
				\\
				&=\frac{\lambda}{N}\left(\omega+\frac{\omega}{A}+(N-A-1)\frac{1+\varepsilon}{\lambda}\omega\right)
				\\
				&\stackrel{\eqref{eq:lim_cond}}{\geq}\frac{\lambda\omega}{N}\left(1+\frac{1}{A}+\left(\frac{N}{\lambda}-1-\frac{1}{A}\right)\right)=\omega.
			\end{align*}
			Now let $\Delta$ be a bounded interval. The computations above show that $|\mathcal{U}(\tau)|$ is a non-zero multiple of $\omega$ so we can determine $\omega$ such that $|\mathcal{U}(\tau)|=|\Delta|$. If we choose $m(\Delta_0)=m(\Delta)$, the construction above yields a desired partition $\tau$ of $\Delta$.
		\end{proof}
		This result essentially tells us, that for any weakly $\lambda$-regular partition, the weak regularity of an arbitrary subpartition cannot be bounded. Therefore the process of piecewise refinement of a coarse partition to a fine partition, that the proof of Lemma \ref{lem:refinement_of_part_kern} heavily relies on, is not applicable to arbitrary weakly $\lambda$-regular partitions. This can be seen as follows. Let $\lambda>1$ be arbitrary and assume that the constant of regularity of any subpartition of a weakly $\lambda$-regular partition can be bounded by some $C$. Now we take the weakly $\lambda$-regular partition $\tau$ constructed in Theorem \ref{thm:counterex_reg} such that the subpartition $\tau_1$ is not $A$-regular, where $A>C$. Construct the partition $\mu$ by halving all intervals in $\tau$ i.e.
		 \begin{align*}
			\mu:=\left\lbrace\left[m(\Delta),m(\Delta)+\frac{|\Delta|}{2}\right],\left[m(\Delta)+\frac{|\Delta|}{2},M(\Delta)\right]:\Delta\in\tau\right\rbrace.
		\end{align*}
		Analogously, construct $\mu_1$ as the collection of all halved intervals in $\tau_1$. Now observe that $\mu$ is again weakly $\lambda$-regular and similarily $\mu_1$ is again not weakly $A$-regular, since
		\begin{align*}
			\sup_{j\in\tau}|j|=2\sup_{j\in\mu}|j| \quad \text{and}\quad |\tau|=2|\mu|
		\end{align*}
		as well as
			\begin{align*}
			\sup_{j\in\tau_1}|j|=2\sup_{j\in\mu_1}|j| \quad \text{and}\quad |\tau_1|=2|\mu_1|.
		\end{align*}
		Obviously $\mu$ is a refinement of $\tau$ and $\mu_1$ a refinement of $\tau_1$.
		\par 
	\subsection{The kernel $\omega_\Delta$ and its properties}
		Finally we have arrived at a point where we can define the kernel $\omega_{\Delta}$ and guarantee its existence. 
		\begin{theorem}[\cite{MuellerRiegler2020}]
			\label{thm:convergence_Pi}
			Fix $\lambda\geq1$, $\Delta\in\mathrm{segm}^+$ and let $(\tau_n(\Delta))_n=(\tau_n)_n$ be a sequence of $\lambda$-regular partitions such that for all $n\in\mathbb{N}$ we have $\tau_{n+1}\succ\tau_{n}$ and
			\begin{align}
			\label{eq:vanishing_fineness}
			\lim_{n\rightarrow\infty}\sup_{j\in\tau_n}|j|=0.
			\end{align}
			Then, for any compact $\mathcal{K}\subset S$, the sequence of continuous kernels $\left(\Pi^{\tau_n}\right)_n\subset C(S\times S)$ is uniformly convergent and independent of the approximating sequence of partitions and their degree of regularity. Furthermore there exists a function $\omega_\Delta\in C(S\times S)$ such that
			\begin{align}
				\omega_\Delta:=\lim_{n\rightarrow\infty}\Pi^{\tau_n} \quad\text{on } \mathcal{K}\times S
			\end{align}
			 for any compact subset $\mathcal{K}$ of $S$.
		\end{theorem}
		\begin{proof}
			Fix any compact $\mathcal{K}\subset S$. By Lemma \ref{lem:refinement_of_part_kern} we have 
			\begin{align}
				\label{eq:Cauchy_ineq}
				\forall m\geq n:\quad\left|\Pi^{\tau_m}-\Pi^{\tau_{n}}\right|\leq C \sup_{j\in\tau_n}|j| k_{m(\Delta)}
			\end{align}
			everywhere on $S\times S$. Lemma \ref{thm:mart-kern-bdd} implies $\left\|k_{m(\Delta)}\right\|_{C(\mathcal{K}\times S)}<\infty$, hence $(\Pi^{\tau_n})_n:\mathcal{K}\times S\longrightarrow\mathbb{R}$ is a Cauchy sequence and converges uniformly to some $\omega_{\Delta,\mathcal{K}}\in C(\mathcal{K}\times S)$. Now consider another sequence of refining $\lambda$-regular partition $(\mu_n)_n$ such that again \eqref{eq:vanishing_fineness} holds, i.e. the uniform limit of $\left(\Pi^{\mu_n}\right)_n$ exists and is denoted by $\omega_{\Delta,\mathcal{K}}^\mu$. We start by examining the difference $\left|\Pi^{\tau_n}-\Pi^{\mu_n}\right|$. For that, consider the joint refinement of $\tau_n$ and $\mu_n$ denoted by $\tilde{\nu}_n$. Now we construct a $\lambda$-regular  refinement $\nu_n$ of $\tilde{\nu}_n$. Define $\eta:=\inf\lbrace|j|:j\in\tilde{\nu}_n\rbrace$ and divide every $j\in\tilde{\nu}_n$ into intervals $j_1,\dots,j_k,j_{k+1}$ such that for each $i\in\lbrace1,\dots,k\rbrace$ we have $|j_i|=\eta$ and $|j_{k+1}|\leq \lambda \eta$. Note that $k=0$ is possible if already $|j|\leq \lambda \eta$ holds. Now $\nu_n$ obviously is a refinement of both $\tau_n$ and $\mu_n$ and $\inf\lbrace|j|:j\in\nu_n\rbrace=\eta$. Therefore $\nu_n$ is $\lambda$-regular since by definition
			\begin{align*}
				\sup_{j\in\nu_n}|j|\leq \lambda \eta = \lambda\inf_{j\in\nu_n}|j|.
			\end{align*}
			Applying Lemma \eqref{lem:refinement_of_part_kern} onto the differences $\Pi^{\tau_n}-\Pi^{\nu_n}$ and $\Pi^{\mu_n}-\Pi^{\nu_n}$, we immediately obtain
			\begin{align*}
				|\Pi^{\tau_n}-\Pi^{\mu_n}| \leq C \max\left\lbrace \sup_{j\in\tau_n}|j|,\sup_{j\in\mu_n}|j|\right\rbrace \left\|k_{m(\Delta)}\right\|_{C(\mathcal{K}\times S)}
			\end{align*}
			everywhere on $\mathcal{K}\times S$. The independence of the approximating sequence of partitions now follows from the inequality:
			\begin{align*}
				\left\|\omega_{\Delta,\mathcal{K}}-\omega_{\Delta,\mathcal{K}}^\mu\right\|_{C(\mathcal{K}\times S)}\leq \left\|\omega_{\Delta,\mathcal{K}}-\Pi^{\tau_n}\right\|_{C(\mathcal{K}\times S)}+\left\|\Pi^{\tau_n}-\Pi^{\mu_n}\right\|_{C(\mathcal{K}\times S)}+\left\|\Pi^{\mu_n}-\omega_{\Delta,\mathcal{K}}^\mu\right\|_{C(\mathcal{K}\times S)}.
			\end{align*}
			Choose as compact sets the exhausting sequence $\lbrace\overline{B}_n\cap S : n\in\mathbb{N}\rbrace$, where $\bar{B}_n$ is the closed ball centered at the origin with radius $n$ and define
			\begin{align*}
				\omega_\Delta(x,\xi):=\omega_{\Delta,B_n}\quad \text{for } (x,\xi)\in B_n\times S.
			\end{align*}
			The kernel $\omega_\Delta$ is well-defined, since for any $n\in\mathbb{N}$ we have $\omega_{\Delta,B_{n+1}}|_{B_n\times S}=\omega_{\Delta,B_{n}}$
		\end{proof}		
		\begin{definition} [Definition of $\omega_{\Delta}$]
			Let $\Delta\in\mathrm{segm}^+$ and for $n\in\mathbb{N}$ define
			\[
				d_n(\Delta):=\left\lbrace\left[m(\Delta)+\frac{(i-1)|\Delta|}{2^n},m(\Delta)+\frac{i|\Delta|}{2^n}\right]:i\in\lbrace1,\dots,2^n\rbrace\right\rbrace
			\].
			The above partition obviously is $1$-regular and although the kernel $\omega_\Delta$ is independent of the approximating sequence of refining partitions we will mostly use $(d_n)_n$. 	
		\end{definition}
	
		\begin{lemma}[Properties of $\omega_\Delta$ \cite{MuellerRiegler2020}]
			\label{lem:omega-prop}
			The kernel $\omega_\Delta$ has the following properties.
			\begin{enumerate}
				\item For all $x\in S$, the identity
				\begin{align*}
					\int_S \omega_\Delta(x,\xi)\mathrm{d}\omega^{z_0}(\xi) =1
				\end{align*}
				holds.
				\item For any segment $\Delta\in\mathrm{segm}^+$ with $|\Delta|\leq \beta m(\Delta)$, for some $\beta>0$, we have the inequality
				\[
				|\omega_\Delta-\tilde{\omega}_\Delta|\leq c_{S,\beta}\frac{\varepsilon|\Delta|^2}{m(\Delta)^{2}} k_{m(\Delta)}.
				\]
				\item For $0<a<c$ and $b\in(a,c)$ we have $\omega_{[a,c]}=\omega_{[b,c]}\circ\omega_{[a,b]}.$
			\end{enumerate}
		\end{lemma}
		\begin{proof}
			\begin{enumerate}
				\item We first observe that for all $x\in S$ and $n\in\mathbb{N}$, we have 
				\begin{align}
					\label{eq:Pi_prob_meas}
					\int_S \Pi^{d_n}(x,\xi)\mathrm{d}\omega^{z_0}(\xi)=1.
				\end{align}
				Choose an arbitrary $\Delta\in\mathrm{segm}^+$, then
				\[
					\int_S\tilde{\omega}_\Delta(x,\xi)\mathrm{d}\omega^{z_0}(\xi)=\int_S k_\Delta(x,\xi)\mathrm{d}\omega^{z_0}(\xi)-\varepsilon\int_Sb_\Delta(x,\xi)\mathrm{d}\omega^{z_0}(\xi).
				\]
				The last summand can be further written as
				\[
					\int_Sb_\Delta(x,\xi)\mathrm{d}\omega^{z_0}(\xi)=\int_S\int_\Delta b_y(x,\xi)\mathrm{d}y\mathrm{d}\omega^{z_0}(\xi)=\int_\Delta(B_y1)(x)\mathrm{d}y=0.
				\]
				Fubini can be applied due to the bound of $|b_y|$ in Lemma \ref{lem:properties_b_y}.
				Reinserting above gives
				\[
					\int_S\tilde{\omega}_\Delta(x,\xi)\mathrm{d}\omega^{z_0}(\xi)=(K_{|\Delta|}1)(x)=1.
				\]
				Now write $d_n=\lbrace j_1,\dots,j_K\rbrace $. Then observe the following:
				\begin{align*}
					\int_S\Pi^{d_n}(x,\xi)\mathrm{d}\omega^{z_0}(\xi)
					&=\int_S\int_S(\tilde{\omega}_{j_K}\circ\cdots\circ\tilde{\omega}_{j_2})(x,\zeta)\tilde{\omega}_{j_1}(\zeta,\xi)\mathrm{d}\omega^{z_0}(\zeta)\mathrm{d}\omega^{z_0}(\xi)\\
					&=\int_S(\tilde{\omega}_{j_K}\circ\cdots\circ\tilde{\omega}_{j_2})(x,\zeta)\underbrace{\int_S\tilde{\omega}_{j_1}(\zeta,\xi)\mathrm{d}\omega^{z_0}(\xi)}_{=1}\mathrm{d}\omega^{z_0}(\zeta)\\
					&=\int_S\Pi^{d_n\setminus\lbrace j_1\rbrace}(x,\zeta)\mathrm{d}\omega^{z_0}(\zeta).
				\end{align*}
				Fubini is applicable since $\Pi^\mu:=\tilde{\omega}_{j_K}\circ\cdots\circ\tilde{\omega}_{j_2}$, where $\mu=\lbrace j_2,\dots,j_K\rbrace$, has an integrable majorant given in Lemma \ref{lem:Pi} and $\Pi^\mu(x,\zeta)\tilde{\omega}_{j_1}(\zeta,\xi)$ is integrable.
				Applying this identity $K$ times finally yields \eqref{eq:Pi_prob_meas}. Since
				\[
					|\Pi^{d_n}-\omega_\Delta|\leq c_\Delta 2^{-n}k_{m(\Delta)},
				\]
				we can apply dominated convergence to obtain
				\begin{align*}
					\int_S\omega_\Delta(x,\xi)\mathrm{d}\omega^{z_0}(\xi)=\lim_{n\rightarrow\infty}\int_S\Pi^{d_n}(x,\xi)\mathrm{d}\omega^{z_0}(\xi)=1
				\end{align*}
				for all $x\in\mathcal{K}$, where $\mathcal{K}$ is an arbitrary compact subset of $S$ and thus for all $x\in S$.
				
				\item We apply Lemma \ref{lem:Pi-omega} with $\mu=d_n(\Delta)$, i.e $\lambda=1$ to obtain a constant $c_{\lbrace S,\lambda,\beta\rbrace} = c_S$ and the inequality
				\begin{align*}
					|\Pi^{d_n}-\tilde{\omega}_\Delta|\leq c_{S,\beta}\frac{\varepsilon|\Delta|^2}{m(\Delta)^{2}} k_{m(\Delta)}.
				\end{align*}
				Passing to the limit $n\rightarrow\infty$ we obtain the desired inequality.			
				
				\item Write $\Delta:=[a,c]$, $\Delta^-:=[a,b]$, $\Delta^+=[b,c]$, then by Theorem \ref{thm:convergence_Pi} we have
				\begin{align*}
					\left\|\Pi^{d_n(\Delta^+)}-\omega_{\Delta^+}\right\|_{C(\mathcal{K}\times S)}\rightarrow0 \quad\text{and}\quad
					\left\|\Pi^{d_n(\Delta^-)}-\omega_{\Delta^-}\right\|_{C(\mathcal{K}\times S)}\rightarrow0
				\end{align*}
				for all compact $\mathcal{K}\subset S$.
				Now define $\tau_n:=\tau_n^+\cup\tau_n^-$ and observe that $\tau_n$ is a $\lambda'$ regular sequence of refining partitions of $\Delta$, where
				\begin{align*}
					\lambda'=\frac{\max(b-a,c-b)}{\min(b-a,c-b)}.
				\end{align*}
				Since $d_n(\Delta)$ is $1$-regular, it is in particular $\lambda'$-regular and again Theorem \ref{thm:convergence_Pi} guarantees uniform convergence of $\Pi^{\tau_n}$ to $\omega_\Delta$ in $C(\mathcal{K}\times S)$, for all compact $\mathcal{K}\subset S$. Obviously we have
				\begin{align*}
					\Pi^{\tau_n}=\Pi^{\tau_n^+}\circ\Pi^{\tau_n^-}
				\end{align*}
				everywhere on $S\times S$ and since Lemma \ref{lem:Pi} provides integrable bounds on $\Pi^{\tau_n^+}(x,\cdot)\cdot\Pi^{\tau_n^-}(\cdot,\xi)$ , we can exchange limits and integration in
				\begin{align*}
					\forall\mathcal{K}\subset S \text{ compact }\forall (x,\xi)\in\mathcal{K}\times S:\\
					\omega_\Delta(x,\xi)=\lim_{n\rightarrow\infty}\Pi^{\tau_n}(x,\xi)&=\int_S\lim_{n\rightarrow\infty}\Pi^{\tau_n^+}(x,\zeta)\cdot\Pi^{\tau_n^-}(\zeta,\xi)\mathrm{d}\omega^{z_0}(\zeta)\\
					&=\lim_{k\rightarrow\infty}\int_{S\cap B_k}\lim_{n\rightarrow\infty}\Pi^{\tau_n^+}(x,\zeta)\cdot\Pi^{\tau_n^-}(\zeta,\xi)\mathrm{d}\omega^{z_0}(\zeta)\\
					&=\lim_{k\rightarrow\infty}\int_{S\cap B_k}\omega_{\Delta^+}(x,\zeta)\cdot\omega_{\Delta^-}(\zeta,\xi)\mathrm{d}\omega^{z_0}(\zeta)\\
					&=\left(\omega_{\Delta^+}\circ\omega_{\Delta^-}\right)(x,\xi),
				\end{align*}
				since part $(2)$ gives an integrable bound on $\omega_{\Delta^+}(x,\cdot)\cdot\omega_{\Delta^-}(\cdot,\xi)$. The above now yields the desired identity.
			\end{enumerate}
		\end{proof}
		\subsubsection{Positivity of $\omega_\Delta$ and consequences}
		\begin{lemma}[\cite{MuellerRiegler2020}]
			\label{lem:positivity_omega}
			There exists $\varepsilon(S)>0$ such that for any $\varepsilon\in(0,\varepsilon(S))$, the kernel $\omega_\Delta$ is positive if $|\Delta|\geq m(\Delta)$.
		\end{lemma}
		\begin{proof}
			We begin by showing that for any $\Delta\in\mathrm{segm}^+$ satisfying $m(\Delta)\leq|\Delta|\leq 3 m(\Delta)$ and any $\lambda$-regular partition $\mu=\lbrace j_1,\dots,j_K\rbrace$ of $\Delta$, the kernel $\Pi^\mu$ is positive and pointwise bounded away from $0$. As can be seen in the discussion following Definition \ref{def:Pi_mu} we are able to write:
			\begin{align*}
				\Pi^\mu=k_{|\Delta|}-\varepsilon \sum_{s=1}^{K}w_s+\rho_\mu,
			\end{align*}
			where $\rho_\mu$ is as in \eqref{def:v_and_rho} and $w_s=k_{M(\Delta)-M(j_s)}\circ b_{j_s}\circ k_{m(j_s)-m(\Delta)}$. Due to the bound on $|b_y|$ in Lemma \ref{lem:b_Delta}, and $\varrho(\Delta)\leq4$, we have
			\begin{align*}
				|w_s|&\leq k_{M(\Delta)-M(j_s)}\circ |b_{j_s}|\circ k_{m(j_s)-m(\Delta)}\\
				&\leq c_S \frac{\varrho(j_s)^{\alpha-1}}{m(j_s)}|j_s| k_{M(\Delta)-M(j_s)+m(j_s)+m(j_s)-m(\Delta)}\\
				&\leq c_S\frac{4^{\alpha-1}}{m(\Delta)}|j_s| k_{|\Delta|-|j_s|+m(j_s)}=\frac{c_S}{m(\Delta)}|j_s| k_{|\Delta|-|j_s|+m(j_s)}.
			\end{align*}
			Since $|\Delta|-|j_s|+m(j_s)\geq m(j_s)\geq m(\Delta)$ 
			and $\frac{|\Delta|-|j_s|+m(j_s)}{m(\Delta)}\leq\frac{|\Delta|+M(\Delta)}{m(\Delta)}\leq7$, we can use the quotient estimate in Lemma \ref{lem:MartinKernel} to obtain
			\begin{align*}
				|w_s|\leq\frac{c_S}{m(\Delta)}|j_s|\left(\frac{|\Delta|-|j_s|+m(j_s)}{m(\Delta)}\right)^\alpha k_{m(\Delta)}\leq c_S\frac{|j_s|}{m(\Delta)} k_{m(\Delta)}.
			\end{align*}
			Summing up yields the bound
			\begin{align*}
				\sum_{s=1}^{K} |w_s|\leq c_S \frac{|\Delta|}{m(\Delta)}k_{m(\Delta)}\leq c_S 3^\alpha k_{m(\Delta)}=c_Sk_{m(\Delta)}.
			\end{align*}
			In Lemma \ref{lem:rho} we derived a bound on $|\rho_\mu|$, in the current setting $\beta=3$, thus
			\begin{align*}
				|\rho_{\mu}|\leq c_{ S,\lambda}\frac{\varepsilon^2|\Delta|^2}{m(\Delta)^2}k_{m(\Delta)}\leq c_{S,\lambda}3^\alpha\varepsilon^2k_{m(\Delta)}.
			\end{align*} 
			The last step comprises finding a lower bound for $k_{|\Delta|}$ of the form: $k_{|\Delta|}\geq Ck_{m(\Delta)}$, for some constant $C>0$. Since $m(\Delta)\leq|\Delta|$ we use Lemma \ref{lem:harnack_chain} and obtain
			\begin{align*}
				k_{|\Delta|}\geq c \left(\frac{m(\Delta)}{|\Delta|}\right)^\alpha k_{m(\Delta)}\geq\frac{c}{3^\alpha}k_{m(\Delta)}=c_Sk_{m(\Delta)}.
			\end{align*}
			In total we estimate
			\begin{align*}
				\Pi^{\mu}\geq c_Sk_{m(\Delta)}-\varepsilon c_Sk_{m(\Delta)}-\varepsilon^2c_{S,\lambda}k_{m(\Delta)}>0
			\end{align*}
			for fixed $\lambda$ and $\varepsilon>0$ small enough, in particular there exists $\varepsilon(S)>0$, such that for all $\varepsilon<\varepsilon(S)$, the kernel $\omega_\Delta$ is positive. Taking the limit on the left side, the inequality is preserved and thus $\omega_\Delta>0$ if $m(\Delta)\leq|\Delta|\leq3m(\Delta)$.
			Now we apply this statement to prove that $\omega_\Delta>0$ for $|\Delta|\geq m(\Delta)$. Choose $n\in\mathbb{N}_0$ such that $2^{n+1}m(\Delta)\leq M(\Delta)\leq 2^{n+2}m(\Delta)$ and partition $\Delta$ in the following way:
			\begin{align*}
				\Delta=\bigcup_{k=0}^{n-1}\left[2^km(\Delta),2^{k+1}m(\Delta)\right]\cup\left[2^{n}m(\Delta),M(\Delta)\right].
			\end{align*}
			Since $2^{k+1}m(\Delta)-2^km(\Delta)=2^km(\Delta)$ and $2^nm(\Delta)\leq M(\Delta)-2^{n}m(\Delta)\leq 3\cdot2^nm(\Delta)$, all intervals $j$ of the above partition, satisfy $m(j)\leq|j|\leq3m(j)$, hence by the first part of the proof we have positivity of the kernels $\omega_j$. The semi-group property in Lemma \ref{lem:omega-prop} now yields
			\begin{align*}
				\omega_\Delta=\omega_{\left[M(\Delta),2^nm(\Delta)\right]}\circ \omega_{\left[2^{n}m(\Delta),2^{n-1}m(\Delta)\right]}\circ\cdots\circ\omega_{\left[2m(\Delta),m(\Delta)\right]}>0,
			\end{align*}
			since all involved kernels are positive.
		\end{proof}
			From Lemma \eqref{lem:omega-prop} we can infer two immediate consequences
		\begin{corollary}
			\label{cor:omega-bounded}
			The kernel $\omega_\Delta$ is bounded on $S\times S$ for any $\Delta\in\mathrm{segm}^+$.
		\end{corollary}
		\begin{proof}
			If $|\Delta|>m(\Delta)$, $\omega_\Delta$ is bounded as positive and continuous density of a probability distribution. Now consider the case where $|\Delta|\leq m(\Delta)$, i.e $\beta=1$. Then, recalling the definition of $\tilde{\omega}_\Delta$ together with part (3) of Lemma \ref{lem:omega-prop} implies
			\begin{align*}
				|\omega_\Delta|&\leq c_S\varepsilon\frac{|\Delta|^2}{m(\Delta)^2}+|\tilde{\omega}_\Delta|\\
				&\leq c_S\varepsilon\frac{|\Delta|^2}{m(\Delta)^2}+ k_{|\Delta|}+\varepsilon c_S\frac{\varrho(\Delta)^{\alpha-1}}{m(\Delta)}|\Delta|k_{m(\Delta)}\\
				&\stackrel{|\Delta|\leq m(\Delta)}{\leq} c_S\varepsilon+c_S\varepsilon2^{\alpha-1}k_{m(\Delta)} +k_{|\Delta|},
			\end{align*}
			which is bounded since by Theorem \eqref{thm:mart-kern-bdd}, the shifted martin kernels $k_{|\Delta|}, k_{m(\Delta)}$ are bounded on $S\times S$.
		\end{proof}
		\begin{remark}
			\label{rem:omega-L1-norm}
			This pointwise bound becomes useless if $|\Delta|\rightarrow0$, however we obtain a tractable bound on $\sup_{x\in S}\|\omega_\Delta(x,\cdot)\|_{L^1(S)}$: Using the inequality in Corollary \ref{cor:omega-bounded} and property (4) in Lemma \ref{lem:MartinKernel}, we obtain
			\begin{align*}
				\int_S|\omega_\Delta(x,\xi)|\mathrm{d}\omega^{z_0}(\xi)\leq c_S\varepsilon+c_S\varepsilon2^{\alpha-1}+1.
			\end{align*}
		\end{remark}
		\begin{corollary}
			\label{cor:omega-operator-bounded}
			For any $\Delta\in\mathrm{segm}^+$, the operator $\Omega_\Delta: \mathcal{C}(S)\longrightarrow \mathcal{C}(S)$ is continuous, and if $|\Delta|>m(\Delta)$, we have $\|\Omega_\Delta\|\leq 1$.
		\end{corollary}
		\begin{proof}
			Let $f\in C(S)$, then
			\begin{align*}
				|(\Omega_\Delta f)(x)|&\leq\|f\|_{\infty}\int_S|\omega_\Delta(x,\xi)|\mathrm{d}\omega^{z_0}(\xi).
			\end{align*}
			By Remark \ref{rem:omega-L1-norm}, the integral on the right is finite and notably independent of $|\Delta|$. Moreover, if $|\Delta|>m(\Delta)$, $\omega_\Delta$ is positive and thus by its mean $1$ property in Lemma \ref{lem:omega-prop}, the integral evaluates to $1$.
		\end{proof}
		
		\begin{remark}
			\label{rem:semigroup-operators}
			Property (2) in Lemma \ref{lem:omega-prop} and the integrability of $\omega_\Delta$ discussed in remark \eqref{rem:omega-L1-norm}, now imply the analogous semi-group property for the induced operators, i.e. with the same variables as in the above Lemma \ref{lem:omega-prop} and some continuous function $f$ on $S$ vanishing at infinity (thus bounded), we obtain with Fubini
			\begin{align*}
				\left(\Omega_{[b,c]}\left(\Omega_{[a,b]}f\right)\right)(x)&=\int_S \omega_{[b,c]}(x,\xi)\int_S \omega_{[a,b]}(\xi,\zeta)f(\zeta)\mathrm{d}\omega^{z_0}(\zeta)\mathrm{d}\omega^{z_0}(\xi)\\
				&=\int_S\int_S\omega_{[b,c]}(x,\xi)\omega_{[a,b]}(\xi,\zeta)\mathrm{d}\omega^{z_0}(\xi)f(\zeta)\mathrm{d}\omega^{z_0}(\zeta)\\
				&=(\Omega_{[a,c]}f)(x).
			\end{align*}
		\end{remark}
		
		\begin{remark}
			In the case where $\mathcal{O}$ has a $C^2$ boundary, treated in Havin Mozolyako \cite{HavinMozol2016}, we have $\alpha=1$ in Lemma \ref{lem:harnack_chain}. Moreover, instead of the Martin kernel of the domain, they use the Poisson kernel and obtain the inequality
			\begin{align}
				\label{eq:quot_bd_HM}
				\frac{p_{y_2}}{p_{y_1}}\leq c_d\frac{y_2}{y_1} \quad \mathrm{for}\quad y_2\geq y_1.
			\end{align}
			The kernels $\tilde{\omega}_\Delta$  and $b_y$ are analogously defined as $\tilde{\omega}_\Delta=p_{|\Delta|}-\varepsilon b_\Delta$ and $b_y=p_y\circ c_y$ with $c_y=\langle\nabla^1 p_y,\sigma_{2y}\rangle$. We can derive similar estimates for these kernels with inequality \eqref{eq:quot_bd_HM} and Harnack. This yields
			\begin{align*}
				\left|b_\Delta\right|&\leq\int_\Delta|b_y|\mathrm{d}y\\
				&\leq c_S\int_\Delta\frac{p_y}{y}\mathrm{d}y\leq c_S\frac{|\Delta|}{m(\Delta)}p_{m(\Delta)}.
			\end{align*}
			For the case $|\Delta|\leq m(\Delta)$, we can now further estimate $|b_\Delta|\leq c_S p_{|\Delta|}$ using \eqref{eq:quot_bd_HM}. This implies
			\begin{align*}
				\tilde{\omega}_\Delta\geq p_{|\Delta|}-\varepsilon p_{|\Delta|},
			\end{align*}
			which means that $\omega_\Delta$ is positive for small enough $\varepsilon>0$ if $\Delta$ is short in the sense that $|\Delta|\leq m(\Delta)$. This line of arguments however fails for $\alpha>1$, so in the more general case of Lipschitz boundaries, we only know $\omega_\Delta>0$ for long intervals, i.e. $m(\Delta)\leq|\Delta|$.´
		\end{remark}
		Next, the all important "$\Phi$-property" of the integral operator $\Omega_\Delta$
		\subsubsection{The $\Phi$-property}
		\begin{lemma}[$\Phi$-property, \cite{MuellerRiegler2020}]
			\label{lem:phi-prop}
			Let $\varepsilon\in(0,1)$, $y\in(0,1]$ and $\psi:S\longrightarrow\mathbb{R}$ be a function that coincides on $S$ with a positive harmonic function $v$ defined on $\mathcal{O}_{-y}$, Then $K(\psi)=v|_\mathcal{O}$, i.e. the harmonic extension of $\psi$ retrieves $v$ on $\mathcal{O}$. For any $\Delta\in\mathrm{segm}^+$ with $\Delta\subset(0,y]$ we have:
			\begin{align}
				\label{eq:phi_prop_ineq}
				|\Omega_\Delta\psi-\psi|\leq c_S\frac{|\Delta|}{y}\psi.
			\end{align}
		\end{lemma}
		The proof of Lemma \ref{lem:phi-prop} will roughly follow the following scheme: We begin by splitting the operator $\Omega_J$ for some $J\subset\Delta$ with $|J|\leq m(J)$ and derive estimates from above and below for $\Omega_J\psi$ with constants, that still depend on $J$. Then, splitting $\Delta$ into sufficiently small intervals for which we have estimates available and iterating these estimates using the semi-group property of $\Omega_\Delta$, we eradicate the dependence on $\Delta$ in the constants and obtain the desired bounds on $\Omega_\Delta\psi$.
		\begin{proof}
			We begin with arbitrary $J\in\textrm{segm}^+$ such that $J\subset\Delta$. Writing $\omega_J=\tilde{\omega}_J+r_J$ we obtain 
			\begin{align}
				\label{eq:splitting_Phi_prop}
				|(\Omega_J\psi)(x)-\psi(x)|\leq |(K_{|J|}\psi)(x)-\psi(x)|+\varepsilon|(B_J\psi)(x)|+|(R_J\psi)(x)|
			\end{align}
			for all $x\in S$. Since $v$ is positive harmonic on $\mathcal{O}_{-y}$ we apply Harnack and the mean value theorem to estimate the first summand as follows:
			\begin{align*}
				|\left(K_{|J|}\psi\right)(x)-\psi(x)|=|v\left(x_{|J|}\right)-v(x)|&\leq\|\nabla v(x_\eta)\||J|\leq c\frac{v(x_\eta)}{d(x_\eta,S_{-y})}|J|,
			\end{align*}
			where $\eta=\eta(x)\in(0,|J|)$. Now since $x_\eta=x_{(\eta+y)-y}$, we have $d(x_\eta,S_{-y})\geq c_S(\eta+y)\geq c_Sy$ and by Lemma \ref{lem:harnack_chain}, $v(x_\eta)\leq c\left(\frac{y+\eta}{y}\right)^\alpha v(x)\leq c2^\alpha v(x)=c_S\psi(x)$, where $c=c(d)$ only depends on the dimension. Therefore
			\begin{align}
				\label{eq:phi_prop1}
				\left|\left(K_{|J|}\psi\right)(x)-\psi(x)\right|\leq c_{S}\frac{\psi(x)}{y}|J|.
			\end{align}
			Recalling the definition of $c_y$ we obtain for $z\in S$
			\begin{align*}
				|(C_t\psi)(z)|&=\left|\int_S\langle\nabla^1k_t(z,\xi)\psi(\xi),\sigma(z_{2t})\rangle\mathrm{d}\omega^{z_0}(\xi)\right|\\
				&\stackrel{(*)}{=}\left|\left\langle\nabla\int_Sk_t(z,\xi)\psi(\xi)\mathrm{d}\omega^{z_0}(\xi),\sigma(z_{2t})\right\rangle\right|\\
				&\leq\|\nabla v(z_t)\|\stackrel{(**)}{\leq} c_S\frac{v_t(z)}{y} = c_S\frac{(K_t\psi)(z)}{y}
			\end{align*}
			since again $d(z_t,S_{-y})\geq c_Sy$. We exchanged differentiation and integration in $(*)$ due to the integrable bound on $\nabla^1k_t$ provided by Corollary \ref{cor:Harnack_gradient} of the Harnack inequality and $(**)$ was another application of \ref{cor:Harnack_gradient}. Now by definition of the kernel $b_\Delta$ and applying Fubini we estimate $|B_J\psi|$:
			\begin{align}
				\label{eq:phi_prop2}
				\begin{aligned}
					|(B_J\psi)(x)|\leq\int_J K_t\left(|C_t\psi|\right)(x)\mathrm{d}t&\leq \frac{c_S}{y}\int_J\left(K_{2t}\psi\right)(x)\mathrm{d}t\\
					&=\frac{c_S}{y}\int_Jv_{2t}(x)\mathrm{d}t\leq c_{S}\frac{\psi(x)}{y}|J|
				\end{aligned}
			\end{align}
			since $v_{2t}(x)\leq c\left(\frac{2t+y}{y}\right)^\alpha v(x)\leq c3^\alpha v(x)=c_S\psi(x)$ by Lemma \eqref{lem:harnack_chain}. Now we turn to the last summand. To appropriately estimate $R_J\psi$, we assume the extra condition $|J|\leq m(J)$, enabling us to apply the bound (3) from Lemma \ref{lem:omega-prop} with $\beta=1$. With that we obtain
			\begin{align}
				\label{eq:phi_prop3}
				\begin{aligned}
					|(R_J\psi)(x)|&\leq\int_S|\omega_J(x,\xi)-\tilde{\omega}_J(x,\xi)|\psi(\xi)\mathrm{d}\omega^{z_0}(\xi)\\
					&\leq \varepsilon c_S\frac{|J|^2}{m(J)^2}\int_Sk_{m(J)}(x,\xi)\psi(\xi)\mathrm{d}\omega^{z_0}(\xi)\\
					&\stackrel{\varepsilon<1}{\leq} c_S\frac{|J|^2}{m(J)^2}v\left(x_{m(J)}\right)\\
					&\leq \frac{c_S2^\alpha}{m(\Delta)^2}|J|^2\psi(x) 
				\end{aligned}
			\end{align}
			since by Lemma \ref{lem:harnack_chain} $v_{m(J)}(x)\leq c \left(\frac{m(J)+y}{y}\right)^\alpha v(x)\leq c2^\alpha \psi(x)$. Combining inequalities \eqref{eq:phi_prop1}, \eqref{eq:phi_prop2} and \eqref{eq:phi_prop3}, we obtain
			\begin{align*}
				|(\Omega_J\psi)(x)-\psi(x)|&\leq \left(c_S+\frac{c_S|J|y}{m(\Delta)^2}\right)\frac{|J|}{y}\psi(x)\\
				&=: c_S\left(1+\frac{|J|y}{m(\Delta)^2}\right)\frac{|J|}{y}\psi(x),
			\end{align*}
			so after reformulating we end up with
			\begin{align}
				\label{eq:almost_phi_prop}
				\left(1-c_S\left(1+\frac{|J|y}{m(\Delta)^2}\right)\frac{|J|}{y}\right)\psi\leq\Omega_J\psi\leq\left(1+c_S\left(1+\frac{|J|y}{m(\Delta)^2}\right)\frac{|J|}{y}\right)\psi
			\end{align}
			everywhere on $S$. Having established this inequality we partition $\Delta$ into small intervals and apply it to each element of the partition. Choose $K:= K(\Delta,y)\in\mathbb{N}$ such that the three conditions
			\begin{align*}
				\frac{|\Delta|}{K}\leq m(\Delta) && \frac{|\Delta|y}{Km(\Delta)^2}\leq1 && 2c_S\frac{|\Delta|}{Ky}\leq\frac{1}{2}
			\end{align*}
			are satisfied. Now divide $\Delta$ into $K$ non-overlapping intervals of equal length $\frac{|\Delta|}{K}$ and denote them by $J_1,\dots,J_K$, i.e. $\Delta=\bigcup_{1\leq k\leq K} J_k$, $m(J_k)<m(J_{k+1})$ and $|J_k|=\frac{|\Delta|}{K}$. Moreover, we have $c_S\left(1+\frac{|J_k|y}{m(\Delta)^2}\right)\leq 2c_S$ and $|J_k|\leq m(J_k)$, by the choice of $K$. Remark \eqref{rem:semigroup-operators} and the above estimates for $J\in\mathrm{segm}^+$ with $|J|\leq m(J)$ now imply
			\begin{align*}
				\Omega_\Delta\psi&=(\Omega_{J_K}\Omega_{J_{K-1}}\cdots\Omega_{J_1})\psi\\
				&\leq \left(1+\frac{2c_S|\Delta|}{Ky}\right)^K\psi\stackrel{(*)}{\leq} \exp\left(\frac{2c_S|\Delta|}{y}\right)\psi\stackrel{(**)}{\leq}\left(1+e^{2c_s}\frac{|\Delta|}{y}\right)\psi\stepcounter{equation}\tag{\theequation}\label{eq:phi_prop_upper_ineq},
			\end{align*}
			where we used $\ln(1+\zeta)\leq \zeta$ for $\zeta>-1$ in $(*)$ and the fact that for $z=2c_S|\Delta|/y$, we have 
			\begin{align*}
				\left(1+\frac{z}{K}\right)^K\leq e^z \quad\Longleftrightarrow\quad \ln\left(1+\frac{z}{K}\right)\leq\frac{z}{K}.
			\end{align*}
			For $(**)$ observe that $\xi=|\Delta|/y\leq1$ and combine with the following consideration for $c=2c_S\geq 0$:
			\begin{align*}´						e^{c\xi}-1=\sum_{k=1}^{\infty}\frac{(c\xi)^k}{k!}=\xi\sum_{k=1}^{\infty}\frac{c^k\xi^{k-1}}{k!}\leq e^c\xi.
			\end{align*}
			Completely analogously we show the other inequality
			\begin{align*}
				\Omega_\Delta\psi&=(\Omega_{J_K}\Omega_{J_{K-1}}\cdots\Omega_{J_1})\psi\\
				&\geq \left(1-\frac{2c_S|\Delta|}{Ky}\right)^K\psi\stackrel{(\star)}{\geq}\exp\left(-\frac{4c_S|\Delta|}{y}\right)\psi\stackrel{(\star\star)}{\geq}\left(1-\sinh(4c_s)\frac{|\Delta|}{y}\right)\psi\stepcounter{equation}\tag{\theequation}\label{eq:phi_prop_lower_ineq}.
			\end{align*}
			Inequality $(\star)$ stems from the fact that $\ln(1-\zeta)\geq-\frac{\zeta}{1-\zeta}$ for $\zeta<1$ i.e. with $z$ as above we have 
			\begin{align*}
				K\ln\left(1-\frac{z}{K}\right)\geq-\frac{z}{1-z/K}\Longleftrightarrow\left(1-\frac{z}{K}\right)^K\geq\exp\left(-\frac{z}{1-z/K}\right).
			\end{align*}
			Since by assumption $z/K<1/2$, we obtain $\left(1-\frac{z}{K}\right)^K\geq e^{-2z}$.
			In $(\star\star)$ we again have $\xi=|\Delta|/y\leq1$, moreover, for $c=4c_S\geq0$,
			\begin{align*}
				\frac{e^{-c\xi}-1}{z}=\sum_{k=1}^{\infty}\frac{(-1)^kc^k\xi^{k-1}}{k!}\geq-\sum_{n=0}^{\infty}\frac{c^{2n+1}\xi^{2n}}{(2n+1)!}\geq-\sinh(c)
			\end{align*}
			holds.
		\end{proof}
		
		The statement, that the action $\Omega_\Delta$ on the trace of any positive harmonic function on the shifted domain $S_y$ is under good control, in the sense of \eqref{eq:phi_prop_ineq}, is heavily relied upon in the remaining arguments. Moreover, if $|\Delta|\rightarrow0$, $\Omega_\Delta\psi\rightarrow\psi$ uniformly on $S$, for $\psi$ as in the setting of Lemma \ref{lem:phi-prop}.\par 
		Sometimes we will use the improved estimates \eqref{eq:phi_prop_upper_ineq} and \eqref{eq:phi_prop_lower_ineq} to estimate $\Omega_\Delta$, i.e.
		\begin{align}
			\exp\left(-\frac{4c_S|\Delta|}{y}\right)\psi\leq\Omega_\Delta\psi\leq\exp\left(\frac{2c_S|\Delta|}{y}\right)\psi,
		\end{align}
		where $\Delta,\psi$ and $y$ are as in Lemma \ref{lem:phi-prop}. The setting in Lemma \ref{lem:phi-prop} is such that $\psi$ is the trace on $S$ of a positive harmonic function  $v$ defined on $\mathcal{O}_{-y}$. The $\Phi$-property will for example be applied to situations where $\phi=\phi\in C_0(S)$ and positive and $\psi(x)=(K_y\phi)(x)$ for $x\in S$. Then $\psi$ is the trace on $S$ of the positive harmonic function $K_\tau\psi$ defined on $\mathcal{O}_{-y}$ and for $\tau<0$, $K\tau\psi=K_{y-\tau}\phi$. \\
		We can derive a slightly weaker form of the $\Phi$-property:
		\begin{remark}
			\label{rem:Phi2}
			Let $\varepsilon\in(0,1)$ and $\psi:S\longrightarrow\mathbb{R}$ be bounded and continuous. Then for $\Delta\in\mathrm{segm}^+$ with $|\Delta|\leq m(\Delta)$ we have
			\begin{align}
				\label{eq:Phi2}
				|\psi-\Omega_\Delta\psi|\leq|\psi-K_{|\Delta|}\psi|+c_S\frac{|\Delta|}{m(\Delta)}\|\psi\|_{C(S)}
			\end{align}
		everywhere on $S$. In particular, Theorem \ref{thm:ex_harm_ext} implies that $\Omega_\Delta\psi$ converges to $\psi$ pointwise, as $|\Delta|\rightarrow 0$. 
		
		The proof of inequality \eqref{eq:Phi2} starts with splitting as in \eqref{eq:splitting_Phi_prop}, estimating $|R_\Delta\psi|$ as in \eqref{eq:phi_prop3} and combining with the inequality
		\begin{align*}
			|B_\Delta\psi|\leq\int_\Delta \left|B_\theta|\psi|\right|\mathrm{d}\theta\leq\int_\Delta\frac{K_\theta|\psi|}{\theta}\mathrm{d}\theta\leq\frac{|\Delta|}{m(\Delta)}\|\psi\|_{C(S)}.
		\end{align*}
		\end{remark}
		
		We continue to derive some immediate consequences of the $\Phi$-property. \par

		\begin{remark}[\textbf{Notation}]
			For $y\in(0,1)$, we now introduce the notational convention
			\begin{align*}
				\omega_y:=\omega_{[y,1]} \quad\text{and respectively}\quad\Omega_y:=\Omega_{[y,1]},
			\end{align*}
			which we shall henceforth abide by.
		\end{remark}
		
		An immediate and crucially important consequence of the $\Phi$-property is stated in the following corollary.
		\begin{corollary}[\cite{MuellerRiegler2020}]
			\label{cor:very_important_ineq}
			Let $\varphi:\mathcal{O}\longrightarrow\mathbb{R}$ be positive harmonic on $\mathcal{O}$. Then for $0<\eta<y\leq1/2$ the following estimate holds:
			\begin{align}
				\Omega_\eta\varphi_y\leq c_S\Omega_y\varphi_y
			\end{align}
			where $\varphi_y:=\varphi|_{S_y}$
		\end{corollary}
		\begin{proof}
			Let $c_S$ be the constant from the previous Lemma \ref{lem:phi-prop},  and define $\Delta:=[\eta,y]$. By the $\Phi$-property in Lemma \ref{lem:phi-prop}, we now estimate
			\begin{align*}
				|\Omega_\Delta\varphi_y-\varphi_y|\leq c_S\frac{y-\eta}{y}\varphi_y
			\end{align*} 
			and thus
			\begin{align*}
				\Omega_\Delta\varphi_y\leq \left(1+c_S\right)\varphi_y.
			\end{align*}
			Applying this inequality and noting that $\omega_y>0$, we immediately obtain for all $x\in S$
			\begin{align*}
				(\Omega_\eta\varphi_y)(x)=(\Omega_y\Omega_\Delta\varphi_y)(x)\leq(1+c_S)(\Omega_y\varphi_y)(x).
			\end{align*}
		\end{proof}
		Another property of $\omega_y$ that holds for sufficiently small $y$ is listed below and will be used later when characterizing the dual operator of $\Omega_y$. The structure of the proof is very similar to that of Lemma \ref{lem:positivity_omega}.
		\begin{lemma}[\cite{MuellerRiegler2020}]
			\label{lem:omega_ky_comparison}
			There exist constants $\varepsilon(S),c_S^-,c_S^+>0$ such that for all $y\in(0,1/4)$, and $\varepsilon\in(0,\varepsilon(S))$ the double estimate
			\begin{align}
				\label{eq:omega_y_k_y_comparison}
				y^{c_S^-\varepsilon }k_{1-y}\leq\omega_y\leq\frac{1}{y^{c_S^+\varepsilon }}k_{1-y}
			\end{align}
			holds.
		\end{lemma}
		\begin{proof}
			Let $\rho>0$, as a first step we derive bounds for $\omega_{[\rho,2\rho]}$. Write
			\begin{align*}
				\omega_{[\rho,2\rho]}=k_{\rho}-\varepsilon b_{[\rho,2\rho]} + r_{[\rho,2\rho]}
			\end{align*}
			and apply available bounds for $|b_{[\rho,2\rho]}|$ and $\left|r_{[\rho,2\rho]}\right|$ provided in Lemmata \ref{lem:b_Delta} and \ref{lem:omega-prop} (3) to obtain
			\begin{align}
				\label{eq:omega_ky_comparison1}
				\omega_{[\rho,2\rho]}&\leq k_{\rho}+c_S\varepsilon2^{\alpha-1}k_{\rho}+c_S\varepsilon k_{\rho}\leq (1+c_S\varepsilon)k_\rho
			\end{align}
			and similarily
			\begin{align}
				\label{eq:omega_ky_comparison2}
				\omega_{[\rho,2\rho]}\geq (1-c_S\varepsilon)k_\rho.
			\end{align}
			Let $K=K(y)\in\mathbb{N}, K>0$ be such that $2^{K+1}y\leq1\leq2^{K+2}y$ is satisfied and partition ${[y,1]}$ in the following manner:
			\begin{align*}
				[y,1]=\bigcup_{k=0}^{K-1}\left[2^ky,2^{k+1}y\right] \cup \left[2^Ky,1\right].
			\end{align*}
			Moreover, by the  semi-group property we have
			\begin{align}
				\label{eq:omega_ky_comparison3}
				\omega_y= \omega_{\left[2^Ky,1\right]}\circ\omega_{\left[2^{K-1}y,2^Ky\right]}\circ\cdots\circ\omega_{\left[y,2y\right]}.
			\end{align}
			The considerations above deliver a bound on each $\omega_{\left[2^ky,2^{k+1}y\right]}$. We now estimate $\omega_{\Delta}$, with $\Delta=\left[2^Ky,1\right]$ in a similar way. By the choice of $K$, we have $m(\Delta)\leq|\Delta|\leq3m(\Delta)$. Lemma \ref{lem:harnack_chain} yields the inequality
			\begin{align*}
				\frac{k_{|\Delta|}}{k_{m(\Delta)}}\geq c_d\left(\frac{m(\Delta)}{|\Delta|}\right)^\alpha\quad\text{thus}\quad k_{m(\Delta)}\leq\frac{1}{c_d}\left(\frac{|\Delta|}{m(\Delta)}\right)^\alpha k_{|\Delta|}\leq c_S k_{|\Delta|}
			\end{align*}
			where $c_d>0$ is a constant that only depends on the dimension $d$. Combining the bound above with estimates given in  Lemmata \ref{lem:b_Delta} and \ref{lem:omega-prop} (3) we obtain
			\begin{align*}
				\omega_\Delta&\leq k_{|\Delta|}+\varepsilon\left|b_\Delta\right|+\left|r_\Delta\right|\\
				&\leq k_{|\Delta|}+c_S\varepsilon\frac{\varrho(\Delta)^{\alpha-1}}{m(\Delta)}|\Delta|k_{m(\Delta)}+c_S\varepsilon\frac{|\Delta|^2}{m(\Delta)^2}k_{m(\Delta)}\\
				&\leq k_{|\Delta|}+c_S\varepsilon2^\alpha k_{m(\Delta)}+c_S\varepsilon k_{m(\Delta)}\leq (1+c_S\varepsilon)k_{|\Delta|}
			\end{align*}
			and completely analogously
			\begin{align*}
				\omega_\Delta&\geq k_{|\Delta|}-\varepsilon\left|b_\Delta\right|-\left|r_\Delta\right|\\
				&\geq k_{|\Delta|}-c_S\varepsilon\frac{\varrho(\Delta)^{\alpha-1}}{m(\Delta)}|\Delta|k_{m(\Delta)}-c_S\varepsilon\frac{|\Delta|^2}{m(\Delta)^2}k_{m(\Delta)}\\
				&\geq k_{|\Delta|}-c_S\varepsilon2^\alpha k_{m(\Delta)}-c_S\varepsilon k_{m(\Delta)}\geq(1-c_S\varepsilon)k_{|\Delta|}.
			\end{align*}
			Combining this with equations \eqref{eq:omega_ky_comparison1}, \eqref{eq:omega_ky_comparison2}, the kernel representation in \eqref{eq:omega_ky_comparison3} and the fact that
			\begin{align*}
				1-2^Ky + \sum_{k=0}^{K-1}2^ky=1-y
			\end{align*}
			yields
			\begin{align*}
				(1-c_S\varepsilon)^{K+1}k_{1-y}\leq\omega_y\leq(1+c_S\varepsilon)^{K+1}k_{1-y}.
			\end{align*}
			Now define $\varepsilon(S):=\frac{1}{2c_S}$ and let $\varepsilon<\varepsilon(S)$. Since $(K+1)\ln(2)\leq \ln(1/y)$ we obtain
			\begin{align*}
				(K+1)\ln(1+\varepsilon c_S)\leq (K+1)c_S\varepsilon\leq\ln\left(\frac{1}{y}\right)\frac{c_S}{\ln(2)}\varepsilon
			\end{align*}	
			and thus $\omega_y\leq\frac{1}{y^{c_S^+\varepsilon}}k_{1-y}$ with $c_S^+:=c_S/\ln(2)$. Similarly we have
			\begin{align*}
				(K+1)\ln(1-c_S\varepsilon)\geq-(K+1)\frac{c_S\varepsilon}{1-c_S\varepsilon}\geq-\ln\left(\frac{1}{y}\right)\frac{2c_S}{\ln(2)}\varepsilon
			\end{align*}
			and hence $\omega_y\geq y^{c_S^-\varepsilon}k_{1-y}$, with $c_S^-:=2c_S/\ln(2)$.
		\end{proof}
		\subsubsection{Differential equation}
		The following theorem will arguably be one of the most important ingredient in the arguments leading up to proving the main theorem for near half spaces \ref{thm:main}.
		\begin{theorem}[Differential equation \cite{MuellerRiegler2020}]
			\label{thm:diff-eq}
			Let $\varphi$ be a positive harmonic function on $\mathcal{O}$ that is bounded on $\mathcal{O}_y$ for any $y>0$ and vanishes at infinity on all positive shifts of $S$, i.e.
			\begin{align*}
				\forall y>0:\quad \lim_{\substack{\|z\|\rightarrow\infty\\z\in S}}\varphi(z_y)=0.
			\end{align*}
			For every $x\in S$, define the real-valued function
			\begin{align*}
				f^x:(0,1]&\longrightarrow \mathbb{R}\\
				&y\longmapsto(\Omega_y(\varphi_y))(x).
			\end{align*}
			Then for any $y_0\in(0,1)$, $f^x$ is Lipschitz on $[y_0,1]$ and we have 
			\begin{align}
				\label{eq:HS-fx}
				f^x(y_2)-f^x(y_1)=\int_{y_1}^{y_2}\varepsilon(\Omega_{y}(B_{y}\varphi_{y}))(x)\mathrm{d}y
			\end{align}
			for all $y_1,y_2\in(y_0,1]$.
		\end{theorem}
		\begin{proof}
			We begin by demonstrating the Lipschitz property of $f^x$. By Rademacher's Theorem $f^x$ is then a.e. differentiable and the principle theorem of integration holds, i.e.
			\begin{align*}
				f^x(y_2)-f^x(y_1)=\int_{y_1}^{y_2}(f^x)'(y)\mathrm{d}y
			\end{align*}
			where $(f^x)'$ denotes the a.e. derivative of $f^x$. Then we will compute the left derivative of $(f^x)'_-(y)=(\Omega_{y}(B_{y}\varphi_{y}))(x)$ which exists everywhere on $(0,1]$ and is continuous. Since the a.e. derivative $(f^x)'$ and $(f^x)'_-$ coincide except on a set of measure $0$, we have established the desired equality \eqref{eq:HS-fx}.\par 
			Fix $x\in S$, $y_0\in(0,1]$ and for $y\in[y_0,1]$, $h>0$ satisfying $y>y-h\geq y_0$, define $\Delta:=[y-h,y]$. We obtain by definition of $f^x$ and remark \eqref{rem:semigroup-operators}:
			\begin{align}
				\label{eq:splitting_diff}
				\begin{aligned}
					f^x(y)-&f^x(y-h)=\\
					&(\Omega_{y}(\varphi_{y}-\Omega_{\Delta}\varphi_{y}+\varphi_{y}-\varphi_{y-h}))(x)+((\Omega_{y-h}-\Omega_{y})(\varphi_{y}-\varphi_{y-h}))(x).
				\end{aligned}
			\end{align}
			Setting $K:=\sup_{y_0\leq y\leq1}\|\varphi_y\|_{C(S)}$, we observe the following. Since $\Delta=[y-h,y]\subset(0,y]$, the $\Phi$-property in Lemma \ref{lem:phi-prop} implies
			\begin{align*}
				\forall z\in S:\quad|\varphi_{y}(z)-(\Omega_{\Delta}\varphi_{y})(z)|\leq c_S\frac{h}{y}\varphi_{y}(z)\leq \frac{c_SK}{y_0}h.
			\end{align*}
			Furthermore, by the mean value theorem we obtain for every $z\in S$ a $\theta=\theta(z)\in\Delta^\circ$, such that for any $z\in S$, we have
			\begin{align*}
				|\varphi_{y}(z)-\varphi_{y-h}(z)|=\left|\langle\nabla \varphi(z_{\theta}),\vec{e_d}\rangle\right|h\stackrel{(*)}{\leq}c_d\frac{\varphi(z_\theta)}{\theta}h\leq c_d\frac{K}{y_0}h,
			\end{align*}
			where we used Harnack's inequality in $(*)$, and $c_d$ only depends on the dimension $d$. By Corollary \ref{cor:omega-operator-bounded} we also obtain the bound
			\begin{align*}
				\left\|\Omega_{y}(\varphi_{y}-\Omega_{\Delta}\varphi_{y}\right.&\left.+\varphi_{y}-\varphi_{y-h})\right\|_{C(S)}
				\\ 
				&\leq C\left(\|\varphi_{y}-\Omega_{\Delta}\varphi_{y}\|_{C(S)}+\|\varphi_{y}-\varphi_{y-h}\|_{C(S)}\right)\\
				&\leq C\frac{K(c_d+c_S)}{y_0}h,
			\end{align*}
			where $C$ is given in Remark \ref{rem:omega-L1-norm} and notably  only depends on $S$ and on $\varepsilon$. For the second term in equation \eqref{eq:splitting_diff}, we apply the operator inequality again to obtain
			\begin{align*}
				\left\|(\Omega_{y-h}-\Omega_y)(\varphi_y-\varphi_{y-h})\right\|_{C(S)}\leq 2C \frac{c_dK}{y_0}h.
			\end{align*}
			Combining estimates, we now obtain
			\begin{align*}
				|f^x(y)-f^x(y-h)|\leq CK\frac{3c_d+c_S}{y_0}h.
			\end{align*}
			Hence $f^x$ is Lipschitz on $[y_0,1]$.
			
			We now turn to computing the left derivative of $f^x$. Using the same idea  as in \eqref{eq:splitting_diff}, we have for any $x\in S$, $y\in(0,1)$ and $h<y$,
			\begin{align*}
				\frac{f^x(y-h)-f^x(y)}{-h}=\Omega_y\left(\frac{\varphi_y-\Omega_\Delta\varphi_y+\varphi_y-\varphi_{y-h}}{h}\right)(x)+(\Omega_{y-h}-\Omega_y)\left(\frac{\varphi_y-\varphi_{y-h}}{h}\right)(x)
			\end{align*}
			and thus
			\begin{align*}
					\left|\frac{f^x(y-h)-f^x(y)}{-h}-\varepsilon\Omega_y(B_y\varphi_y)(x)\right|&\leq 
					\left|\Omega_y\left(\frac{\varphi_y-\Omega_\Delta\varphi_y+\varphi_y-\varphi_{y-h}}{h}-\varepsilon B_y\varphi_y\right)(x)\right|\\
					&+\left|(\Omega_{y-h}-\Omega_y)\left(\frac{\varphi_y-\varphi_{y-h}}{h}\right)(x)\right|.
			\end{align*}
			In a first step we treat the last term and show
			\begin{align}
				\label{eq:diff_quot_sec_part}
				\lim_{h\downarrow0}(\Omega_{y-h}-\Omega_y)\left(\frac{\varphi_y-\varphi_{y-h}}{h}\right)=0.
			\end{align}
			We rewrite (and drop the notational dependence on $x\in S$)
			\begin{align*}
				(\Omega_{y-h}-\Omega_y)\left(\frac{\varphi_y-\varphi_{y-h}}{h}\right)=(\Omega_y\Omega_\Delta-\Omega_y)\left(\frac{\partial\varphi_y}{\partial\vec{e_d}}+\frac{\varphi_y-\varphi_{y-h}}{h}-\frac{\partial\varphi_y}{\partial\vec{e_d}}\right).
			\end{align*}
			By Harnack we have for all $z\in S$:
			\begin{align*}
				\left|\frac{\partial\varphi}{\partial\vec{e_d}}(z_y)\right|\leq c \frac{\varphi_y(z)}{y}
			\end{align*}
			and thus $\lim_{|z|\rightarrow\infty, z\in S} \frac{\partial\varphi}{\partial\vec{e_d}}(z_y)=0$, i.e. $\frac{\partial\varphi_y}{\partial\vec{e_d}}\in C_0(S)$. For small enough $h>0$ (in particular such that $|\Delta|\leq m(\Delta)$), we can apply Remark \ref{rem:Phi2} to obtain 
			\begin{align*}
				(\Omega_y\Omega_\Delta-\Omega_y)\left(\frac{\partial\varphi_y}{\partial\vec{e_d}}\right)(x)=\int_S\omega_y(x,z)\left(\Omega_\Delta\frac{\partial\varphi_y}{\partial\vec{e_d}}(z)-\frac{\partial\varphi_y}{\partial\vec{e_d}}(z)\right)\mathrm{d}\omega^{z_0}(z)\stackrel{h\rightarrow0}{\longrightarrow} 0,
			\end{align*}
			since by Remark \ref{rem:maximum_principle}, $
			c_S\omega_y(x,z)\|\psi\|_{C(S)}$ is an integrable majorant.
			
			We now estimate the norm of
			\begin{align*}
				\frac{\varphi_y-\varphi_{y-h}}{h}-\frac{\partial\varphi_y}{\partial\vec{e_d}}.
			\end{align*}
			Boundedness of $\varphi$ on $\mathcal{O}_{y_0}$ allows us to uniformly bound all second derivatives on $\overline{\mathcal{O}_{3y/4}}=:\Omega$. Consider $z\in\Omega$ and let $r>0$ be such that $B(z,r)\subseteq \mathcal{O}_{y/2}$. Applying Remark \ref{rem:grad_harm_fctn} to $\frac{\partial}{\partial_{x_j}}\varphi$ on $B(z,r/2)$, we obtain 
			\begin{align*}
				\left\|\nabla\frac{\partial}{\partial_{x_j}}\varphi(z)\right\|\leq \frac{2n}{r}\sup_{\zeta\in\partial B(z,r/2)} \left|\frac{\partial}{\partial_{x_j}}\varphi(\zeta)\right| = \frac{2n}{r}\left|\frac{\partial}{\partial_{x_j}}\varphi(\xi)\right|
			\end{align*}
			for some $\xi\in\partial B(z,r/2)$. We apply Remark \ref{rem:grad_harm_fctn} again, this time to $\varphi$ on $B(\xi,r/2)\subseteq B(x,r)$ and obtain
			\begin{align*}
				\|\nabla \varphi(\xi)\|\leq \frac{2n}{r}\sup_{\zeta\in \partial B(\xi,r/2)} \varphi(\zeta)\leq \frac{2n}{r} \sup_{\zeta\in \mathcal{O}_{y/2}} \varphi(\zeta),
			\end{align*}
			so in total we have
			\begin{align*}
				\left|\frac{\partial^2}{\partial x_i\partial x_j}\varphi(z)\right|\leq\left(\frac{2n}{r}\right)^2\sup_{\zeta\in \mathcal{O}_{y/2}} \varphi(\zeta)\leq \left(\frac{8n}{c_S y}\right)^2\sup_{\zeta\in \mathcal{O}_{y/2}} \varphi(\zeta),
			\end{align*}
			since $r>0$ was arbitrary, $B(z,\mathrm{dist}(z,\partial \mathcal{O}_{y/2}))\subseteq \mathcal{O}_{y/2}$ and $\mathrm{dist}(z,\partial \mathcal{O}_{y/2})\geq c_Sy/4$. Finally, this leads to the desired estimate
			\begin{align*}
				\sup_{\Omega} \left|\frac{\partial^2}{\partial x_i\partial x_j}u\right|\leq\left(\frac{8n}{c_S y}\right)^2\sup_{\mathcal{O}_{y/2}} u.
			\end{align*}
			For any $z\in S$, Taylor expansion yields
			\begin{align*}
				\varphi_y(z)-\varphi_{y-h}(z)=h\frac{\partial\varphi}{\partial\vec{e_d}}(z_y)-\frac{h^2}{2}\frac{\partial^2\varphi}{\partial\vec{e_d}^2}(z_\theta),
			\end{align*}
			where $\theta=\theta(z)\in(y-h,y)$, i.e. $z_\theta\in\Omega$ if $h<y/4$. This implies
			\begin{align*}
				\left|\frac{\varphi(z_y)-\varphi(z_{y-h})}{h}-\frac{\partial\varphi}{\partial\vec{e_d}}(z_y)\right|\leq\frac{h}{2}\left(\frac{8n}{c_S y}\right)^2\sup_{\mathcal{O}_{y/2}} u.
			\end{align*}
			Therefore
			\begin{align*}
				\left\|(\Omega_y\Omega_\Delta-\Omega_y)\left(\frac{\varphi_y-\varphi_{y-h}}{h}-\frac{\partial\varphi_y}{\partial\vec{e_d}}\right)\right\|_{C(S)}\leq h \left(\frac{8n}{c_S y}\right)^2\sup_{\mathcal{O}_{y/2}} u
			\end{align*}
			and thus establishing \eqref{eq:diff_quot_sec_part}.
			
			In the last step of this proof we  write $\omega_\Delta=\tilde{\omega}_\Delta+r_\Delta = k_h-\varepsilon b_\Delta+r_\Delta$ and since $K_h\varphi_y=\varphi_{y+h}$, we obtain for every $z\in S$
			\begin{align}
				\label{eq:diffeq1}
				\begin{aligned}
				\left|\frac{\varphi_y(z)-(\Omega_\Delta\varphi_y)(z)+\varphi_y(z)-\varphi_{y-h}(z)}{h}-\varepsilon (B_y\varphi_y)(z)\right|\leq\\ 
				\left\|\frac{\varphi_y-\varphi_{y+h}+\varphi_y-\varphi_{y-h}}{h}\right\|_{C(S)}+\frac{1}{h}\left\|R_\Delta(\varphi_y)\right\|_{C(S)}
				+\varepsilon\left|\frac{(B_\Delta\varphi_y)(z)}{h}-(B_y\varphi_y)(z)\right|,
				\end{aligned}
			\end{align}
			where in the first two summands we immediately pass to the norm. We now treat each of the summands on the right separately.
			Again, for any $z\in S$, Taylor expansion yields 
			\begin{align*}
				\frac{\varphi(z_{y+h})-\varphi(z_y)-(\varphi(z_y)-\varphi(z_{y-h}))}{h} = \frac{h}{2}\left(\frac{\partial^2}{\partial\vec{e_d}^2}\varphi(z_{\theta_1})+\frac{\partial^2}{\partial\vec{e_d}^2}\varphi(z_{\theta_2})\right),
			\end{align*}
			where $\theta_1\in(y,y+h)$ and $\theta_2\in(y-h,y)$, i.e.  $z_{\theta_i}\in\Omega$, if $h<y/4$. Since all second partial derivatives are bounded on $\Omega$, the first summand in \eqref{eq:diffeq1} converges to $0$ as $h\rightarrow0$. For the next summand, recall the inequality for $|r_\Delta|$ in Lemma \ref{lem:omega-prop}, establishing the following estimate:
			\begin{align*}
				|(R_\Delta\varphi_y)(z)| &\leq \int_S|\omega_\Delta(z,\xi)-\tilde{\omega}_\Delta(z,\xi)|\varphi(\xi)\mathrm{d}\omega^{z_0}(\xi)\\
				&\leq c_S\varepsilon\frac{h^2}{(y-h)^2}\int_Sk_{y-h}(z,\xi)\varphi(\xi)\mathrm{d}\omega^{z_0}(\xi),
			\end{align*}
			since $|\Delta|=h\leq y-h=m(\Delta)$.
			Another application of the quotient bound in Lemma \ref{lem:MartinKernel}, yields that also this term, converges to $0$. For the last term observe that with Fubini we obtain:
			\begin{align*}
				\left|\frac{1}{h}(B_\Delta\varphi_y)(z)-(B_y\varphi_y)(z)\right|\leq\frac{1}{h}\int_{y-h}^{h}|(B_\theta\varphi_y)(z)-(B_y\varphi_y)(z)|\mathrm{d}\theta,
			\end{align*}
			which converges to $0$ for every $z\in S$ since $\theta\mapsto B_\theta\varphi_y(z)$ is continuous by Remark \ref{rem:operator_height_continuity}. Furthermore, writing  $\|\varphi_y\|_{S,\infty}=: D$, we observe
			\begin{align*}
				\frac{1}{h}\int_{y-h}^{y}|(B_\theta\varphi_y)(z)-(B_y\varphi_y)(z)|\mathrm{d}\theta&\leq \frac{D}{h}\int_{y-h}^{y}\int_S|b_\theta(z,\xi)-b_y(z,\xi)|\mathrm{d}\omega^{z_0}(\xi)\mathrm{d}\theta\\
				&\leq\frac{c_SD}{h}\int_{y-h}^{y}\int_S\left(\frac{k_\theta(z,\xi)}{\theta}+\frac{k_y(z,\xi)}{y}\right)\mathrm{d}\omega^{z_0}(\xi)\mathrm{d}\theta\\
				&\leq \frac{3c_SD}{y}
			\end{align*}
			for $h<y/2$, using Lemma \ref{lem:properties_b_y}. Now by dominated convergence we have
			\begin{align*}
				\left|\Omega_y\left(\frac{1}{h}B_\Delta\varphi_y-B_y\varphi_y\right)(x)\right| &\leq \int_S\omega_y(x,z)\left|\frac{1}{h}(B_\Delta\varphi_y)(z)-(B_y\varphi_y)(z)\right|\mathrm{d}\omega^{z_0}(z)\stackrel{h\rightarrow0}{\longrightarrow}0,
			\end{align*}
			since $\Omega_y(1)=1$ and the integrand satisfies the estimate above. Applying the operator inequality to the other terms in equation \eqref{eq:diffeq1}, we finally obtain
			\begin{align*}
				\left|\Omega_y\left(\frac{\varphi_y-\Omega_\Delta\varphi_y+\varphi_y-\varphi_{y-h}}{h}-\varepsilon B_y\varphi_y\right)(x)\right|\stackrel{h\rightarrow0}{\longrightarrow}0,
			\end{align*}
			establishing the end of the proof.
		\end{proof}
		\begin{remark}[\cite{Gilbarg2001}(p.22-23)]
			\label{rem:grad_harm_fctn}
			Let $f$ be harmonic on a domain $\Omega\subseteq\mathbb{R}^n$ and $x\in\Omega$, $r>0$ such that $B=B(x,r)\subseteq\Omega$. By the mean value theorem and divergence theorem, we have
			\begin{align*}
				\nabla f(x) \stackrel{(*)}{=} \frac{1}{r^n\omega_n}\int_{B}\nabla f(z)\mathrm{d}z \stackrel{(**)}{=} \frac{1}{r^n\omega_n}\int_{\partial B} f\vec{n}\mathrm{d}s,
			\end{align*}
			where $\vec{n}$ denotes the unit outward normal to $\partial B$, $\omega_n$ denotes the $n$-dimensional volume of the unit sphere and $\mathrm{d}s$ indicates the $(n-1)$-dimensional surface measure.
			Equality $(*)$ holds since all partial derivatives of $f$ are again harmonic and equality $(**)$ holds since for every $v\in\mathbb{R}^n$ we have 
			\begin{align*}
				\left\langle v,\int_{B}\nabla f(z)\mathrm{d}z\right\rangle &= \int_{B}\langle v,\nabla f(z)\rangle\mathrm{d}z\\
				&=\int_B \mathrm{div}(vf(z))\mathrm{d}z\\
				&=\int_{\partial B}	\langle vf(z),\vec{n}\rangle \mathrm{d}s=\left\langle v,\int_{\partial B}f(z)\vec{n}\mathrm{d}s\right\rangle.			
			\end{align*}
			Hence in total, we obtain 
			\begin{align*}
				\|\nabla f(x)\|\leq \frac{1}{r^n\omega_n}\int_{\partial B}\mathrm{d}s\sup_{z\in\partial B} |f(z)|=\frac{n}{r}\sup_{z\in\partial B} |f(z)|.
			\end{align*}
		\end{remark}
		\subsection{Measures and Duality}
		We begin this chapter by observing that the operator $\Omega_y$ maps functions from $C_0(S)$ to $C_0(S)$ for any $y>0$ small enough. Beforehand, we check, that the harmonic extension of a continuous and vanishing boundary function again vanishes at infinity on any positive shift of the boundary $S_y$.
		\begin{lemma}
			\label{lem:ImageK_y}
			For $y>0$, $K_y\left(C_0(S)\right)\subseteq C_0(S)$.
		\end{lemma}
		\begin{proof}
			Let $\varphi\in C_0(S)$ and fix $\epsilon>0$. Then there exists a compact set $\mathcal{K}_\epsilon\subseteq S$, such that $\sup_{x\in S\setminus \mathcal{K}_\epsilon}|\varphi(x)|<\epsilon/2$, furthermore write $M:=\max_{x\in \mathcal{K}_\epsilon}|\varphi(x)|$. Now let $x\in S$, satisfying $d(x,\mathcal{K}_\epsilon)>2^ny$, i.e. $\mathcal{K}_\varepsilon \subseteq S\setminus B(x,2^ny)$. Furthermore, let $c,\alpha>0$ be the constants from Lemma \ref{lem:MartinKernelDecay} and choose $n\in\mathbb{N}$ such that 
			\begin{align*}
				M\frac{c2^{-n\alpha}}{\omega^{z_0}\left(B(x,y)\cap S\right)}<\frac{\epsilon}{2}.
			\end{align*}
			Lemma \ref{lem:MartinKernelDecay} implies $k(x_y,\xi)\leq\frac{c}{\omega^{z_0}(B(x,y)\cap S)}2^{-n\alpha}$ for $\xi\in S\setminus B(x,2^ny)$, thus
			\begin{align*}
				|(K_y\varphi)(x)|&\leq\int_{\mathcal{K}_\epsilon}k(x_y,\xi)|\varphi(\xi)|\mathrm{d}\omega^{z_0}(\xi) + \int_{S\setminus\mathcal{K}_\epsilon}k(x_y,\xi)|\varphi(\xi)|\mathrm{d}\omega^{z_0}(\xi)\\
				&\leq M\int_{\mathcal{K}_\epsilon} k(x_y,\xi)\mathrm{d}\omega^{z_0}(\xi)+ \frac{\epsilon}{2}\int_{S\setminus\mathcal{K}_\epsilon}k(x_y,\xi)\mathrm{d}\omega^{z_0}(\xi)\\
				&\leq \frac{c2^{-n\alpha}}{\omega^{z_0}\left(B(x,y)\cap S\right)}+\frac{\epsilon}{2}<\epsilon.
			\end{align*}	
		\end{proof}
		\begin{figure}[h!]
			\includegraphics[width=\textwidth]{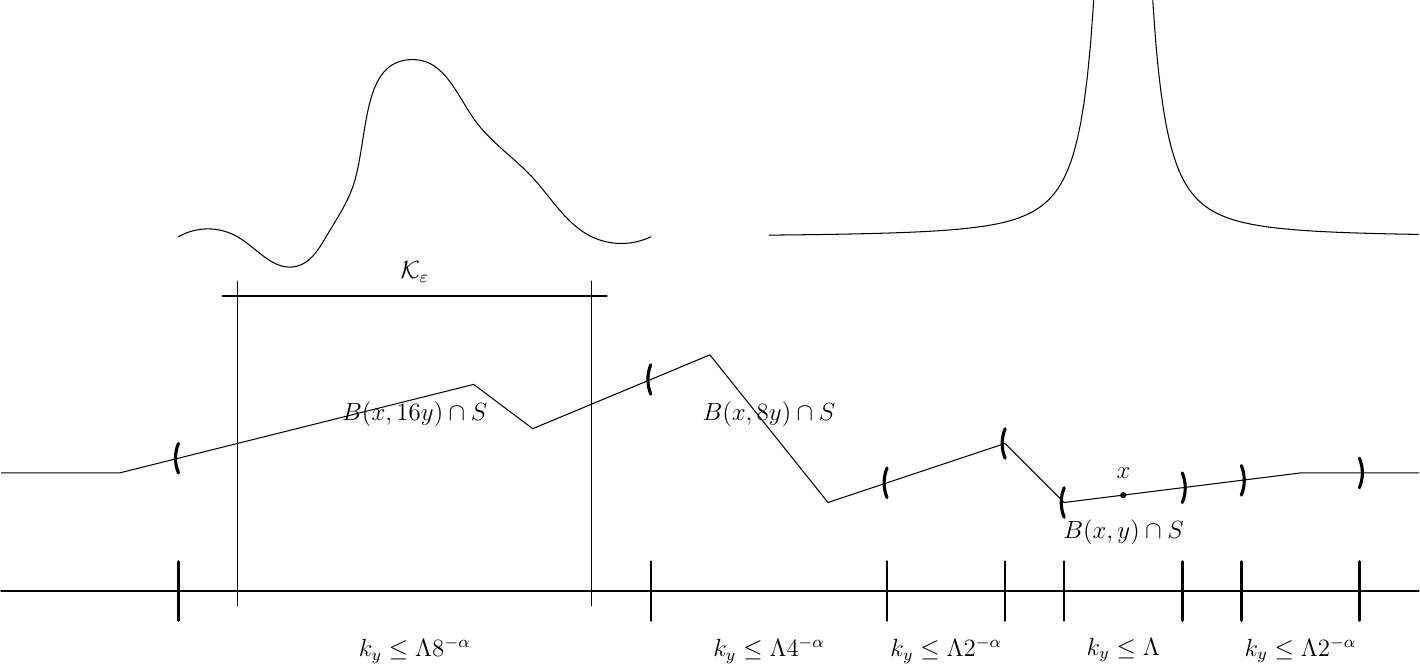}
			\caption{$\Lambda:= \frac{c}{\omega^{z_0}\left(B(x,y)\cap S\right)}$}
		\end{figure}
		An immediate consequence now is the announced property:
		\begin{corollary}
			\label{cor:Im_Omega_y}
			For $y\in(0,1/4)$, $\Omega_y\left(C_0(S)\right)\subseteq C_0(S)$.
		\end{corollary}
		\begin{proof}
			Let $\varphi\in C_0(S)$ and observe that $\omega_y>0$. By Lemma \ref{lem:omega_ky_comparison} we have 
			\begin{align*}
				|\Omega_y\varphi|\leq \frac{1}{y^{c_S^+\varepsilon}} K_{1-y}|\varphi|.
			\end{align*}
			The right side is in $C_0(S)$ by Lemma \ref{lem:ImageK_y}.
		\end{proof}	
		\subsubsection{The measure $\nu_\varepsilon$ and its properties}
			A classical Riesz representation theorem characterizes the dual of $C_0(S)$ as the space of regular Borel measures $\mathcal{M}(S)$ on $S$ \cite{Rudin1999}. Since we have ascertained ourselves, that $\Omega_y$ maps a continuous and vanishing functions again to a continuous and vanishing function, i.e. the restriction of $\Omega_y:C_0(S)\longrightarrow C_0(S)$ is well-defined, it is natural to ask for a characterization of the dual operator $\Omega_y^*:\mathcal{M}(S)\longrightarrow\mathcal{M}(S)$.  With Fubini we obtain for all $\kappa\in\mathcal{M}(S)$ and any $\varphi\in C_0(S)$ 
		\begin{align}
			\label{eq:dual}
			\begin{aligned}
				\left\langle\Omega_y\varphi,\kappa\right\rangle&=\left\langle\int_S\omega_y(\cdot,\xi)\varphi(\xi)\mathrm{d}\omega^{z_0}(\xi), \kappa(\cdot)\right\rangle\\
				&=\int_S\varphi(\xi)\int_S\omega_y(x,\xi)\mathrm{d}\kappa(x)\mathrm{d}\omega^{z_0}(\xi)\\
				&=\left\langle\varphi,\left(\int_S\omega_y(x,\cdot)\mathrm{d}\kappa(x)\right)\mathrm{d}\omega^{z_0}(\cdot)\right\rangle,
			\end{aligned}
		\end{align}
		i.e. $\Omega_y^*(\kappa)$ has density
		\begin{align}
			\label{eq:gamma}
			\gamma_y(\xi):=\int_S\omega_y(x,\xi)\mathrm{d}\kappa(x)
		\end{align}
		with respect to the harmonic measure $\omega^{z_0}$. For our purpose, it suffices to determine $\Omega_y^*$ evaluated at probability measures $\kappa\in\mathcal{M}(S)$. In that case we can repeat the same calculation as in \eqref{eq:dual} with continuous and bounded $\varphi$. Notice, that $\gamma_y>0$ for $y<1/2$ and by setting $\varphi=1$ we obtain
		\begin{align*}
			\int_S\gamma_y\mathrm{d}\omega^{z_0}=\int_S\Omega_y(1)\mathrm{d}\kappa,
		\end{align*}
		thus $\gamma_y\cdot\omega^{z_0}=\Omega^*_y(\kappa)$ is in fact again a probability measure. We will now proceed to show that the $\gamma_y\cdot\omega^{z_0}$ converges in the weak* topology to some measure, that features certain crucial properties.
		\subsubsection{Construction of $\nu_\varepsilon$}
		Fix a probability measure $\kappa\in\mathcal{M}(S)$. We may later choose $\kappa$ to our liking. For $y\in(0,1)$, each measure $\gamma_y\cdot\omega^{z_0}$ generates a functional on $C_0(S)$ in the canonical way, i.e.
		\begin{align*}
			F_y:C_0(S)\longrightarrow\mathbb{R};\quad\alpha\longmapsto\int_S\alpha\gamma_y\mathrm{d}\omega^{z_0}=\int_S\Omega_y\alpha\mathrm{d}\kappa.
		\end{align*}
		Furthermore,  for $y<1/2$, we immediately observe that
		\begin{align*}
			\left\|F_y\alpha\right\|_{C_0(S)}\leq\left\|\alpha\right\|_{C_0(S)}\int_S \Omega_y1\mathrm{d}\kappa,
		\end{align*}
		since $\omega_y>0$ and due to $\Omega(1)=1$, we have $\left\|F_y\right\|_{C_0(S)^*}\leq1$. By Banach-Alaoglu, any sequence of scalars $(y_k)_k$ tending to $0$, contains a subsequence, again denoted by $(y_k)_k$ such that $F_{y_k}$ has a weak* limit $F\in C_0(S)^*$ satisfying $\left\|F\right\|_{C_0(S)^*}\leq1$. Applying Riesz's representation theorem for $C_0(X)^*$, where $X$ is Hausdorff and locally compact \cite{Rudin1999}, the limit functional $F$ can be represented by a measure, which we will denote by $\nu_\varepsilon$, such that $\nu_\varepsilon(S)=\left\|F\right\|_{C_0(S)^*}$. In total we have 
		\begin{align*}
			\forall \alpha\in C_0(S):\quad\lim_{k\rightarrow\infty}\int_S\alpha\gamma_{y_k}\mathrm{d}\omega^{z_0} = \int_S\alpha\mathrm{d}\nu_\varepsilon.
		\end{align*}
		Since the measures $\gamma_y\cdot\omega^{z_0}$ are in fact probability measures, the functionals $F_y$ could of course also be defined on $C_b(S)$ and then restricted to $C_0(S)$. A priori, their weak* limit however only exists in $C_0(S)^*$.\par 
		We will now check, that $\nu_\varepsilon$ is not the trivial measure and, that it is independent of the choice of the sequence $(y_k)$
	
		\begin{lemma}[\cite{MuellerRiegler2020}]
			\label{lem:nu_epsilon_prob}
			The measure $\nu_\varepsilon$ constructed above is a probability measure and independent of the choice of $(y_k)_k$.
		\end{lemma}
		\begin{proof}
			We begin by proving $\nu_\varepsilon(S)=1$, by the construction $\nu_\varepsilon(S)\leq1$.¸ For any $\rho\geq 1$ and $n\in\mathbb{N}$, observe that the function 
			\begin{align*}
				z\in S\longmapsto\omega^{z+\rho}(S\cap B_n),
			\end{align*}
			where $B_n:=B(0,n)$, is continuous and vanishes at infinity, i.e. $\omega^{\cdot+\rho}(S\cap B_n)\in C_0(S)$. Moreover, it is the restriction of the positive harmonic function $\omega^{\cdot}(S\cap B_n)$ on $S_\rho$. Since $[y_k,1]\subset(0,\rho]$ for all $k\in\mathbb{N}$, we use the improved bound \eqref{eq:phi_prop_lower_ineq} in the proof of the $\Phi$-property to estimate
			\begin{align*}
				\int_S\omega^{z+\rho}(S\cap B_n)\mathrm{d}\nu_\varepsilon&=\lim_{k\rightarrow\infty}\int_S\omega^{z+\rho}(S\cap B_n)\gamma_{y_k}\mathrm{d}\omega^{z_0}\\
				&=\lim_{k\rightarrow\infty}\int_S\Omega_{y_k}(\omega^{\cdot+\rho}(S\cap B_n))\mathrm{d}\kappa\\
				&\geq\lim_{k\rightarrow\infty}\exp\left(-\frac{c_S(1-y_k)}{\rho}\right)\int_S\omega^{z+\rho}(S\cap B_n)\mathrm{d}\kappa\\
				&=\exp\left(-\frac{c_S}{\rho}\right)\int_S\omega^{z+\rho}(S\cap B_n)\mathrm{d}\kappa.
			\end{align*}
			For fixed $\rho\geq1$, $\omega^{\cdot+\rho}(S\cap B_n)\in[0,1]$ and converges pointwise to $1$ as $n\longrightarrow\infty$. Since the constant function $1$ is thus an integrable majorant of $\omega^{\cdot+\rho}(S\cap B_n)$ w.r.t to both $\kappa$ and $\nu_\varepsilon$, we obtain the inequality
			\begin{align}
				\label{eq:pos_nu}
				\nu_\varepsilon(S)\geq\exp\left(-\frac{c_S}{\rho}\right)
			\end{align}
			after passing to the limit $n\rightarrow\infty$. \eqref{eq:pos_nu} now holds for all $\rho\geq1$ and therefore $\nu_\varepsilon(S)\geq1$.\par 
			We now turn to showing, that the obtained probability measure $\nu_\varepsilon$ is independent of the choice of $(y_k)_k$. For that, let $(\eta_k)_k$ be another sequence of scalars tending to $0$. Since $(y_k)_k$ was arbitrary the construction, we can guarantee the existence of a weak* limit of the functionals $F_{\eta_k}$. Let $\mu$ denote the probability measure that generates the weak* limit functional of the sequence $(F_{\eta_k})_k\subset C_0(S)^*$. We want to show that then again
			\begin{align}
				\label{eq:limit_functional}
				\int_S\alpha\mathrm{d}\mu=\lim_{k\rightarrow\infty}\int_S\alpha\gamma_{\eta_k}\mathrm{d}\omega^{z_0}=\int_S\alpha\mathrm{d}\nu_\varepsilon
			\end{align}
			holds for any $\alpha\in C_0(S)$. We begin by extracting a subsequence of $(\eta_k)_k$ such that $\eta_k<y_k$ for all $k$ sufficiently large and show \eqref{eq:limit_functional} for any $K_\sigma\alpha$ with $\alpha\in C_0(S)$ and $\sigma>0$ (Note that $K_\sigma\alpha\in C_0(S)$ by Lemma \ref{lem:ImageK_y}). Let $k_0$ be large enough such that $\eta_k<y_k<\max(\sigma,1/2)$ for any $k>k_0$. Then we obtain
			\begin{align*}
				\left|\int_SK_\sigma(\alpha)\gamma_{\eta_k}\mathrm{d}\omega^{z_0}-\int_SK_\sigma(\alpha)\gamma_{y_k}\mathrm{d}\omega^{z_0}\right|&=\left|\int_S(\Omega_{\eta_k}-\Omega_{y_k})(K_\sigma\alpha)\mathrm{d}\kappa\right|\\
				&=\left|\int_S(\Omega_{y_k}(\Omega_{[\eta_k,y_k]}K_\sigma\alpha-K_\sigma\alpha)\mathrm{d}\kappa\right|\\
				&\leq\left\|\Omega_{[\eta_k,y_k]}K_\sigma\alpha-K_\sigma\alpha\right\|_{C_0(S)},
			\end{align*}
			since $\Omega_{\eta_k}=\Omega_{y_k}\Omega_{[\eta_k,y_k]}$ and $\Omega_{y_k}1=1$. Using the $\Phi$-property yields
			\begin{align*}
				\left\|\Omega_{[\eta_k,y_k]}K_\sigma\alpha-K_\sigma\alpha\right\|_{C_0(S)}\leq c_S\frac{y_k}{\sigma}\left\|K_\sigma\alpha\right\|_{C_0(S)}\leq c_S\frac{y_k}{\sigma}\left\|\alpha\right\|_{C_0(S)},
			\end{align*}
			by the maximum principle for unbounded domains and harmonic functions thereon, that vanish at infinity (see remark \ref{rem:maximum_principle}). In total we have 
			\begin{align*}
				\left|\int_SK_\sigma(\alpha)\gamma_{\eta_k}\mathrm{d}\omega^{z_0}-\int_SK_\sigma(\alpha)\mathrm{d}\nu_\varepsilon\right|\leq c_S\frac{y_k}{\sigma}\left\|\alpha\right\|_{C_0(S)}+\left|\int_SK_\sigma(\alpha)\gamma_{y_k}\mathrm{d}\omega^{z_0}-\int_SK_\sigma(\alpha)\mathrm{d}\nu_\varepsilon\right|
			\end{align*}
			for any $\alpha\in C_0(S)$, implying \eqref{eq:limit_functional} for all $K_\sigma\alpha$ with $\alpha\in C_0(S)$ and $\sigma>0$.  Now by the triangle inequality we obtain
		\begin{align*}
			\left|\int_S\alpha\gamma_{\eta_k}\mathrm{d}\omega^{z_0}-\int_S\alpha\mathrm{d}\nu_\varepsilon\right|&\leq\int_S\left|\alpha-K_\sigma\alpha\right|\gamma_{\eta_k}\mathrm{d}\omega^{z_0}+\left|\int_S K_\sigma\alpha\gamma_{\eta_k}\mathrm{d}\omega^{z_0}-\int_SK_\sigma\alpha\mathrm{d}\nu_\varepsilon\right|\\
			&+\int_S\left|K_\sigma\alpha-\alpha\right|\mathrm{d}\nu_\varepsilon
		\end{align*}.
		Taking the limit $k\rightarrow\infty$, the middle term on the right vanishes and we obtain 
		\begin{align}\label{eq:nu_is_mu}
			\left|\int_S\alpha\mathrm{d}\mu-\int_S\alpha\mathrm{d}\nu_\varepsilon\right|\leq\int_S\left|\alpha-K_\sigma\alpha\right|\mathrm{d}\mu+ \int_S\left|K_\sigma\alpha-\alpha\right|\mathrm{d}\nu_\varepsilon
		\end{align}
		for all $\sigma\in(0,1/2)$. Observe that again the maximum principle in Remark \ref{rem:maximum_principle}  implies $\left\|K_\sigma\alpha-\alpha\right\|_{C_0(S)}\leq 2\left\|\alpha\right\|_{C_0(S)}$ and by Theorem \ref{thm:ex_harm_ext} we have $K_\sigma\alpha\rightarrow\alpha$ pointwise on $S$ if $\sigma\rightarrow0$. So by dominated convergence and the fact that $2\|\alpha\|_{C_0(S)}$ is integrable w.r.t both $\nu_\varepsilon$ and $\mu$, the right side in \eqref{eq:nu_is_mu} vanishes and thus $\nu_\varepsilon=\mu$.
		\end{proof}
		Lemma tells us that the $\nu_\varepsilon$ is a probability measure, however this is a weaker property than is required for our final argument in Theorem \ref{thm:main}. We will thus show in the next Lemma, that in fact any surface ball is contained in the support of $\nu_\varepsilon$.
		\begin{lemma}[Properties of $\nu_\varepsilon$ \cite{MuellerRiegler2020}]
			\label{lem:nu_prop}
			The measure $\nu_\varepsilon$ defined above satisfies the following:
			\begin{enumerate}
				\item For any $\varepsilon\in(0,\varepsilon(S))$ we have the estimate
				\begin{align*}
					\int_S V\mathrm{d}\nu_\varepsilon\leq\frac{c_S}{\varepsilon}\int_Su_1\mathrm{d}\kappa.
				\end{align*}
				\item For any $x\in S$ and $r>0$ there exists $\varepsilon(x,r)>0$ such that we have
				\begin{align*}
					\forall \varepsilon\in(0,\varepsilon(x,r)): \quad \nu_\varepsilon( B(x,r)\cap S)>c,
				\end{align*}
				where $c=c(\kappa,S,r)>0$.
			\end{enumerate}
		\end{lemma}
		\begin{proof}
			\begin{enumerate}
				\item For $y\in(0,1]$, we denote $g_y:=B_yu_y$. Notice that with Fubini and Lemma \ref{lem:cy_properties}, we obtain for all $x\in S$
				\begin{align}
					\label{eq:B_yK_y}
					\begin{aligned}
							(B_yu_y)(x)&=\int_S\int_S k_y(x,\zeta) c_y(\zeta,\xi)u_y(\xi)\mathrm{d}\omega^{z_0}(\zeta)\mathrm{d}\omega^{z_0}(\xi)\\
						&=\int_S k_y(x,\zeta) (C_yu_y)(\zeta)\mathrm{d}\omega^{z_0}(\zeta) = \left(K_y\|\nabla u_{2y}\|\right)(x).
					\end{aligned}
				\end{align}
				Since for fixed $\tau>0$, the function $z=x_\vartheta\in \mathcal{O}\longmapsto \left(K_\vartheta\|\nabla u_{2\tau}\|\right)(x)$ is positive harmonic, choosing $\tau=y$ yields that $g_y$ is the trace on $S_y$ of a positive harmonic function defined on $\mathcal{O}$. Corollary \ref{cor:very_important_ineq} now implies that $\Omega_\eta g_y\leq\Omega_y g_y$ for all $0<\eta<y<\frac{1}{2}$.
				Now define 
				\begin{align*}
					J(\delta):=\int_S\int_\delta^1 g_y\mathrm{d}y\mathrm{d}\nu_\varepsilon.
				\end{align*}
				Next observe that  $f_\beta(x):=\int_\delta^{\beta} g_y(x)\mathrm{d}y\in C_0(S)$ for any $\beta\in\mathbb{R}$. To begin, we show that $f_\beta$ vanishes at infinity on $S$. Using the pointwise inequality in \eqref{eq:B_yK_y} and Corollary \ref{cor:Harnack_gradient}, we get
				\begin{align*}
					f_\beta(x)&=\int_\delta^{\beta} \int_Sk_y(x,\xi)\|\nabla u_{2y}(\xi)\|\mathrm{d}\omega^{z_0}(\xi)\mathrm{d}y\\
					&\leq c\int_\delta^{\beta} \int_Sk_y(x,\xi)\frac{u_{2y}(\xi)}{d(\xi_{2y},S)}\mathrm{d}\omega^{z_0}(\xi)\\
					&\leq\frac{c_S}{2\delta}\int_\delta^{\beta} \int_S\left(\frac{y}{\delta}\right)^\alpha k_\delta(x,\xi)\left(\frac{2y}{\delta}\right)^\alpha u_{\delta}(\xi)\mathrm{d}\omega^{z_0}(\xi)\\
					&\leq c_S\frac{\beta^{2\alpha}}{\delta^{2\alpha+1}}\left(\beta-\delta\right)\left(K_\delta u_\delta\right)(x).
				\end{align*} 
				Now by Lemma \ref{lem:ImageK_y}, the claim follows since $u_\delta|_S\in C_0(S)$. For continuity recall again \eqref{eq:B_yK_y}, i.e. $(x,y)\longmapsto g_y(x)$ is continuous and so is $f$. We now prove that $J(\delta)$ is uniformly bounded. Since $f_1\in C_0(S)$, we exploit that $\nu_\varepsilon$ is the weak*-limit of $\gamma_\eta\mathrm{d}\omega^{z_0}$ and their duality to the operator $\Omega_\eta$ to obtain
				\begin{align*}
					J(\delta)&=\lim_{\eta\rightarrow0}\int_S\int_\delta^{1} g_y\mathrm{d}y\gamma_\eta\mathrm{d}\omega^{z_0}\\
					&=\lim_{\eta\rightarrow0}\int_S\Omega_\eta\left(\int_\delta^{1} g_y\mathrm{d}y\right)\mathrm{d}\kappa\\
					&=\lim_{\eta\rightarrow0}\int_S\int_\delta^{1} \Omega_\eta g_y\mathrm{d}y\mathrm{d}\kappa\\
					&= \lim_{\eta\rightarrow0}\left(\int_S\int_\delta^{\frac{1}{2}}\Omega_\eta g_y\mathrm{d}y\mathrm{d}\kappa+\int_S\int_{\frac{1}{2}}^1\Omega_\eta g_y\mathrm{d}y\mathrm{d}\kappa\right)\stepcounter{equation}\tag{\theequation}\label{eq:splitting_integral}.
				\end{align*}
				Now by Corollary \ref{cor:very_important_ineq} and the remarks below equation \eqref{eq:B_yK_y}, we can further estimate the integrand in the first term in \eqref{eq:splitting_integral} to obtain
				\begin{align*}
					\int_S\int_\delta^{\frac{1}{2}}\Omega_\eta g_y\mathrm{d}y\mathrm{d}\kappa&\leq c_S\int_S\int_\delta^{\frac{1}{2}}\Omega_yg_y\mathrm{d}y\mathrm{d}\kappa\\
					&\stackrel{(*)}{=}\frac{c_S}{\varepsilon}\int_S\left(\Omega_{\frac{1}{2}} u_{\frac{1}{2}}-\Omega_\delta u_\delta\right)\mathrm{d}\kappa\\
					&\leq\frac{c_S}{\varepsilon}\int_S\Omega_{\frac{1}{2}} u_{\frac{1}{2}}\mathrm{d}\kappa,
				\end{align*}
				where $\delta>0$ is fixed and small enough such that $\omega_\delta>0$ and in $(*)$, the differential equation in Theorem \ref{thm:diff-eq} was used. Now the quotient estimates in Lemma \ref{lem:harnack_chain} yield
				\begin{align*}
					\frac{u_1}{u_{\frac{1}{2}}}\geq\left(\frac{1}{2}\right)^\alpha \Longleftrightarrow u_{\frac{1}{2}}\leq2^\alpha u_1
				\end{align*}
				thus, by the positivity of $\omega_{\frac{1}{2}}$, we have $\Omega_{\frac{1}{2}}u_\frac{1}{2}\leq c_S \Omega_\frac{1}{2}u_1$. Moreover, the $\Phi$-property implies $\Omega_{\frac{1}{2}}u_1\leq(1+c_S)u_1$, since $[1/2,1]\subset(0,1]$.
				Combining these two bounds, we arrive at
				\begin{align*}
					\int_S\int_\delta^{\frac{1}{2}}\Omega_\eta g_y\mathrm{d}y\mathrm{d}\kappa\leq\frac{c_S}{\varepsilon}\int_S u_1\mathrm{d}\kappa.
				\end{align*}
				For the second term in \eqref{eq:splitting_integral} observe that due to \eqref{eq:B_yK_y} we have
				\begin{align*}
					g_y = K_y\|\nabla u_{2y}\|&\leq \frac{c_S}{2y}K_yu_{2y}=\frac{c_S}{2y}u_{3y}\leq c_S\frac{(3y)^\alpha}{2y}u_1\leq c_Su_1
				\end{align*} 
				for $y\in(1/2,1)$, so in particular $3y>1$. For $\eta<1/2$, $\omega_\eta$ is positive, thus we have the inequality $\Omega_\eta g_y\leq c_S \Omega_\eta u_1$. Furthermore, the $\Phi$-property implies that $\Omega_\eta u_1\leq \left(1+c_S(1-\eta)\right)u_1$, since $[\eta,1]\subset(0,1]$. Plugging into the integral yields
				\begin{align*}
					\int_S\int_{\frac{1}{2}}^1\Omega_\eta g_y\mathrm{d}y\mathrm{d}\kappa\leq c_S(1+c_S(1-\eta))\int_Su_1\mathrm{d}\kappa.
				\end{align*}
				In total this the above considerations lead to 
				\begin{align*}
					J(\delta)\leq\frac{c_S}{\varepsilon}\int_Su_1\mathrm{d}\kappa.
				\end{align*}
				Taking limits on both sides we now obtain
				\begin{align*}
					\limsup_{\delta\downarrow0} J(\delta)\leq\frac{c_S}{\varepsilon}\int_S  u_1\mathrm{d}\kappa.
				\end{align*} 
				Since $f\geq0$, $J(\delta)$ is monotonically increasing for falling $\delta>0$, we obtain by monotone convergence 
				\begin{align*}
					\limsup_{\delta\downarrow0} J(\delta)=\int_SV\mathrm{d}\nu_\varepsilon
				\end{align*} 
				and thus the desired inequality
				\begin{align*}
					\int_S V\mathrm{d}\nu_\varepsilon\leq\frac{c_S}{\varepsilon}\int_Su_1\mathrm{d}\kappa.
				\end{align*}
				\item We prove the second statement. Let $\zeta\in S$, $r<1/4$ and define $B:= B(\zeta,r)$.  There exists a differentiable function $\psi$ defined on $\mathcal{O}_{-1}$, such that $\psi(S)\subseteq[0,1]$,  $\psi|_{B/2}\equiv1$,  $\psi|_{S\setminus B}\equiv0$ and $\|\nabla\psi\|\leq2/r$. Now let $\varphi:=\psi|_S\in C_0(S)$, for convenience we will also denote by $\varphi$ the harmonic extension of $\psi|_S$ to $\mathcal{O}$. Using the differential equation in Theorem \ref{thm:diff-eq} we write for $y\in(0,r)$
				\begin{align*}
					\Omega_y\varphi_y = \Omega_r\varphi_r-\varepsilon\int_{y}^{r}\Omega_t(B_t\varphi_t)\mathrm{d}t.
				\end{align*}
				Let us estimate the first term from below and the second term from below. To begin, notice that with Lemma \ref{lem:omega_ky_comparison} and positivity of $\omega_y$, we obtain
				\begin{align*}
					\Omega_r\varphi_r\geq r^{c_S\varepsilon}K_{1-r}\varphi_r=r^{c_S\varepsilon}\varphi_1
				\end{align*}
				for sufficiently small $\varepsilon<\varepsilon(S)$ and thus integrating over $S$ w.r.t $\kappa$ gives
				\begin{align*}
					\int_S\Omega_r\varphi_r\mathrm{d}\kappa\geq  c_{\lbrace\kappa, B\rbrace}r^{c_S\varepsilon} 
				\end{align*}
				since $\varphi_1$ has a special form i.e.,
				\begin{align*}
					\varphi_1(x) &= \int_{B} k(x_1,\xi)\psi(\xi)\mathrm{d}\omega^{z_0}(\xi)\\
					&=\int_{\frac{1}{2}B\cap S}k(x_1,\xi)\mathrm{d}\omega^{z_0}(\xi)+\int_{B\setminus\frac{1}{2}B\cap S}k(x_1,\xi)\psi(\xi)\mathrm{d}\omega^{z_0}(\xi)\\
					&\geq \omega^{x_1}\left(\frac{1}{2}B\cap S\right).
				\end{align*}
				In particular we choose $c_{\kappa,B}=\int_S\omega^{x_1}\left(\frac{1}{2} B\right)\mathrm{d}\kappa(x)$. Moreover, using positivity of $\omega_t$ for $t\in(y,r)$ and the fact that $\Omega_t(1)=1$, we get
				\begin{align*}
					\left|\int_y^r\Omega_t(B_t\varphi_t)\mathrm{d}t\right|\leq (r-y)\sup_{\substack{x\in S\\t\in(y,r)}}|(B_t\varphi_t)(x)|.
				\end{align*}
				By Lemma \ref{lem:cy_properties} and \ref{lem:properties_b_y}, $|(B_t\varphi_t)(x)|\leq\sup_{S}\|\nabla\varphi_{2t}\|$, hence
				\begin{align*}
					\sup_{\substack{x\in S\\t\in(y,r)}}|(B_t\varphi_t)(x)|\leq\sup_{\mathcal{ O}}\|\nabla\varphi\|\leq\frac{c_{S,B}}{r}.
				\end{align*}
				In total, we combine estimates to obtain
				\begin{align}
					\label{eq:Omega_yphi_y-dkappa}
					\int_S\Omega_y\varphi_y\mathrm{d}\kappa\geq c_{\lbrace S,\kappa,B\rbrace}\left(r^{c_S\varepsilon}+\varepsilon\right),
				\end{align}
				which is greater than a constant depending on $B$ and $\kappa$ for $\varepsilon\in(0,\varepsilon(B,\kappa))$. Having established this lower bound, we exploit duality and estimate
				\begin{align*}
					\int_S \Omega_y\varphi_y\mathrm{d}\kappa &= \int_S\varphi_y\gamma_y\mathrm{d}\omega^{z_0}\\
					&\stackrel{(*)}{\leq} c_S\int_S \Omega_{[\delta,y]}\varphi_y\gamma_y\mathrm{d}\omega^{z_0}\\
					&=c_S\int_S \Omega_y(\Omega_{[\delta,y]}\varphi_y)\mathrm{d}\kappa\\
					&=c_S\int_S\varphi_y\gamma_\delta\mathrm{d}\omega^{z_0}
				\end{align*}
				 for $\delta<y$, where $\gamma_y=\gamma_y(\kappa)$ is defined in \eqref{eq:gamma} and is positive for $y>0$ small enough. In $(*)$ we used inequality \eqref{eq:phi_prop_lower_ineq} from the proof of the $\Phi$-property with the constant $\exp\left(\frac{4c_S(y-\delta)}{y}\right)$. Since $\nu_\varepsilon$ is the weak* limit of $\gamma_\delta\mathrm{d}\omega^{z_0}$ as $\delta\downarrow0$, taking the limit and combining with the estimate in equation \eqref{eq:Omega_yphi_y-dkappa} yields
				 \begin{align*}
				 	\int_S\varphi_y\mathrm{d}\nu_\varepsilon>c_{\lbrace S,\kappa,B\rbrace}
				 \end{align*}
			 	for every $y\in(0,r)$. Taking the limit $y\rightarrow0$ we obtain 
			 	\begin{align*}
			 		\lim_{y\rightarrow0}\int_S\varphi_y\mathrm{d}\nu_\varepsilon = \int_S\psi\mathrm{d}\nu_\varepsilon\leq\nu_\varepsilon(B),
			 	\end{align*}
		 		since $\varphi_y\rightarrow\psi$ pointwise on $S$. Note that we can apply dominated convergence because by the maximum principle for harmonic functions vanishing at infinity on unbounded domains, we have
		 		\begin{align*}
		 			\sup_{\mathcal{O}}\varphi=\sup_S\psi = 1,
		 		\end{align*}
	 			see Remark \ref{rem:maximum_principle}.
			\end{enumerate}
		\end{proof}
		\begin{remark}
			\label{rem:maximum_principle}
			The classical weak maximum principle states \cite{Gilbarg2001}:\par
			\textit{Let $\Omega$ be a bounded domain and $L$ an elliptic differential operator of the form
			\[
				L=\sum_{i,j=1}^da_{ij}D_{ij}+\sum_{i=1}^db_iD_i.
			\]
			Note that there is no contribution of the function $f$ in $Lf$. Functions $\varphi\in C^2(\Omega)\cap C(\bar{\Omega})$ satisfying $L\varphi=0$ on $\Omega$ follow the (weak) maximum principle}
			\begin{align*}
				\sup_\Omega |\varphi|=\sup_{\partial\Omega}|\varphi|.
			\end{align*} 
			We will now extend the classical maximum principle to unbounded domains $\Omega$ and solution $\varphi$ to $L\varphi=0$ on $\Omega$, that vanish at infinity, i.e. $\varphi\in C^2(\Omega)\cap C(\bar{\Omega})\cap C_0(\Omega)$. Denote $M=\sup_{\Omega}|\varphi|$. Suppose $x\in\partial\Omega$, and choose $R>0$ large enough, such that on $\Omega\setminus B$, we have $|\varphi|<M/2$ and $\Omega\cap B$ is connected, where $B=B(x,R)$. Now on the domain $\Omega\cap B$, the weak maximum principle holds i.e. 
			\begin{align*}
				\sup_{\Omega\cap B}|\varphi|=\sup_{\partial\left(\Omega\cap B\right)}|\varphi|=M,
			\end{align*}
			since $|\varphi|<M/2$ on $\Omega\setminus B$. Due to the same reasoning, we have $\sup_\Omega|\varphi|=\sup_{\Omega\cap B}|\varphi|$ and $\sup_{\partial\Omega}|\varphi|=\sup_{\partial\left(\Omega\setminus B\right)}|\varphi|$ and thus
			\begin{align*}
				\sup_{\Omega}|\varphi|=\sup_{\partial\Omega}|\varphi|.
			\end{align*}
		\end{remark}
		This concludes all preparations needed for proving the main statement of \cite{MuellerRiegler2020}.
		\subsection{Proof of Main Theorem}
		First, let us restate Theorem \ref{thm:main} :\par 
		\textit{
			Let $\mathcal{O}$ be a near-half space, $S=\partial\mathcal{O}$, and $y>1$. Then for any $B=B(z,r)$, with $z\in S$ and $r>0$ there exists $x\in B\cap S$ such that the variation $V(x)$ is bounded in the following way:\\
			There exists a constant $c=c_{S,r}$, depending on the Lipschitz function parameterizing $S$ and the radius of $B$, such that
			\[
			V(x)\leq cu(x_y).
			\]
		}
		\begin{proof}
			The constraint $y>1$ is arbitrary, we only require large distance to $0$. Let $B=B(z,r)\cap S$ be any surface ball. We may choose for $\kappa$ a specific probability measure on $S$, and construct $\nu_\varepsilon=\nu_\varepsilon(\kappa)$.
			With $\kappa=\omega^{z_y-\vec{e_d}}$, we find $\varepsilon(B)>0$ such that for any $\varepsilon<\varepsilon(B)$, we have $\nu_\varepsilon(B)>c(S,\kappa,r)>0$, by Lemma \ref{lem:nu_prop}. Choose an appropriate $\varepsilon>0$ and denote $c(S,\kappa,r)=:\tilde{c}$. By the same Lemma \ref{lem:nu_prop}, we obtain
			\begin{align*}
				\int_SV(\xi)\mathrm{d}\nu_\varepsilon(\xi)\leq\frac{c_S}{\varepsilon}\int_Su_1(\xi)\mathrm{d}\omega^{z_y-\vec{e_d}}(\xi) = \frac{c_S}{\varepsilon}u(z_y).
			\end{align*}
			Now assume that for any $x\in B$, 
			\begin{align*}
				V(x)>\frac{2c_S}{\varepsilon \tilde{c}}u(z_y)
			\end{align*}
			holds. We obtain an immediate contradiction by
			\begin{align*}
				\frac{2c_S}{\varepsilon}u(z_y)<\frac{2c_S}{\varepsilon}\frac{\nu_\varepsilon(B)}{\tilde{c}}u(z_y)\leq\int_BV(\xi)\mathrm{d}\nu_\varepsilon(\xi)\leq\frac{c_S}{\varepsilon}u(z_y).
			\end{align*}
			This yields the existence of $x\in B$ such that $V(x)\leq c u(z_y)$ with $c=c(S,r,\kappa,\varepsilon)$. Observe that by Lemma \ref{lem:dist_to_boundary} $B(z_y,c_sy)\subset \mathcal{O}$, thus since $y>1$, $u$ is positive harmonic on the ball $B(z_y,c_S)$. Additionally, notice that for small enough $r<c_S/2$, $d(x_y,z_y)<c_S/2$ and therefore by Harnack's inequality \ref{thm:Harnack_ineq}
			\begin{align*}
				u(z_y)\leq\frac{\left(1+\frac{d(z_y,x_y)}{c_s}\right)^{d-1}}{1-\frac{d(z_y,x_y)}{c_s}}u(x_y)\leq \frac{1}{2} \left(\frac{3}{2}\right)^{d-1}u(x_y).
			\end{align*} 
			In total we obtain $V(x)\leq cu(x_y)$.
		\end{proof}
		
		\section{Hausdorff dimension of $\mathcal{V}$}\label{sec:Hausdorff}
		In this last section we will show that the Hausdorff dimension of $\mathcal{V}$ is at least $(d-1)\frac{1+\eta}{2+\eta}$, for some $\eta>0$, that depends on $S$ and $z_0$. This is a considerably weaker statement than in Havin and Mozolyako who showed in \cite{HavinMozol2016} that for domains with $C^2$ boundary, $\dim(\mathcal{V})=d-1$. What is more, they even show a much stronger notion of "density" of $\mathcal{V}$ inside $S$: that $\mathcal{V}$ lies \textit{ultradense} in $S$, meaning 
		\begin{align*}
			\forall x\in S \quad\forall r>0:\quad \dim\left(\mathcal{V}\cap B(x,r)\right)=d-1
		\end{align*}
		Obviously ultradensity of $\mathcal{V}$ implies $\dim(\mathcal{V}(S))=d-1$. 
		
		It is not far-fetched to conjecture that also in our case of near half spaces, we have utradensity of $\mathcal{V}$, however the method that works in \cite{HavinMozol2016}, yields weaker estimates for our domains. It is plausible to extend this method to obtain stronger estimates, but not further discussed here.
		
		We begin with estimating the measure of surface balls under $\nu_\varepsilon$. The proof is similar to the corresponding lemma in \cite{HavinMozol2016} for $C^2$-domains and implements the necessary adaptations.
		\begin{lemma}[\cite{HavinMozol2016}]
			\label{lem:nu_upper_bd}
			There exist positive constants $d_S,\varepsilon_{S,d}$ and $\eta=\eta_S>0$, such that for any $x\in S$, $r\in(0,1/2)$ and $\varepsilon\in(0,\varepsilon_{S,d})$ we have
			\begin{align*}
				\nu_\varepsilon(B(x,r)\cap S)\leq cr^{(d-1)\frac{1+\eta}{2+\eta}-d_S\varepsilon},
			\end{align*}
			where $c=c_{\lbrace S,d,\kappa,z_0\rbrace}$.
		\end{lemma}
		\begin{proof}
			In the following, $\lambda B:=B(x,\lambda r)$, where $B=B(x,r)$.
			
			We now fix a surface ball $B:=B(x,r)$ and set $\hat{B}:=4B$. We can now find a function $\psi:=\psi_{\hat{B}}:\mathbb{R}^d\rightarrow\mathbb{R}$ which is continuous and on $\mathcal{O}$ also harmonic. Furthermore, $\psi(S)\subseteq[0,1]$, with $\psi|_{S\setminus2\hat{B}}\equiv0$ and $\psi>q$ on $\frac{1}{2}\hat{B}=2B$, where $q\in(0,1)$ may depend on $S$ but is independent of $\hat{B}$, an thus of $B$. The existence of such a family of functions is shown in \cite{HavinMozol2016} (and requires the condition $||\nabla\Phi||\leq1/100$). 
			
			Let $z\in(S\cap B)_{\frac{r}{2}}$, then $z=\xi_{\frac{r}{2}}$ for some $\xi\in S\cap B$ and thus $\|z-x\|\leq1+r/2$. Hence $(S\cap B)_{\frac{r}{2}}\subset2B$ and $\psi>q$ on $(S\cap B)_{\frac{r}{2}}$. These consideration yield
			\begin{align}
				\label{eq:upper_bound_nu_eps}
				\nu_\varepsilon(B\cap S)\leq\frac{1}{q}\int_{S\cap B}\psi_{\frac{r}{2}}\mathrm{d}\nu_\varepsilon\leq\frac{1}{q}\int_{S}\psi_{\frac{r}{2}}\mathrm{d}\nu_\varepsilon=\frac{1}{q}\lim_{\eta\rightarrow0}\int_S\Omega_{\eta}\psi_{\frac{r}{2}}\mathrm{d}\kappa.
			\end{align}
			By the semi-group property and the $\Phi$-property
			\begin{align*}
				\Omega_\eta\psi_{\frac{1}{2}} = \Omega_{\frac{r}{2}}\Omega_{\left[\eta,\frac{r}{2}\right]}\psi_{\frac{r}{2}}\leq c_S\frac{r/2-\eta}{r/2}\Omega_{\frac{r}{2}}\psi_{\frac{r}{2}},
			\end{align*}
			where the last inequality holds since $r/2<1/4$ and thus $\omega_{\frac{r}{2}}>0$. Now applying Lemma \ref{lem:omega_ky_comparison}, we obtain
			\begin{align*}
				\Omega_{\frac{r}{2}}\psi_{\frac{r}{2}}\leq\left(\frac{2}{r}\right)^{d_S\varepsilon} K_{1-\frac{r}{2}}\psi_{\frac{r}{2}},
			\end{align*}
			where $d_S>0$ is the corresponding constant in equation \eqref{eq:omega_y_k_y_comparison}. Plugging into \eqref{eq:upper_bound_nu_eps} and collecting the above estimates, yields
			\begin{align*}
				\nu_\varepsilon(B\cap S)&\leq \frac{c_S}{q}r^{-d_S\varepsilon}\int_SK_1\psi\mathrm{d}\kappa\\
				&=c_Sr^{-d_S\varepsilon}\int_S\psi(\xi)\int_Sk_1(x,\xi)\mathrm{d}\kappa(x)\mathrm{d}\omega^{z_0}(\xi)\\
				&\leq c_{S,\kappa} r^{-d_S\varepsilon}\omega^{z_0}\left(S\cap2\hat{B}\right),
			\end{align*}
			since $\kappa$ is a probability measure, $\psi$ is supported on $S\cap2\hat{B}$ and $\psi\leq1$. 
			For $\xi\in S$, the mapping
			\begin{align*}
				\xi\longmapsto\int_S k_1(x,\xi)\mathrm{d}\kappa(x)
			\end{align*}
			is bounded on $S$ since it is continuous, non-negative and integrable on $S$, which can be seen by applying Fubini:
			\begin{align*}
				1=\int_S\int_S k_1(x,\xi)\mathrm{d}\kappa(x)\mathrm{d}\omega^{z_0}(\xi).
			\end{align*}
			Now, by Dahlberg's Theorem \ref{thm:Dahlberg}, $\omega^{z_0}$ is absolutely continuous w.r.t $\mathcal{H}^{d-1}$ and the Radon-Nikodym derivative of $\omega^{z_0}$ w.r.t $\mathcal{H}^{d-1}$, denoted by $\delta$, is in $L^{2+\eta}_\text{loc}(S)$, for some $\eta>0$ depending on $S$ and $z_0$. Thus
			\begin{align*}
				\omega^{z_0}(S\cap2\hat{B})=\int_{S\cap2\hat{B}}\delta\mathrm{d}\mathcal{H}^{d-1}\leq\left\|\delta\right\|_{L^{2+\eta}(S\cap2\hat{B})}\left(\mathcal{H}^{d-1}(S\cap2\hat{B})\right)^{\frac{1+\eta}{2+\eta}}=c_{\lbrace d, S,\kappa,z_0\rbrace} r^{(d-1)\frac{1+\eta}{2+\eta}}.
			\end{align*}
		\end{proof}
		In total this yields the desired inequality
		\begin{align*}
			\nu_\varepsilon(B\cap S)\leq c_{\lbrace d, S,\kappa,z_0\rbrace} r^{(d-1)\frac{1+\eta}{2+\eta}-d_S\varepsilon}
		\end{align*}
		for $\varepsilon<(d-1)(1+\eta)/d_S(2+\eta)$.
		
		Equipped with this estimate, we can finally bound the Hausdorff dimension of $\mathcal{V}$.
		
		\begin{theorem}
			For any ball $B=B(x,r)$, with $x\in S$ and $r>0$, we have 
			\begin{align*}
				\dim(\mathcal{V}\cap B)\geq(d-1)\frac{1+\eta}{2+\eta},
			\end{align*}
			where $\eta$ is the constant from Lemma \ref{lem:nu_upper_bd}.
		\end{theorem}
		\begin{proof}
			Let $B$ be as above and denote $M:=\mathcal{V}\cap B$. Due to Lemma \ref{lem:nu_prop} 
			\begin{align*}
				\int_S V(x)\mathrm{d}\nu_\varepsilon<\infty \quad\text{ for } \varepsilon\text{ small enough}.
			\end{align*}
			Thus, the set $\lbrace x\in S:V(x)=\infty\rbrace$ is a $\nu_\varepsilon$ null-set and $\nu_\varepsilon(M)=\nu_\varepsilon(B)$. Now let $(B_j)_{j\in\mathbb{N}}$ be a $\delta$-covering of $M$ of balls centered on $S$ and denote the radius of $B_j$ to be $r_j>0$, i.e. $\sup r_j<\delta$. Then, due to Lemma \ref{lem:nu_prop} part $(2)$ and Lemma \ref{lem:nu_upper_bd}, we have
			\begin{align*}
				0<\nu_\varepsilon(B)=\nu_\varepsilon(M)\leq c\sum_{j\in\mathbb{N}}r_j^{d(\varepsilon)},
			\end{align*} 
			where $d(\varepsilon):=(d-1)\frac{1+\eta}{2+\eta}-d_S\varepsilon$. We thus see that for all $\varepsilon>0$ small enough, $\mathcal{H}^{d(\varepsilon)}(M)>0$. This in turn implies that $\dim(M)\geq(d-1)\frac{1+\eta}{2+\eta}$.
		\end{proof}
		
		\section{From Near Half Spaces to Lipschitz domains}\label{sec:nhs_lip}
		
		In this last section, we turn to lifting the results on boundedness of radial variation on near half spaces to general Lipschitz domains. This will be done in essence by locally identifying the boundary of a Lipschitz domain with a near half space. 
		
		Let $D\subset\mathbb{R}^d$ denote a Lipschitz domain (not necessarily bounded) and let $u$ be positive harmonic and bounded from above on $D$. For any $\rho\in\partial D$ we can (after rotating and translating the space) find quantities $r=r(\rho)>0 $, $h=h(\rho)>0$ such that $\rho$ coincides with the origin and the cylinder
		\begin{align*}
			\mathcal{C}=\mathcal{C}(\rho):=\left\lbrace(x,y)\in\mathbb{R}^{d-1}\times\mathbb{R}:\|x\|< r\text{ and }|y|< h\right\rbrace
		\end{align*}	
		satisfies
		\begin{enumerate}
			\item $D\cap\mathcal{C}=\left\lbrace(x,y)\in\mathbb{R}^{d-1}\times\mathbb{R}:\|x\|< r\text{ and } \phi(x)<y<h\right\rbrace$ and $|\phi|< h/2$ on $B^{d-1}(r)$,
			\item $\mathbb{R}^d\setminus D\cap\mathcal{C}=\left\lbrace(x,y)\in\mathbb{R}^{d-1}\times\mathbb{R}:\|x\|< r\text{ and } -h<y<\phi(x)\right\rbrace$,
			\item $\partial D\cap\mathcal{C}=\phi|_{B^{d-1}(r)}$,
		\end{enumerate}
		where $\phi$ denotes the local parametrization of $\partial D$ at $p$. We extend the Lipschitz function $\phi$ on $B^{d-1}(r)$ to a Lipschitz function $\Phi$ on $\mathbb{R}^{d-1}$ such that $\Phi\equiv0$ outside $B^{d-1}(2r)$ and define the near half space
		\begin{align*}
			\mathcal{O}=\mathcal{O}_\rho:= \left\lbrace(x,y):x\in\mathbb{R}^{d-1}\text{ and }y>\Phi(x)\right\rbrace,
		\end{align*}
		i.e. $\mathcal{O}$ is the epigraph of the Lipschitz function $\Phi$. Notice that $S:=\partial\mathcal{O}=\Phi(\mathbb{R}^{d-1})$ and $S\cap\mathcal{C}\subset S\cap\partial D$. By $\lambda\mathcal{C}$ we will denote the set $\left\lbrace(x,y)\in\mathbb{R}^{d-1}\times\mathbb{R}:\|x\|< \lambda r\text{ and }|y|< h\right\rbrace$.
		
		In the following we will show that for every $\rho\in\partial D$ and its associated cylinder $\mathcal{C}$, we find a positive harmonic function $w$ on $\mathcal{O}$, satisfying the assumptions of the previous sections, such that the equality 
		\begin{align}
			\label{eq:Bourgain_pts_in_cylinders}
			\mathcal{V}^u(\partial D)\cap\frac{1}{4}\mathcal{C}\supseteq\mathcal{V}^w(\partial\mathcal{O})\cap\frac{1}{4}\mathcal{C}
		\end{align}
		holds. (Recall \eqref{eq:BourgainPoints}).
		
		We fix a $\rho\in\partial D$. Recall $\rho=0$ after rotation and translation The domain $U:=D\cap\mathcal{C}$ is contained in the intersection of $D$ and $\mathcal{O}$ and clearly is again a  Lipschitz domain that is bounded. In particular, the function $u$ is positive harmonic and bounded on $U$ and therefore, by the crucial theorem of Hunt and Wheeden \cite[Section 2]{HuntWheeden68}, $u$ admits the following integral representation:
		\begin{align*}
			\forall p\in U:\quad u(p)=\int_{\partial U}k_{p_0}^U(p,q)f(q)\mathrm{d}\omega^{p_0}_U(q),
		\end{align*}
		where $f\in L^1(\partial U,\omega^{p_0}_U)$, $p_0\in U$ is some fixed point, $\omega_U^{p_0}$ denotes the harmonic measure in $U$ with pole at $p_0$ and $k_U^{p_0}(p,\cdot)$ is the Martin kernel, i.e. the Radon-Nikodym derivative of $\omega_U^p$ w.r.t. $\omega_U^{p_0}$. Additionally note, as the nontangential limit of $u$, $f$ inherits the same upper and lower bounds as the positive harmonic function $u$, i.e. $0\leq f\leq C$. Now let $\Sigma:=\phi(B^{d-1}(r/2))\subset\partial U\cap\partial D\cap\mathcal{C}$ and define the function $w$ on the near half space $\mathcal{O}$, as follows:
		\begin{align*}
			\forall p\in\mathcal{O}:\quad w(p)=\int_\Sigma k^\mathcal{O}_{p_0}(p,q)f(q)\mathrm{d}\omega^{p_0}_U(q).
		\end{align*}
		Notice, that $w$ is positive harmonic on $\mathcal{O}$. Furthermore, $w$ is bounded on $\mathcal{O}$ and vanishes at infinity on $S_y$, for $y>0$. This can be seen as follows:
		
		We first proof boundedness. Let $p\in\mathcal{O}$ be arbitrary, then $p=x_y$ for some $x\in S$ and $y>0$. By Lemma \ref{lem:MartinKernelDecay} we have the inequality
		\begin{align*}
			\int_\Sigma k^\mathcal{O}_{p_0}(x_y,q)f(q)\mathrm{d}\omega^{p_0}_U(q)\leq\sum_{j\in\mathbb{N}}\frac{c2^{-\alpha j}}{\omega^{\mathcal{O}}_{p_0}(\Delta_j)}\int_{\Sigma\cap R_j}f(q)\mathrm{d}\omega_U^{p_0}(q),
		\end{align*}
		where $\Delta_j=B(x,2^jy)\cap S$ and $R_j=\Delta_j\setminus\Delta_{j-1}$. By Remark \ref{rem:HarmMeasComp} we can estimate $\omega^{p_0}_\mathcal{O}(\Delta_j)\geq\omega^{p_0}_\mathcal{O}(\Delta_j\cap\partial U)\geq\omega^{p_0}_U(\Delta_j\cap\partial U)$. Since $0\leq f\leq C$, we obtain the uniform bound
		\begin{align*}
			\int_\Sigma k^\mathcal{O}_{p_0}(x_y,q)f(q)\mathrm{d}\omega^{p_0}_U(q)\leq cC\sum_{j\in\mathbb{N}}2^{-\alpha j}<\infty.
		\end{align*}
		
		The second property to check is $\lim_{\substack{\|x\|\rightarrow\infty\\x\in S}}w(x_y)=0$ for any $y>0$. This is again a direct consequence of Lemma \ref{lem:MartinKernelDecay}. Fix $y>0$, then for $x\in S$ satisfying $d(x,\Sigma)\in(2^{j-1}y,2^jy)$, we have $\Sigma\subset B(x,2^jy)\setminus B(x,2^{j-1}y)$ and thus by Lemma \ref{lem:MartinKernelDecay}
		\begin{align*}
			w(x_y)= \int_\Sigma k^\mathcal{O}_{p_0}(x_y,q)f(q)\mathrm{d}\omega^{p_0}_U(q)\leq C\frac{c2^{-\alpha j}}{\omega^{p_0}_U(\Sigma)}.
		\end{align*}
		
		In total, $w$ satisfies all the assumptions from the previous sections, in particular, we have information about $\mathcal{V}^w(\partial U)\cap\frac{1}{4}\mathcal{C}$. Also, since $w$ is bounded, we can again apply the aforementioned theorem in Section 2 of \cite{HuntWheeden68} by Hunt and Wheeden, yielding the representation
		\begin{align*}
			\forall p\in U:\quad w(p)=\int_{\partial U} k^U_{p_0}(p,q)g(q)\mathrm{d}\omega^{p_0}_U(q).
		\end{align*}
		where again $g\in L^1(\partial U)$ (and bounded). Moreover, by a very similar argument as in \cite[Section 2]{HuntWheeden68}, we obtain $f=g$ on $\Sigma$ and thus 
		\begin{align*}
			u(p)-w(p) = \int_{\partial U\setminus\Sigma}k^U_{p_0}(p,q)(f(q)-g(q))\mathrm{d}\omega^{p_0}_U(q)
		\end{align*}
		everywhere on $U$. We will from now on drop the sub- and superscripts $U$ and $p_0$ in our notation, i.e. $k:=k_{p_0}^U$ and $\omega:=\omega^{p_0}_U$. Differentiating and applying the bound on gradients of positive harmonic functions due to Harnack, we obtain 
		\begin{align}
			\label{eq:grad_bound}
			\|\nabla(u-w)(p)\|\leq\int_{\partial U\setminus\Sigma}\frac{k(p,q)}{d(p,\partial U)}|f(q)-g(q)|\mathrm{d}\omega(q)
		\end{align}
		for $p\in U$. Equality \eqref{eq:Bourgain_pts_in_cylinders} will be proven if we can show
		\begin{align}
			\label{eq:mutually_finite_int}
			\int_0^1\left\|\nabla u(p^*+yv)\right\|\mathrm{d}y<\infty \quad\Longleftrightarrow\quad\int_0^1\left\|\nabla w(p^*+yv)\right\|\mathrm{d}y<\infty
		\end{align}
		for any $p^*\in\Sigma/2$, where $v=h/2\vec{e_d}$ and $\Sigma/2=\phi(B^{d-1}(r/4))$. To prove \eqref{eq:mutually_finite_int} we plug into \eqref{eq:grad_bound} and estimate
		\begin{align*}
			\int_0^1\|\nabla(u-w)(p^*+yv)\|\mathrm{d}y\leq c_S\int_{\partial U\setminus\Sigma}\int_0^1\frac{k(p^*+yv,q)}{y}\mathrm{d}y|f(q)-g(q)|\mathrm{d}\omega(q).
		\end{align*}
		Denote $c=d(\Sigma/2,\partial U\setminus\Sigma)$ and rewrite
		\begin{align}
			\label{eq:DyadicSumOfKernel}
			\int_0^1\frac{k(p^*+yv,q)}{y}\mathrm{d}y=\lim_{N\rightarrow\infty}\sum_{k=0}^{N}\int_{2^{-k-1}}^{2^{-k}}\frac{k(p^*+yv,q)}{y}\mathrm{d}y
		\end{align}
		Let $p^*\in\Sigma/2$ and $q\in\partial U\setminus\Sigma$ and w.l.o.g $\|v\|=1$. Fix $k\in\mathbb{N}$ and $y\in[2^{-k-1},2^{-k}]$ . Now choose $j=j_k$ such that
		\begin{align*}
			\frac{c}{4}\leq2^{j-k-2}\leq 2^{j-1}y\leq\|p^*-q\|\leq 2^jy\leq2^{j-k}.
		\end{align*}
		The first inequality holds since otherwise $2^{j-k}<c\leq\|p^*-q\|$.	This implies that in the notation of Lemma \ref{lem:MartinKernelDecayHuntWheeden} 
		\begin{align*}
			q\in R_{j-1}^{k}\cup R_{j}^{k},
		\end{align*}	
		where $R^k_l:=\left\lbrace q\in\partial U: 2^{l-k-1}\leq\|p^*-q\|\leq 2^{l-k}\right\rbrace$. Thus, by Lemma \ref{lem:MartinKernelDecayHuntWheeden}
		\begin{align*}
			k(p^*+vy,q)\leq c'\max\left\lbrace\frac{c_{j_k}}{\omega(B(2^{j_k}y,p^*))},\frac{c_{j_k+1}}{\omega(B(2^{j_k+1}y,p^*))}\right\rbrace\leq\tilde{c}\frac{2^{-\alpha j_k}}{\omega(B(c/4,p^*))},
		\end{align*}
		where $\tilde{c}$ only depends on $U$ and $p_0$ and $\alpha>0$ only depends on $U$. Inserting into \eqref{eq:DyadicSumOfKernel} we obtain
		\begin{align*}
			\int_0^1\frac{k(p^*+yv,q)}{y}\mathrm{d}y\leq \frac{\tilde{c}}{\omega(B(c/4,p^*))}\lim_{N\rightarrow\infty}\sum_{k=0}^{N}2^{-\alpha j_k}<\infty
		\end{align*}
		and thus 
		\begin{align*}
			\int_0^1\|\nabla(u-w)(p^*+yv)\|\mathrm{d}y<\infty,
		\end{align*}
		implying \eqref{eq:mutually_finite_int}.
	
		\begin{figure}[h!]
			\includegraphics[width=\textwidth]{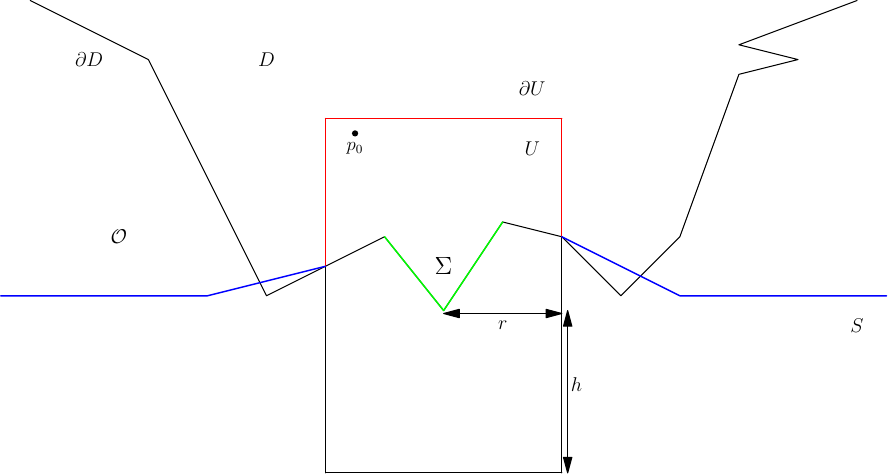}
			\caption{Identifying the Near Half Space $\mathcal{O}$ and the special Lipschitz domain $U$.}
		\end{figure}
		\begin{remark}\label{rem:HarmMeasComp}
			In the above situation we want to show
			\begin{align}\label{eq:HarmMeasComp}
				\omega^{p_0}_\mathcal{O}(\Delta_j)\geq\omega_U^{p_0}(\Delta_j\cap\partial U).
			\end{align}
			W.l.o.g., assume $\Delta_j\cap\partial U\neq\emptyset$. We will prove the desired inequality in a more general setting.
			
			Let $\mathcal{O}$ be any epigraph of a Lipschitz function, $U\subset\mathcal{O}$ be a bounded Lipschitz domain and $p\in U$. Let $M\subset\partial\mathcal{O}$ denote an open set and suppose $R:=M\cap\partial U\neq\emptyset$. We want to show $\omega^p_\mathcal{O}(R)\geq\omega_U^p(R)$ as this obviously implies \eqref{eq:HarmMeasComp}.
			
			For fixed $\epsilon>0$, find a continuous function $\varphi:\partial\mathcal{O}\longrightarrow\mathbb{R}$ satisfying $\varphi|_R\equiv1$ and $\varphi\equiv0$ on $\partial\mathcal{O}\setminus R_\epsilon$
			where $R_\epsilon:=\lbrace q\in\partial\mathcal{O}:d(q,R)<\epsilon\rbrace$. Consider the harmonic extension of $\varphi$ to $\mathcal{O}$;
			\begin{align*}
				\Phi(p):=\int_{\partial\mathcal
				{O}}\varphi(q)\mathrm{d}\omega_\mathcal{O}^p(q),\quad p\in\mathcal{O}.
			\end{align*}
			Then, obviously $\Phi(p)\leq\omega_\mathcal{O}^p(R_\epsilon)$. Now define the function $\tilde{\varphi}:\partial U\longrightarrow\mathbb{R}$ satisfying 
			\begin{align*}
				\tilde{\varphi}(q)=
				\begin{cases}
					\varphi(q),\quad q\in\partial\mathcal{O}\cap\partial U\\
					\Phi(q),\quad q\in\partial U\setminus\partial\mathcal{O}
				\end{cases}.
			\end{align*}
			$\tilde{\varphi}$ is bounded and continuous hence by uniqueness, its harmonic extension to $U$ retrieves $\Phi$ on $U$, yielding
			\begin{align*}
				\forall p\in U:\quad\Phi(p)=\int_{\partial U}\tilde{\varphi}(q)\mathrm{d}\omega_U^p(q).
			\end{align*}
			The last integral is bounded from below by $\omega^p_U(R)$ and since $\epsilon>0$ was arbitrary, we obtain $\omega_\mathcal{O}^p(R)\geq\omega^p_U(R)$.
			
		\end{remark}
		\begin{figure}[h!]
			\includegraphics[width=\textwidth]{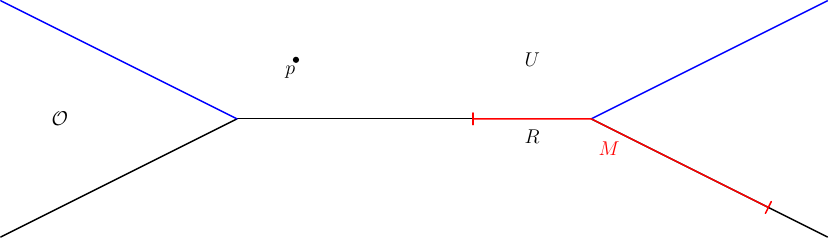}
			\caption{Schematic representation of sets in above proof.}
		\end{figure}
	
 		\newpage
 		\printbibliography[heading=bibintoc, title={References}]

\textsc{J. Fromherz, Institute of Analysis, Johannes Kepler University Linz, 				Altenberger Strasse 69, A-4040 Linz, Austria}

E-mail address: Jakob.Fromherz@jku.at

\bigskip

\textsc{P.F.X. M\"uller, Institute of Analysis, Johannes Kepler University Linz, 				Altenberger Strasse 69, A-4040 Linz, Austria}

E-mail address: Paul.Mueller@jku.at

\bigskip

\textsc{K. Riegler, Institute of Analysis, Johannes Kepler University Linz, Altenberger
Strasse 69, A-4040 Linz, Austria}

\end{document}